\documentclass[12pt]{amsart}
\usepackage{amsmath,bbm, color}
\usepackage{latexsym,mathrsfs,bm}
\usepackage{url}
\usepackage[english]{babel}
\usepackage{paralist}
\usepackage{amsthm,amsfonts,amssymb,amsmath}
\usepackage[latin1]{inputenc}

\newcommand{\tRe}{\textup{Re }}
\newcommand{\tIm}{\textup{Im }}
\newcommand{\lr}[1]{\left(#1\right)}
\newcommand{\ab}{\mathbf a}

\newcommand{\boldab}{{\bf a}, {\bf b}}
\newcommand{\al}{\boldsymbol \alpha}

\newcommand{\altt}{\boldsymbol \alpha_{2,3}}
\newcommand{\bb}{\boldsymbol \beta}
\newcommand{\bbtt}{\boldsymbol \beta_{2,3}}
\newcommand{\sumstar}{\sideset{}{^*}\sum}
\newcommand{\sumh}{\sideset{}{^h}\sum}
\newcommand{\sumd}{\sideset{}{^d}\sum}
\newcommand{\sumsharp}{\sideset{}{^\#}\sum}
\newcommand{\eps}{\varepsilon}

\newcommand{\es}[1]{\begin{equation}\begin{split}#1\end{split}\end{equation}}
\newcommand{\est}[1]{\begin{equation*}\begin{split}#1\end{split}\end{equation*}}

\newcommand{\Sc}{\mathcal S}

\newcommand{\K}{\mathcal{K}}
\newcommand{\q}{\mathfrak{q}}

\renewcommand{\mod}[1]{~\pr{\textnormal{mod}~#1}}
\newtheorem{thm}{Theorem}[section]

\newtheorem{prop}[thm]{Proposition}
\newtheorem{lem}[thm]{Lemma}
\newtheorem{lemma}[thm]{Lemma}
\newtheorem{conj}[thm]{Conjecture}
\theoremstyle{remark}

\newtheorem{remark}{Remark}
\newtheorem{rem*}{Remark}
\newcommand{\pr}[1]{\left( #1\right)}
\newcommand{\pg}[1]{\left\{ #1\right\}}

\newcommand{\e}[1]{\operatorname{e}\pr{ #1}}

\newcommand{\bfrac}[2]{\left(\frac{#1}{#2}\right)}
\newcommand{\tD}{\tilde{\Delta}}

\newcommand{\Res}{\operatorname*{Res}}
\def\sumstar{\operatornamewithlimits{\sum\nolimits^*}}

\newcommand{\sumtwo}{\operatorname*{\sum\sum}}
\newcommand{\sumthree}{\operatorname*{\sum\sum\sum}}
\newcommand{\sumfour}{\operatorname*{\sum\sum\sum\sum}}
\newcommand{\sumsix}{\operatorname*{\sum\sum\sum\sum\sum\sum}}

\newcommand{\comment}[1]{}

\let\originalleft\left
\let\originalright\right
\renewcommand{\left}{\mathopen{}\mathclose\bgroup\originalleft}
\renewcommand{\right}{\aftergroup\egroup\originalright}

\numberwithin{equation}{section}

\begin{document}
\title{The sixth moment of automorphic $L$-functions}

\date{\today}

\author[V. Chandee]{Vorrapan Chandee}
\address{Department of Mathematics \\ Burapha University \\ 169 Long-Hard Bangsaen rd, Saensuk, Mueng, Chonburi, Thailand,  20131}
\email{vorrapan@buu.ac.th }

\author[X. Li]{Xiannan Li}
\address{ \parbox{1\linewidth}{%
		University of Oxford, Andrew Wiles Building, Radcliffe Observatory Quarter, Woodstock Road, Oxford, UK, OX2 6GG \\ %
		Current address: Department of Mathematics, Kansas State University, 138 Cardwell Hall, Manhattan, Kansas, USA, 66506}
	}
\email{xiannan@math.ksu.edu }
%\address{Current: Kansas State University }

\allowdisplaybreaks
\numberwithin{equation}{section}
\selectlanguage{english}
\begin{abstract}
In this paper, we consider the $L$-functions $L(s, f)$ where $f$ is an eigenform for the congruence subgroup $\Gamma_1(q)$.  We prove an asymptotic formula for the sixth moment of this family of automorphic $L$-functions.  
\end{abstract}

\maketitle  

\section{Introduction}   

Moments of $L$-functions are of great interest to analytic number theorists.  For instance, for $\zeta(s)$ denoting the Riemann zeta function and
$$I_k(T) := \int_0^T |\zeta(\tfrac{1}{2} + it)|^{2k}dt,$$  
asymptotic formulae were proven for $k = 1$ by Hardy and Littlewood and for $k=2$ by Ingham (see \cite{Ti} VII).  This work is closely related to zero density results and the distribution of primes in short intervals.  More recently, moments of other families of $L$-functions for studied for their numerous applications, including non-vanishing and subconvexity results.  In many applications, it is important to develop technology which can understand such moments for larger $k$.

The behavior of moments for larger $k$ remain mysterious.  However, recently there has been great progress in our understanding.  First, good heuristics and conjectures on the behavior of $I_k(T)$ appeared in the literature.  To be precise, a folklore conjecture states that  
$$I_k(T) \sim c_k T(\log T)^{k^2}$$ for constants $c_k$ depending on $k$ but the values of $c_k$ were unknown for general $k$ until the work of Keating and Snaith \cite{KS} which related these moments to circular unitary ensembles and provided precise conjectures for $c_k$.  The choice of group is consistent with the Katz-Sarnak philosophy \cite{KaSa}, which indicates that the symmetry group associated to this family should be unitary.  Based on heuristics for shifted divisor sums, Conrey and Ghosh derived a conjecture in the case $k=3$ \cite{CGh} and Conrey and Gonek derived a conjecture in the case $k=4$ \cite{CGo}. 
In particular, the conjecture for the sixth moment is 
$$ I_3(T) \sim 42 a_3 \frac{T(\log T)^9}{9!}$$
for some arithmetic factor $a_3.$ 
% $\prod_{p} \left( 1 - \frac 1p\right)^4 \left( 1 + \frac 4p + \frac 1{p^2}\right)$
Further conjectures including lower order terms, and for other symmetry groups are available from the work of Conrey, Farmer, Keating, Rubinstein and Snaith \cite{CFKRS} as well as from the work of Diaconu, Goldfeld and Hoffstein \cite{DGH}. 

In support of these conjectures, lower bounds of the the right order of magnitude are available due to Rudnick and Soundararajan \cite{Lower}, while good upper bounds of the right order of magnitude are available conditionally on RH, due to Soundararajan \cite{Sound} and later improved by Harper \cite{Har}.

Despite this, verifications of the moment conjectures for high moments remain elusive.  Typically, even going slightly beyong the fourth moment to obtain a twisted fourth moment is quite difficult, and there are few families for which this is known.  

Quite recently, Conrey, Iwaniec and Soundararajan \cite{CIS} derived an asymptotic formula for the sixth moment of Dirichlet $L$-functions with a power saving error term.  Instead of fixing the modulus $q$ and only averaging over primitive characters $\chi \bmod q$, they also average over the modulus $q \leq Q$, which gives them a larger family of size $Q^2$.  Further, they include a short average on the critical line. In particular, they showed that 
\est{\sum_{q \leq Q} \sumstar_{\chi \mod q} \int_{-\infty}^{\infty} \left| L\left( \tfrac 12 + it, \chi \right)\right|^6\left| \Gamma\left( \tfrac 14 + \tfrac{it}{2} \right)\right|^6 \> dt \sim 42 b_3  \frac{Q^2(\log Q)^9}{9!} \int_{-\infty}^{\infty} \left| \Gamma\left( \tfrac 14 + \tfrac{it}{2} \right)\right|^6 \> dt }
for some constant $b_3.$ This is consistent with the analogous conjecture for the Riemann zeta function above. 

  The authors of this paper subsequently derived an asymptotic formula for the eight moment of this family of $L$-functions, conditionally on GRH \cite{CL}, which is 
  \est{\sum_{q \leq Q} \sumstar_{\chi \mod q} \int_{-\infty}^{\infty} \left| L\left( \tfrac 12 + it, \chi \right)\right|^8\left| \Gamma\left( \tfrac 14 + \tfrac{it}{2} \right)\right|^8 \> dt \sim 24024 b_4  \frac{Q^2(\log Q)^{16}}{16!} \int_{-\infty}^{\infty} \left| \Gamma\left( \tfrac 14 + \tfrac{it}{2} \right)\right|^8 \> dt }
  for some constant $b_4.$

In this paper, we study a family of $L$-functions attached to automorphic forms on $GL(2)$.  To be more precise, let $S_k(\Gamma_0(q), \chi)$ be the space of cusp forms of weight $k \ge 2$ for the group $\Gamma_0(q)$ and the nebentypus character $\chi \mod q$,  where $$\Gamma_0(q) = \left\{ \left( \left.{\begin{array}{cc}
   a & b \\
   c & d \\
  \end{array} } \right)  \ \right|  \  ad- bc = 1 , \ \  c \equiv 0 \mod q \right\}.$$ 
Also, let $S_k(\Gamma_1(q))$ be the space of holomorphic cusp forms for the group 
$$\Gamma_1(q) = \left\{ \left( \left.{\begin{array}{cc}
   a & b \\
   c & d \\
  \end{array} } \right)  \ \right|  \  ad- bc = 1 , \ \  c \equiv 0 \mod q, \ \ \ a \equiv d \equiv 1 \mod q \right\}.$$
Note that $S_k(\Gamma_1(q))$ is a Hilbert space with the Petersson's inner product
$$ <f, g> = \int_{\Gamma_1(q) \backslash \mathbb H} f(z)\bar{g}(z) y^{k-2} \> dx \> dy,$$
and
$$ S_k(\Gamma_1(q)) = \bigoplus_{\chi \mod q} S_k(\Gamma_0(q), \chi).$$  

Let $\mathcal H_\chi \subset  S_k(\Gamma_0(q), \chi)$ be an orthogonal basis of $ S_k(\Gamma_0(q), \chi)$ consisting of Hecke cusp forms, normalized so that the first Fourier coefficient is $1$. For each $f \in \mathcal H_\chi $, we let $L(f, s)$ be the $L$-function associated to $f$, defined for $\tRe(s) > 1$ as 
\es{\label{def:Lfnc} 
L(f,s) = \sum_{n \geq 1} \frac {\lambda_f(n)}{n^{s}} = \prod_p \pr{1 - \frac{\lambda_f(p)}{p^s} + \frac{\chi(p)}{p^{2s}}}^{-1},}
where $\{\lambda_f(n)\}$ are the Hecke eigenvalues of $f$.  With our normalization, $\lambda_f(1) = 1 $.  In general, the Hecke eigenvalues satisfy the Hecke relation
\es{\label{eqn:Heckerelation}
\lambda_f(m) \lambda_f(n) = \sum_{d|(m,n)} \chi(d) \lambda_f\pr{\frac{mn}{d^2}},} 
for all $m,n \geq 1$. 
We define the completed $L$-function as
\es{\label{def:completedLfnc}\Lambda\pr{f, \tfrac 12 + s} = \pr{\frac q{4\pi^2}}^{\frac s2} \Gamma \pr{s + \frac k2} L\pr{f, \tfrac 12 + s},}
which satisfies the functional equation
\est{\Lambda\pr{f, \tfrac 12 + s} = i^k \overline{\eta}_f\Lambda\pr{\bar{f}, \tfrac 12 - s},}
where $|\eta_f| = 1$ when $f$ is a newform.  

Suppose for each $f \in \mathcal H_\chi$, we have an associated number $\alpha_f$. Then we define the harmonic average of $\alpha_f$ over $\mathcal H_\chi$ to be 
\est{\sumh_{f \in \mathcal H_\chi} \alpha_f = \frac{\Gamma(k-1)}{(4\pi)^{k-1}}\sum_{f \in \mathcal H_\chi} \frac{\alpha_f}{\|f\|^2}.} 

We note that when the first coefficient $\lambda_f(1) = 1$, $\|f\|^2$ is essentially the value of a certain $L$-function at $1$, and so on average, $\|f\|^2$ is constant.  As in other works, it is possible to remove the weighting by $\|f\|^2$ through what is now a standard argument.

We shall be interested in moments of the form 

\est{\frac{2}{\phi(q)}\sum_{\substack{\chi \mod q \\ \chi(-1) = (-1)^k}} \sumh_{f \in \mathcal H_\chi} |L(f, 1/2)|^{2k}.}

 We note that the size of the family is around size $q^2$. For prime level, $\eta_f$ can be expressed in terms of Gauss sums, and in particular we expect $\eta_f$ to equidistribute on the circle as $f$ varies over an orthogonal basis of $S_k(\Gamma_1(q))$.  Thus, we expect our family of $L$-functions to be unitary.

In this paper, we prove an asymptotic formula for the sixth moment - this will be the first time that the sixth moment of a family of $L$-functions over $GL(2)$ has been understood.  Following \cite{CFKRS}, we have the following conjecture for the sixth moment of our family. We refer the reader to Appendix \ref{sec:Eulerproduct} for a brief derivation of the arithmetic factor in the conjecture.
\begin{conj} \label{conj:CFKRSnoshift}Let $q$ be a prime number.
As $q \rightarrow \infty$, we have
\est{\frac{2}{\phi(q)}\sum_{\substack{\chi \mod q \\ \chi(-1) = (-1)^k}} \sumh_{f \in \mathcal H_\chi} |L(f, 1/2)|^6 \sim 42 \mathscr C_3  \left(1- \frac 1q\right)^{4} \left(1 + \frac 4q + \frac {1}{q^2} \right) C_q^{-1} \frac{(\log q)^9}{9!},}
where 
\es{\label{def:c3} 
	\mathscr C_3 := \prod_{p} C_p,  \ \ \  \textrm{and,} \ \ \ C_p:= \left( 1 + \frac 4p + \frac 7{p^2} - \frac 2{p^3} + \frac 9{p^4} - \frac{16}{p^5} + \frac{1}{p^6} - \frac{4}{p^7}\right)\left( 1 + \frac 1{p}\right)^{-4}.}
\end{conj}

Iwaniec and Xiaoqing Li proved a large sieve result for this family in \cite{IL}, and Djankovic used their result to prove \cite{Dj} that for an odd integer $k 
\geq 3$ and prime $q$ that
\begin{equation*}
\frac{2}{\phi(q)}\sum_{\substack{\chi \mod q \\ \chi(-1) = -1}} \sumh_{f \in \mathcal H_\chi} |L(f, 1/2)|^6   \ll q^{\eps}. 
\end{equation*}
as $q \rightarrow \infty.$  In this paper, we shall prove the following

\begin{thm} \label{cor:main} Let $q$ be a prime and $k\geq 5$ be odd.  Then, as $q \rightarrow \infty$, we have

\est{\frac{2}{\phi(q)}\sum_{\substack{\chi \mod q \\ \chi(-1) = -1}} & \sumh_{f \in \mathcal H_\chi} \int_{-\infty}^{\infty} \left|\Lambda\left(f, \tfrac 12 + it\right)\right|^6 \> dt \\
	&\sim 42 \mathscr C_3 \left(1- \frac 1q\right)^{4} \left(1 + \frac 4q + \frac {1}{q^2} \right) C_q^{-1}  \frac{(\log q)^9}{9!} \int_{-\infty}^{\infty} \left|\Gamma\left( \tfrac k2 + it \right) \right|^6 \> dt,}
where $\mathscr C_3$ and $C_p$ are defined in (\ref{def:c3}).

\end{thm}

In fact, we are able to prove this with an error term of $q^{-1/4}$, as opposed to the $q^{-1/10}$ error term in the work of Conrey, Iwaniec and Soundararajan\cite{CIS}.  The reason behind this superior error term is explained in the outline in Section \ref{subsec:outline}.  In future work, we hope to extend our attention to the eighth momment.  

The assumption that $k$ is odd implies that all $f\in \mathcal H_\chi$ are newforms.  This is for convenience only and is not difficult to remove.  Indeed, when $k$ is even, all $f \in \mathcal H_\chi$ are newforms except possibly when $\chi$ is the principal character and $f$ is induced by a cusp form of full level.  We avoid this case for the sake of brevity.  Similarly, the assumption that $k\geq 5$ simplifies parts of the calculation; it is possible to prove Theorem \ref{cor:main} for smaller $k$.

Since the $\Gamma$ function decays rapidly on vertical lines, the average over $t$ is fairly short.  It is included for the same reason as in the works \cite{CIS} and \cite{CL} in that it allows us to avoid certain unbalanced sums in the computation of the moment.  Although this appears to be a small technical change in the main statement, evaluating such moments without the short integration over $t$ is a significant challenge.  Our Theorem will follow from the more general Theorem \ref{thm:mainmoment} for shifted moments in Section \ref{sec:shift}.

\subsection{Outline of paper} \label{subsec:outline}
To help orient the reader, we provide a sketch of the proof, and introduce the various sections of the paper.  After applying the approximate functional equation developed in Section \ref{sec:prelem}, the main object to be understood is roughly of the form
$$\frac{2}{\phi(q)}\sum_{\substack{\chi \mod q \\ \chi(-1) = (-1)^k}} {\sum_{f \in \mathcal H_\chi}}^h  \sum_{m, n \asymp q^{3/2}} \frac{\sigma_3(m)\sigma_3(n)\lambda_f(n)\overline{\lambda_f}(m)}{\sqrt{mn}}.
$$In fact, since the coefficients $\lambda_f(n)$ are not completely multiplicative, the expression is significantly more complicated for the purpose of extracting main terms.  

Applying Peterson's formula for the average over $f\in \mathcal H_\chi$ leads to diagonal terms $m=n$ which are evaluated fairly easily in Section \ref{sec:diag} as well as off-diagonal terms which involve sums of the form
$$ \sum_{m, n \asymp q^{3/2}} \frac{\sigma_3(m)\sigma_3(n)}{\sqrt{mn}} \frac{2}{\phi(q)}\sum_{\substack{\chi \mod q \\ \chi(-1) = (-1)^k}} \sum_c S_{\chi}(m, n; cq) J_{k-1}\bfrac{4\pi \sqrt{mn}}{cq},
$$ 
where $S_{\chi}(m, n; cq)$ is the Kloosterman sum defined in (\ref{def:Kloosterman}), and $J_{k-1}(x)$ is the J-Bessel function of order $k-1$.

Let us focus on the transition region for the Bessel function where $c \asymp q^{1/2}$, so that the conductor is a priori of size $qc \asymp q^{3/2}$.  It is here that the addition average over $\chi \bmod q$ comes into play.  To be more precise, to understand the exponential sum $$\frac{2}{\phi(q)}\sum_{\substack{\chi \mod q \\ \chi(-1) = (-1)^k}} S_{\chi}(m, n; cq),$$ it suffices to understand
$$\sumstar_{\substack{a \bmod cq\\a \equiv 1 \bmod q}} e\bfrac{am + \bar an}{cq},
$$which, assuming that $(c, q) = 1$, is
$$e \bfrac{m+n}{cq}\sumstar_{a \bmod c} e\bfrac{\bar q(a-1)m + \bar q (\bar a - 1)n}{c},
$$using Chinese remainder theorem and reciprocity.  The factor $e \bfrac{m+n}{cq}$ has small derivatives and may be treated as a smooth function, while the conductor of the rest of the exponential sum has decreased to $c \asymp q^{1/2}$.  The details of these calculations are in Section \ref{sec:offdiagsetup}.  

This phenomenon of the drop in conductor appears in other examples.  In the case of the sixth moment of Dirichlet $L$-functions in \cite{CIS}, it occurs when replacing $q$ with the complementary divisor $\frac{m-n}{q} \asymp q^{1/2}$.  It is quite interesting that the same drop in conductor occurs by seemingly very different mechanisms.  However, note that when the complementary divisor is small, the ordered pair $(m, n)$ is forced to be in a narrow region.  That this does not occur in our case is one of reasons behind the superior error term in our result; the assumption that $q$ is prime also plays a role.

After the conductor drop, we apply Voronoi summation to the sum over $m$ and $n$ in \S \ref{sec:appvoronoi}.  We need a version of Voronoi summation including shifts.  The proof of this is essentially the same as the proof of the standard Voronoi summation formula for $\sigma_3(n)$ by Ivic \cite{Ivic}.  We state the result required in Appendix \ref{sec:voronoi}.  

After applying Voronoi, it is easy to guess which terms should contribute to the main terms and which terms should be error terms.  The main terms are described in Proposition \ref{prop:mainTM} and the error terms are bounded in Proposition \ref{prop:Terror}.  Essentially, we expect the main terms to be a sum of products of $9$ factors of $\zeta$, the same as the diagonal contribution but with permutations in the shifts, as in Theorem \ref{thm:mainmoment}.  This is by no means immediately visible from the expression in Proposition \ref{prop:mainTM}.  Indeed, it takes some effort to see that we get the right number of $\zeta$ factors.  Along the way, we use, among other things, a calculation of Iwaniec and Xiaoqing Li in \cite{IL}.  This is done in Section \ref{sec:provepropTM}.  In order to finish the verifications, we need to check that the local factors of two expressions agree.  The details here are standard but intricate, and are provided in Appendix \ref{sec:Eulerverif}.  

Finally, the error terms from Voronoi summation are bounded in Section \ref{sec:properror}.  Here, one needs to show that the dual sums from Voronoi summation are essentially quite short, which is related to the reduction in conductor from $cq$ to $c$ earlier.

\section{Notation and the shifted sixth moment} \label{sec:shift}
We begin with some notation. Let $\al := (\alpha_1, \alpha_2, \alpha_3)$ and  $\bb := (\beta_1, \beta_2, \beta_3)$.  For a complex number $s$, we shall write $\al + s := (\alpha_1 + s, \alpha_2 + s, \alpha_3 + s).$  We define
\es{\label{def:deltaalphabeta}
\delta(\al, \bb) := \frac 12 \sum_{j = 1}^3 (\alpha_j - \beta_j),}
\es{\label{def:GprodGamma}
G(s; \al, \bb) := \prod_{j = 1}^3 \Gamma \pr{s + \tfrac {k-1}2 + \alpha_j}\Gamma \pr{s + \tfrac {k-1}2 - \beta_j}, }and
\es{\label{def:completedSprodLfnc}
\Lambda(f, s; \al, \bb) : = \prod_{j = 1}^3 \Lambda\pr{f, s + \alpha_j}\Lambda \pr{\bar{f}, s - \beta_j}.}
Note that we have
\es{\label{def:prodCompletedLfnc}
\Lambda(f; \al, \bb) = \Lambda\pr{f, \tfrac 12; \al, \bb} = \pr{\frac q{4\pi^2}}^{\delta(\al, \bb)} G\pr{\tfrac 12; \al, \bb} \prod_{j = 1}^3 L\pr{f, \tfrac 12 + \alpha_j}L\pr{\bar{f}, \tfrac 12 - \beta_j}.}
We define the shifted $k$-divisor function by
\es{\label{def:sigma_k} \sigma_k(n; \alpha_1,..., \alpha_k) = \sum_{n_1n_2...n_k = n} n_1^{-\alpha_1}n_2^{-\alpha_2}...n_k^{-\alpha_k}. }
Let 
\es  { \label{eqn:defB}\mathscr B(a, b; \al) := \frac{\mu(a)\sigma_3(b; \alpha_1 + \alpha_2, \alpha_2 + \alpha_3, \alpha_3 + \alpha_1)}{a^{\alpha_1 + \alpha_2 + \alpha_3}}.}
Next we need the following lemmas, which help us generate the conjecture of the sixth moment, namely
\es{ \label{momentS1q}\frac{2}{\phi(q)}\sum_{\substack{\chi \mod q \\ \chi(-1) = (-1)^k}} \sumh_{f \in \mathcal H_\chi} \Lambda\pr{f; \al , \bb }. }

\begin{lemma} \label{lem:multofsigma_k}
	 We have
	\est{\sigma_2(n_1n_2; \alpha_1, \alpha_2) = \sum_{d| (n_1, n_2)} \mu(d) d^{-\alpha_1 - \alpha_2} \sigma_2 \pr{\frac{n_1}{d}; \alpha_1, \alpha_2}\sigma_2\pr{\frac{n_2}{d}; \alpha_1, \alpha_2}.}
\end{lemma}

\begin{proof}
	Since both sides are multiplicative functions, it is enough to prove the Lemma when $n_1n_2$ is a prime power. We set $n_1 = p^a$ and $n_2 = p^b,$ where $1 \leq a \leq b.$  Then
	\es{ \label{eqn:idprimesigma}\sigma_2(p^{a}p^b; \alpha_1, \alpha_2) &= \sum_{0 \leq k \leq a} \mu(p^k) p^{-k(\alpha_1 +\alpha_2)} \sigma_2 \pr{p^{a-k}; \alpha_1, \alpha_2}\sigma_2\pr{p^{b-k}; \alpha_1, \alpha_2} \\
		&=   \sigma_2 \pr{p^{a}; \alpha_1, \alpha_2}\sigma_2\pr{p^{b}; \alpha_1, \alpha_2} - p^{-(\alpha_1 +\alpha_2)} \sigma_2 \pr{p^{a-1}; \alpha_1, \alpha_2}\sigma_2\pr{p^{b-1}; \alpha_1, \alpha_2}.}
	On the other hand, 
	\est{\sigma_2 (p^{n}; \alpha_1, \alpha_2) = \sum_{\ell = 0}^n p^{-\ell\alpha_1} p^{-(n-\ell)\alpha_2} = p^{-n\alpha_2} \frac{1 - \frac{1}{p^{(n+1)(\alpha_1 - \alpha_2)}}}{1 - \frac{1}{p^{\alpha_1 - \alpha_2}}}, }
	and the lemma follows by substituting the above formula into (\ref{eqn:idprimesigma}).
	
\end{proof}
We write the product of $L$-functions in term of Dirichlet series in the following lemma

\begin{lem} \label{eqn:prod3Lfnc} Let $L(f, w)$ be an $L$-function in $\mathcal H_\chi.$  For $\tRe(s + \alpha_i) > 1$, we have
	\est{  L\pr{f, s + \alpha_1}&L\pr{f, s +  \alpha_2}L\pr{f,  s+ \alpha_3} = \sumtwo_{a, b \geq 1} \frac{\chi(ab)\mathscr B(a, b; \al)}{(ab)^{ 2s}}   \sum_{n \geq 1} \frac{\lambda_f(an) \sigma_3(n;\al)}{(an)^{ s}}.}
\end{lem}

\begin{proof}
	
	 From the Hecke relation (\ref{eqn:Heckerelation}) and Lemma \ref{lem:multofsigma_k}, we have
	\es{ \label{prod3L}
		&L\pr{f, s + \alpha_1}L\pr{f,  s + \alpha_2}L\pr{f,  s +\alpha_3} \\
		%&= \sum_{n_1} \frac{\lambda_f(n_1)}{n_1^{1/2 + s + \alpha_1}} \sumtwo_{n_2, n_3} \frac{1}{(n_2n_3)^{1/2 + s}n_2^{\alpha_2}n_3^{\alpha_3}} \sum_{d | (n_2, n_3)} \lambda_f \left(\frac{n_2n_3}{d^2}\right) \chi(d) \\
		&= \sum_{n_1 \geq 1} \frac{\lambda_f(n_1)}{n_1^{ s + \alpha_1}}  \sum_{d \geq 1} \frac{\chi(d)}{d^{2s + \alpha_2 + \alpha_3}}  \sum_{j\geq 1} \frac{\lambda_f \left(j\right) \sigma_2(j; \alpha_2, \alpha_3)}{j^{ s}} \\
		&= \sum_{d \geq 1} \frac{\chi(d)}{d^{ 2s + \alpha_2 + \alpha_3}}  \sumtwo_{n_1, j \geq 1} \frac{\sigma_2(j; \alpha_2, \alpha_3)}{n_1^{ s + \alpha_1} j^{ s}} \sum_{e|(n_1,j)} \chi(e) \lambda_f \pr{\frac{jn_1}{e^2}} \\
		%&=  \sum_d \frac{\chi(d)}{d^{1 + 2s + \alpha_2 + \alpha_3}} \sum_{n_1} \frac{1}{n_1^{1/2 + s + \alpha_1}} \sum_{a} \frac{\mu(a)}{a^{3/2 + 3s + \alpha_1 + \alpha_2 + \alpha_3}} \sum_{j} \frac{\sigma_2(m; \al) \lambda_f(n_1ja) }{j^{1/2 + s}} \sum_e \frac{\chi(ej) \sigma_2(e; \al)}{e^{1 + 2s + \alpha_1}}  \\
		&= \sum_{a \geq 1} \frac{\mu(a)\chi(a)}{a^{2s + \alpha_1 + \alpha_2 + \alpha_3}} \sumtwo_{d, e \geq 1} \frac{\chi(de) \sigma_2(e; \alpha_2, \alpha_3)}{(de)^{2s}d^{\alpha_2+ \alpha_3}e^{\alpha_1}} \sumtwo_{j, n_1 \geq 1} \frac{\lambda_f(n_1ja) \sigma_2(j; \alpha_2, \alpha_3)}{(ajn_1)^{ s}n_1^{\alpha_1}} \\
		&= \sum_{a \geq 1} \frac{\mu(a)\chi(a)}{a^{2s + \alpha_1 + \alpha_2 + \alpha_3}} \sum_{b \geq 1} \frac{\chi(b) \sigma_3(b; \alpha_1+ \alpha_2,\alpha_2+ \alpha_3, \alpha_3 + \alpha_1)}{b^{2s}} \sum_{n \geq 1} \frac{\lambda_f(an) \sigma_3(n;\alpha_1, \alpha_2, \alpha_3)}{(an)^{  s}} .}
	This completes the lemma.
\end{proof}

\begin{lem} \label{lem:orthogonality} The orthogality relation for Dirichlet characters is 
\es{ \label{eqn:orthodirichlet} \frac{2}{\phi(q)}\sum_{\substack{\chi \mod q \\ \chi(-1) =(-1)^k}} \chi(m) \overline{\chi} (n)  = \left\{\begin{array}{ll} 1 & {\rm if} \ m \equiv n \mod q,  \ \ (mn, q) = 1\\
		(-1)^k & {\rm if} \ m \equiv - n \mod q,  \ \ (mn, q) = 1 \\
		0 & {\rm otherwise}. \end{array} \right.}
Petersson's formula gives
\es{ \label{eqn:Petersson}\sumh_{f \in \mathcal H_\chi } \overline{\lambda}_f(m) \lambda_f(n) = \delta_{m = n} + \sigma_\chi(m, n),}
where
\est{\sigma_\chi(m, n) &= 2\pi i^{-k} \sum_{c \equiv 0 \mod q} c^{-1}S_\chi(m, n; c) J_{k - 1}\pr{\frac{4\pi}{c} \sqrt{mn}} \\
	&=  2\pi i^{-k} \sum_{c = 1}^{\infty} (cq)^{-1}S_\chi(m, n; cq) J_{k - 1}\pr{\frac{4\pi}{cq} \sqrt{mn}}, }
and $S_\chi$ is the Kloosterman sum defined by 
\es{ \label{def:Kloosterman}S_{\chi} (m, n; cq) = \sum_{a\bar{a} \equiv 1 \mod {cq}} \chi(a) e \pr{\frac{am + \bar{a}n}{cq}}.}

\end{lem}

From Lemma \ref{eqn:prod3Lfnc}, we have that
\est{\prod_{i = 1}^3  L(f, s+ \alpha_i)L(f, s - \beta_i) &= \sumfour_{a_1, b_1, a_2, b_2 \geq 1} \frac{\chi(a_1b_1)\overline{\chi}(a_2b_2)\mathscr B(a_1, b_1; \al)\mathscr B(a_2, b_2; -\bb)}{(a_1b_1a_2b_2)^{ 2s}} \\
	&  \ \ \ \ \ \ \ \ \ \ \times  \sumtwo_{n, m \geq 1} \frac{\lambda_f(a_1n)\overline{\lambda_f}(a_2m) \sigma_3(n;\al)\sigma_3(m;-\bb)}{(a_1na_2m)^{ s}}.}
By the orthogonality relation of Dirichlet characters and Petersson's formula in Lemma \ref{lem:orthogonality}, a naive guess might be that the main contribution comes from the diagonal terms $a_1b_1 = a_2b_2$ and $a_1n = a_2m$, where $(a_ib_i, q) = 1,$ which is 
\es{\label{def:Cs} \mathcal C(s, \al, \bb) := \sumsix_{\substack{a_1, b_1, a_2, b_2, m, n \geq 1\\ a_1n = a_2m \\ a_1b_1 = a_2b_2 \\ (a_i, q) = (b_j, q) = 1 }}  \frac{ \mathscr B(a_1, b_1; \al)}{(a_1b_1)^{2s}} \frac{ \mathscr B(a_2, b_2; -\bb)}{(a_2b_2)^{2s}} \frac{\sigma_3(n;\al)\sigma_3(m; -\bb)}{(a_1n)^{ s} (a_2m)^{ s}}}
for $\tRe(s)$ large enough. This can be written as the Euler product 
$$ \mathcal C(s, \al, \bb) = \prod_p \mathcal C_p(s, \al, \bb), $$
where for $p \neq q$, 

\es{\label{def:Cps} \mathcal C_p(s, \al, \bb) &:= \sumsix_{\substack{r_1, t_1, r_2, t_2, u_1, u_2 \geq 0 \\ r_1 + u_1 = r_2 + u_2 \\ r_1 + t_1 = r_2 + t_2  }}  \frac{ \mathscr B(p^{r_1}, p^{t_1}; \al)}{p^{2s(r_1 + t_1)}} \frac{ \mathscr B(p^{r_2}, p^{t_2}; -\bb)}{p^{2s(r_2 + t_2)}} \frac{\sigma_3(p^{u_1};\al)\sigma_3(p^{u_2}; -\bb)}{p^{ s(r_1 + u_1)} p^{s(r_2 + u_2)}}, }
and for $p = q,$ 
\es{\label{def:Cqs} \mathcal C_q(s, \al, \bb) &:= \sum_{\substack{u \geq 0   }}   \frac{\sigma_3(q^{u};\al)\sigma_3(q^{u}; -\bb)}{q^{ 2us} }. }
 
Next, for $\zeta_p(w) := \pr{1 - \frac{1}{p^w}}^{-1}$, let
\es{ \label{def:mathcalZp}\mathcal Z_p(s;\al, \bb) := \prod_{i = 1}^3 \prod_{j = 1}^3 \zeta_p (2s + \alpha_i - \beta_j), \hskip 0.7in \mathcal Z (s; \al, \bb) := \prod_{i = 1}^3 \prod_{j = 1}^3 \zeta(2s + \alpha_i - \beta_j).}

and 
\es{\label{def:mathcalAs} \mathcal A (s; \al, \bb) := \mathcal C(s; \al , \bb) \mathcal Z(s;\al, \bb)^{-1} = \prod_{p} \mathcal C_p(s; \al , \bb) \mathcal Z_p(s;\al, \bb)^{-1}.}

We define
\es{\label{def:Mqalbb} \mathcal M(q; \al, \bb) := \pr{\frac{q}{4\pi^2}}^{\delta(\al, \bb)} G\pr{\tfrac 12; \al, \bb }\mathcal A \mathcal Z \pr{\tfrac 12; \al, \bb}. }
When $\tRe(\alpha_i), \tRe(\beta_i) \ll \frac{1}{\log q}$, the term $\mathcal A\pr{\tfrac 12, \al, \bb}$ is absolutely convergent.  Now, let $S_{j}$ be the permutation group of $j$ variables.  Based on the analysis of the diagonal contribution, we expect $\mathcal M(q; \al, \bb) $ to be a part of the average in (\ref{momentS1q}), and we also notice that the expression $\mathcal M(q; \al, \bb)$ is fixed by the action of $S_3 \times S_3$.  Since we expect our final answer to be symmetric under the full group $S_6$, we sum over the cosets $S_6/(S_3 \times S_3)$.  In fact, the method of Conrey, Farmer, Keating, Rubinstein and Snaith \cite{CFKRS} gives the following conjecture for the average of $\Lambda(f;\al, \bb)$.

\begin{conj} \label{conj:mainterm}Assume that $\al, \bb$ satisfy $\tRe(\alpha_i), \tRe(\beta_i)  \ll \frac 1{\log q}$ and $\tIm (\alpha_i), \tIm (\beta_i) \ll q^{1-\eps}.$ We have
\est{\frac{2}{\phi(q)}\sum_{\substack{\chi \mod q \\ \chi(-1) = (-1)^k}} \sumh_{f \in \mathcal H_\chi} \Lambda\pr{f; \al , \bb }  =  \sum_{\pi \in S_{6}/(S_3\times S_3)} \mathcal M(q; \pi(\al, \bb)) \pr{1 + O(q^{-\frac 12 + \eps})}, }
where  we define $\pi(\al, \bb) = \pi (\alpha_1, ...,\alpha_k, \beta_1,...,\beta_k)$ for $\pi \in S_{2k}$, where $\pi$ acts on the $2k$ tuple $(\alpha_1, ...,\alpha_k, \beta_1,...,\beta_k)$ as usual. 

\end{conj}

We will also write $\pi(\al, \bb) = (\pi(\al), \pi(\bb))$ by an abuse of notation, where $\pi(\al, \bb)$ is as above.  Our main goal is to find an asymptotic formula for
 \es{ \label{eqn:mainConsider} \mathcal M_6(q) :=  \frac{2}{\phi(q)}\sum_{\substack{\chi \mod q \\ \chi(-1) = (-1)^k}} {\sum_{f \in \mathcal H_\chi}}^h \int_{-\infty}^{\infty} \Lambda\pr{f; \al + it, \bb + it} \> dt,}
 and we will prove the following result. 
 \begin{thm} \label{thm:mainmoment}  Let $q$ be prime and $k \geq 5$ be odd.  For $\alpha_i, \beta_j \ll \frac 1{\log q}$, we have that 
 \est{\mathcal M_6(q) =   \int_{-\infty}^{\infty} \sum_{\pi \in S_{6}/(S_3\times S_3)}  \mathcal M(q, \pi(\al) + it, \pi(\bb) + it) \> dt + O\pr{q^{-\frac 14 + \eps}}. }
 \end{thm}

We note that as the shifts go to $0$, the main term of this moment is of the size $(\log q)^9$, and we derive Theorem (\ref{cor:main}).  We refer the reader to \cite{CFKRS} for the details of this type of calculation. 
%The error term here is superior to the corresponding error term in Theorem 1 in the work of Conrey, Iwaniec and Soundararajan \cite{CIS}, as explained in Section \ref{subsec:outline}. 

\section{Approximate functional equation}\label{sec:prelem}
%\subsection{Approximate functional equation}

In this section, we will prove an approximate functional equation for the product of $L$-functions. Let 
\est{H(s; \al, \bb) = \prod_{j = 1}^3 \prod_{\ell = 1}^3 \pr{s^2 - \pr{\frac{\alpha_j - \beta_\ell}{2}}^2}^3,}
and define for any $\xi> 0,$
\est{W(\xi; \al, \bb) = \frac{1}{2\pi i} \int_{(1)} G\pr{\tfrac 12 + s; \al, \bb} H\pr{s; \al, \bb} \xi^{-s} \ \frac{ds}{s}.}
Moreover, let $\Lambda_0(f, \al, \bb) $ be
\es{\label{def:Lambda_0} 
& \pr{\frac{q}{4\pi^2}}^{\delta(\alpha, \beta)} \sumfour_{a_1, b_1, a_2, b_2 \geq 1}  \frac{\chi(a_1b_1) \mathscr B(a_1, b_1; \al)}{a_1b_1}  \frac{\overline{\chi}(a_2b_2) \mathscr B(a_2, b_2; -\bb)}{a_2b_2} \\
& \cdot  \sumtwo_{n, m \geq 1} \frac{\lambda_f(a_1n) \sigma_3(n;\al)}{(a_1n)^{1/2}} \frac{\overline{\lambda_f}(a_2m) \sigma_3(m; -\bb)}{(a_2m)^{1/2}} W\pr{\frac{(2\pi)^6 a_1^{3}b_1^2a_2^{3}b_2^2nm}{q^{3}}; \al, \bb}. }

\begin{lemma} \label{lem:approxFncLmodular} We have
\est{H(0;\al,\bb)\Lambda(f; \al, \bb) = \Lambda_0(f, \al, \bb) + \Lambda_0(f, \bb, \al).}

\end{lemma}

\begin{proof}
We consider
\est{I := \frac{1}{2\pi i}\int_{(1)} \Lambda\pr{f, s + 1/2; \al, \bb} H(s; \al, \bb) \> \frac{ds}{s}.}
Moving the contour integral to $(-1)$, we obtain that 
\est{I &= \Lambda\pr{f; \al, \bb} H(0; \al, \bb) + \frac{1}{2\pi i}\int_{(-1)}\Lambda\pr{f, s + 1/2; \al, \bb} H(s; \al, \bb) \> \frac{ds}{s} \\
&= \Lambda\pr{f; \al, \bb} H(0; \al, \bb) - \frac{1}{2\pi i}\int_{(1)}\Lambda\pr{f, -s + 1/2; \al, \bb} H(-s; \al, \bb) \> \frac{ds}{s}.}
By the functional equation, we have $\Lambda(f, s + 1/2 ; \al, \bb) = \Lambda(f, -s + 1/2; \bb, \al)$. Moreover, $H$ is an even function, and $H(s;\al, \bb)  = H(s; \bb, \al).$ Therefore, 
\est{\Lambda\pr{f; \al, \bb} H(0; \al, \bb) = & \ \frac{1}{2\pi i}\int_{(1)}\Lambda\pr{f, s + 1/2; \al, \bb} H(s; \al, \bb) \frac{ds}{s} \\
&+\frac{1}{2\pi i}\int_{(1)}\Lambda\pr{f, s + 1/2; \bb, \al} H(s; \al, \bb) \> \frac{ds}{s}.}
The Lemma follows after writing $\Lambda$ as a product of $L$-functions and Gamma functions and using Lemma \ref{eqn:prod3Lfnc}. 

\end{proof}

Next, we let 
\est{V_{\al, \bb}(\xi, \eta ; \mu) = \pr{\frac{\mu}{4\pi^2}}^{ \delta(\al, \bb)} \int_{-\infty}^{\infty} \pr{\frac{\eta}{\xi}}^{it} W\pr{\frac{\xi\eta (4\pi^2)^3}{\mu^{3}}; \al + it, \bb + it} \> dt,}
and
\es{\label{def:Lambda_1} 
\Lambda_1(f; \al, \bb) &= \sumfour_{a_1, b_1, a_2, b_2 \geq 1}  \frac{\chi(a_1b_1) \mathscr B(a_1, b_1; \al)}{a_1b_1} \frac{\overline{\chi}(a_2b_2) \mathscr B(a_2, b_2; -\bb)}{a_2b_2} \\
& \cdot  \sumtwo_{n, m \geq 1} \frac{\lambda_f(a_1n) \sigma_3(n;\al)}{(a_1n)^{1/2}} \frac{\overline{\lambda_f}(a_2m) \sigma_3(m; -\bb)}{(a_2m)^{1/2}} V_{\al, \bb}\pr{a_1^3 b_1^2 n, a_2^3b_2^2m; q}. }

\begin{lemma} \label{lem:approxV} With notation as above, we have
\est{H(0; \al, \bb) \int_{-\infty}^{\infty} \Lambda(f; \al + it, \bb + it) \> dt = \Lambda_1(f; \al, \bb) + \Lambda_1(f; \bb, \al).}

\end{lemma}

The proof follows easily from Lemma \ref{lem:approxFncLmodular}.

\begin{remark} \label{rem:boundV}
The integration over $t$ is added so that the main contribution comes from when $a_1^3b_1^2n \ll q^{3/2 + \eps}$ and $a_2^3b_2^2m \ll q^{3/2 + \eps}$, and we will see this from Lemma \ref{lem:decayV} below. Without the integration over $t$, the ranges of $a_i, b_j, m, n$ that we need to consider satisfy the weaker condition $a_1^3b_1^2na_2^3b_2^2m \ll q^{3 + \eps},$ and the proof presented here does not extend to this range. 
\end{remark}

\begin{lem} \label{lem:decayV}
If $\xi$ or $\eta \gg q^{\frac 32 + \eps},$ then for any $A > 1,$ we have
\est{V_{\al, \bb} (\xi, \eta; q) \ll q^{-A},}
where the implied constant depends on $\eps$ and $A.$
\end{lem}

\begin{proof}
From the definition of $W$ and $V$ and a change of variables ($s + it = w, s- it = z$), we can write $V_{\al, \bb}(\xi, \eta; \mu)$ as
\est{ %& \pr{\frac{q}{4\pi^2}}^{\delta(\al, \bb)} \int_{-\infty}^{\infty} \pr{\frac{\eta}{\xi}}^{it} \frac{1}{2\pi i} \int_{(1)} \prod_{j = 1}^3 \Gamma \pr{s + \tfrac {k-1}2 + \alpha_j + it}\Gamma \pr{s + \tfrac {k-1}2 - \beta_j - it} \\
%& \hskip 2.5in H\pr{s; \al, \bb} \pr{\frac{\xi\eta (4\pi^2)^3}{q^{3}}}^{-s} \ \frac{ds}{s} \> dt \\
&  \pr{\frac{q}{4\pi^2}}^{\delta(\al, \bb)} \frac{4\pi}{(2\pi i)^2} \int_{(0)} \prod_{j = 1}^3 \Gamma \pr{w + \tfrac {k-1}2 + \alpha_j }  \pr{\frac{q^{\frac 32}}{ (4\pi^2)^{\frac 32} \xi} }^{w}  \\
& \hskip 2in \times \int_{(1)}   \prod_{j = 1}^3 \Gamma \pr{z + \tfrac {k-1}2 - \beta_j }\pr{\frac{q^{\frac 32}}{ (4\pi^2)^{\frac 32} \eta} }^{z}  H\pr{\tfrac{z + w}{2}; \al, \bb}  \> \frac{dz \> dw}{z + w}.  }
When $\xi \gg q^{\frac 32 + \eps},$ we move the contour integral over $w$ to the far right, and similarly, when $\eta \gg q^{\frac 32 + \eps}$, we move the contour integral over $z$ to the far right.  The lemma then follows.
\end{proof}

\section{Setup for the proof of \ref{thm:mainmoment} and diagonal terms} \label{sec:setupmoment}

From Lemma \ref{lem:approxV}, we have that for $\alpha_i, \beta_i \ll 1/\log q,$
	\es{\label{eqn:HM6} H(0;\al,\bb) \mathcal M_6(q) = \frac{2}{\phi(q)}  \sum_{\substack{\chi \mod q \\ \chi(-1) = (-1)^k}} {\sum_{f \in \mathcal H_\chi}}^h  \{\Lambda_1\pr{f; \al, \bb} + \Lambda_1\pr{f; \bb, \al}\}.}
	 Therefore, to evaluate $\mathcal M_6(q)$, it is sufficient to compute asymptotically
\es{\label{eqn:MainboundLambda1} 
\mathscr M_1(q; 
\al, \bb) := \frac{2}{\phi(q)}  \sum_{\substack{\chi \mod q \\ \chi(-1) = (-1)^k}} {\sum_{f \in \mathcal H_\chi}}^h  \Lambda_1\pr{f; \al, \bb}.}
Applying the Petersson's formula, we obtain that
\est{\sumh_{f \in \mathcal H_\chi } \overline{\lambda}_f(a_2m) \lambda_f(a_1n) = \delta_{a_2m = a_1n} + \sigma_\chi(a_2m, a_1n),}
where $\sigma_\chi(a_2m, a_1n)$ is defined as in (\ref{lem:orthogonality}).
%\est{\sigma_\chi(a_2m, a_1n) &= 2\pi i^{-k} \sum_{c \equiv 0 \mod q} c^{-1}S_\chi(a_2m, a_1n; c) J_{k - 1}\pr{\frac{4\pi}{c} \sqrt{a_2ma_1n}} \\
%&=  2\pi i^{-k} \sum_{c = 1}^{\infty} (cq)^{-1}S_\chi(a_2m, a_1n; cq) J_{k - 1}\pr{\frac{4\pi}{cq} \sqrt{a_2ma_1n}}, }
%and $S_\chi$ is the Kloosterman sum defined by 
%\est{S_{\chi} (a_2m, a_1n; cq) = \sum_{a\bar{a} \equiv 1 \mod {cq}} \chi(a) e \pr{\frac{aa_2m + \bar{a}a_1n}{cq}}.}
We then write
\es{\label{decomM1} \mathscr M_1(q; \al, \bb) = \mathscr D(q; \al, \bb) + \mathscr K(q; \al, \bb),} 
where $\mathscr D(q; \al, \bb)$ is the diagonal contribution from $\delta_{a_2m = a_1n},$ and $\mathscr K(q; \al, \bb)$ is the contribution from $\sigma_\chi(a_2m, a_1n).$

In Section \ref{sec:diag} below, we will show that the term $\mathscr D(q; \al, \bb)$ contributes one of the twenty terms in Conjecture \ref{conj:mainterm}, specifically the term corresponding to $\mathcal M(q; \al + it, \bb + it).$ Moreover, $\mathscr K(q;\al, \bb)$ gives another nine terms in the conjecture, namely those transpositions in $S_6/S_3 \times S_3$ which switches $\alpha_i$ and $\beta_j$ for a fixed $i, j = 1, 2, 3$.   We explicitly work out one of these terms in Proposition \ref{prop:mainTM}. Similarly, $\mathscr D(q; \bb, \al)$ gives rise to the term corresponding to $\mathcal M(q; \bb + it, \al + it),$ and the last nine expressions arise from $\mathscr K(q; \bb, \al)$.

\subsection{Evaluating the diagonal terms $\mathscr D (q; \al, \bb)$} \label{sec:diag}
We recall that
\est{ %\label{diag1} 
\mathscr D(q; \al, \bb) &= \frac{2}{\phi(q)} \sum_{\substack{\chi \mod q \\ \chi(-1) =(-1)^k}} \sumsix_{\substack{a_1, b_1, a_2, b_2, m, n \geq 1 \\ a_1n = a_2m}} \chi(a_1b_1) \overline{\chi}(a_2b_2) \frac{ \mathscr B(a_1, b_1; \al)}{a_1b_1}\frac{ \mathscr B(a_2, b_2; -\bb)}{a_2b_2} \\
& \hskip 2in \times  \frac{\sigma_3(n;\al)\sigma_3(m; -\bb)}{(a_1n)^{\frac 12} (a_2m)^{\frac 12}} V_{\al, \bb}\pr{a_1^3 b_1^2 n, a_2^3b_2^2m; q}.}
We will compute the diagonal contribution in the following lemma.
\begin{lem} \label{lem:diagonal} With the same notations as above, we have
\est{\mathscr D(q; \al, \bb) = H(0; \al, \bb)\int_{-\infty}^{\infty} \mathcal M(q; \al + it, \bb + it) \> dt +  O\pr{q^{-3/4 + \eps}}.}
\end{lem}
\begin{proof}
We apply the orthogonality relation for Dirichlet characters in (\ref{eqn:orthodirichlet}) and obtain that  for $(a_ib_i, q) = 1,$
%\es{ \label{diagonality:chi} \frac{2}{\phi(q)}\sum_{\substack{\chi \mod q \\ \chi(-1) =(-1)^k}} \chi(m) \overline{\chi} (n) &= \frac{1}{\phi(q)} \sum_{\chi \mod q} \pg{\chi(m \overline{n}) + (-1)^k \chi(-m\overline{n})} \\
%&= \left\{\begin{array}{ll} 1 & {\rm if} \ m \equiv n \mod q \\
%(-1)^k & {\rm if} \ m \equiv - n \mod q \\
%0 & {\rm otherwise}, \end{array} \right. }
\est{\frac{2}{\phi(q)}\sum_{\substack{\chi \mod q \\ \chi(-1) =(-1)^k}} \chi(a_1b_1) \overline{\chi} (a_2b_2)  = \left\{\begin{array}{ll} 1 & {\rm if} \ a_1b_1 \equiv a_2b_2 \mod q \\
(-1)^k & {\rm if} \ a_1b_1 \equiv - a_2b_2 \mod q \\
0 & {\rm otherwise}. \end{array} \right.}
When $a_1b_1 \geq q/4$ or $a_2b_2 \geq q/4$, we have that  $a_1^3b_1^2n \geq q^2/16$ or $a_2^3b_2^2m \geq q^2/16$. From Lemma \ref{lem:decayV}, $V_{\al, \bb}(a_1^3b_1^2n, a_2^3b_2^2m; q) \ll q^{-A}$ in that range so the contribution from these terms is negligible. Hence the main contribution from $\mathscr D(q; \al, \bb)$ comes from the terms with $a_1b_1 = a_2b_2$ when $(a_ib_i, q) = 1,$ and 
\est{ \mathscr D(q; \al, \bb) &=  \sumsix_{\substack{a_1, b_1, a_2, b_2, m, n \geq 1 \\ a_1n = a_2m \\ a_1b_1 = a_2b_2 \\ (a_ib_i, q) = 1 }}  \frac{ \mathscr B(a_1, b_1; \al)}{a_1b_1} \frac{ \mathscr B(a_2, b_2; -\bb)}{a_2b_2}\\
& \hskip 1in \times  \frac{\sigma_3(n;\al)\sigma_3(m; -\bb)}{(a_1n)^{\frac 12} (a_2m)^{\frac 12}} V_{\al, \bb}\pr{a_1^3 b_1^2 n, a_2^3b_2^2m; q} + O(q^{-A}).}
Since $a_1b_1 = a_2b_2$ and $a_1n = a_2m$, $a_1^3b_1^2n = a_2^3 b_2^2 m.$ Therefore, $\mathscr D(q; \al, \bb)$ can be written as
\est{ &  \frac{1}{2\pi i}\pr{\frac{q}{4\pi^2}}^{\delta(\al, \bb)}\int_{-\infty}^{\infty} \int_{(1)} G\pr{\tfrac 12 + s; \al + it, \bb + it} H(s; \al, \bb) \pr{\frac{q}{4\pi^2}}^{3s} \mathcal A \mathcal Z\pr{\tfrac 12 + s, \al, \bb} \frac{ds}{s} \> dt,}
where we have used Equations (\ref{def:Cs}) to (\ref{def:mathcalAs}). 

Note that $\mathcal A(s; \al, \bb)$ is absolutely convergent when $\tRe(s) > \frac 14 + \eps.$  Furthermore, the pole at $s = -(\alpha_i - \beta_j)/2$ from the zeta factor $\mathcal Z\pr{\tfrac 12 + s; \al, \bb}$ is cancelled by the zero at the same point from $H(s; \al, \bb).$  Thus, in the region $\tRe(s) > -\tfrac 14 + \eps,$ the integrand is analytic except for a simple pole at $s = 0.$   Moving the line of integration to $\tRe(s) = -1/4 + \eps,$ we obtain that $\mathscr D(q; \al, \bb)$ is
\est{ \pr{\frac{q}{4\pi^2}}^{\delta(\al, \bb)}\int_{-\infty}^{\infty}  G\pr{\tfrac 12; \al + it, \bb + it} H(0; \al, \bb)\mathcal A \mathcal Z \pr{\tfrac 12; \al, \bb} \> dt + O\pr{q^{-3/4 + \eps}}.}
The lemma now follows from (\ref{def:Mqalbb}) and upon noting that $\mathcal A \mathcal Z \pr{\tfrac 12; \al, \bb} = \mathcal A \mathcal Z \pr{\tfrac 12; \al + it, \bb + it}.$ 
\end{proof}

\section{Setup for the off-diagonal terms $\mathscr K(q; \al, \bb)$} \label{sec:offdiagsetup}

Define $\mathcal K f = i^{-k} f + i^k \bar{f}$. If $g$ is a real function, then $g\mathcal K f = \mathcal K (gf).$ Applying orthogonality relation for $\chi$ from (\ref{eqn:orthodirichlet}) to $\mathscr K(q; \al, \bb)$, we obtain that 

\est{& \mathscr K(q; \al, \bb) = 2\pi \sumfour_{\substack{a_1, b_1, a_2, b_2 \geq 1 \\ (a_1a_2b_1b_2, q) = 1}}  \frac{ \mathscr B(a_1, b_1; \al)}{a_1b_1} \frac{ \mathscr B(a_2, b_2; -\bb)}{a_2b_2} \sumtwo_{\substack{n, m \geq 1}} \frac{ \sigma_3(n;\al)}{(a_1n)^{1/2}} \frac{ \sigma_3(m; -\bb)}{(a_2m)^{1/2}} \\
& \hskip 0.3in \times V_{\al, \bb}\pr{a_1^3 b_1^2 n, a_2^3b_2^2m; q} \sum_{c = 1}^\infty \frac{1}{cq} J_{k-1} \pr{\frac{4\pi}{cq} \sqrt{a_2ma_1n} }    \mathcal K \sumstar_{\substack{a \mod{cq} \\ a \equiv \overline{a_1b_1}a_2b_2\mod q} } e\pr{\frac{aa_2m + \bar{a}a_1n}{cq}} ,}
% + i^k\sumstar_{\substack{a \mod{cq} \\ a \equiv \overline{a_1b_1}a_2b_2\mod q}}  e\pr{-\frac{aa_2m + \bar{a}a_1n}{cq}} }, }
where $\sumstar $ denotes a sum over reduced residues. Let $f$ be a smooth partition of unity such that 
\est{\sumd_M f\pr{\frac m M} = 1, }
where $f$ is supported in [1/2, 3] and $\sumd_M$ denotes an dyadic sum over $M = 2^k$, $k\geq 0$.

Rearranging the sum, we have
\est{\mathscr K(q; \al, \bb) = \frac {2\pi}{q}&\sumfour_{\substack{a_1, b_1, a_2, b_2 \geq 1 \\ (a_1a_2b_1b_2, q) = 1}} \frac{ \mathscr B(a_1, b_1; \al)}{a_1^{\frac 32}b_1} \frac{ \mathscr B(a_2, b_2; -\bb)}{a_2^{\frac 32}b_2} \sumd_M \sumd_N S(\mathbf a, \mathbf b, M, N; \al, \bb), }
where 
\est{S(\mathbf a, \mathbf b, M, N ; \al, \bb) &:= \sum_{c = 1}^{\infty}\frac{1}{c} \sumtwo_{m, n \geq 1} \frac{\sigma_3(n;\al)\sigma_3(m; -\bb)}{n^{\frac 12}m^{\frac 12}}\mathcal F(\mathbf a, \mathbf b, m, n, c),}

\es{\label{def:F} \mathcal F(\mathbf a, \mathbf b, m, n, c) := 
&  \ \mathcal G(\mathbf a, \mathbf b, m,n, c) \ \mathcal K \sumstar_{\substack{a \mod{cq} \\ a \equiv \overline{a_1b_1}a_2b_2\mod q} } e\pr{\frac{aa_2m + \bar{a}a_1n}{cq}} J_{k-1} \pr{\frac{4\pi}{cq} \sqrt{a_2ma_1n} } ,}
and
\es{\label{def:G} \mathcal G(\mathbf a, \mathbf b, m, n, c) := 
	&   V_{\al, \bb}\pr{a_1^3 b_1^2 n, a_2^3b_2^2m; q}  f\pr{\frac{m}{M}}f\pr{\frac{n}{N}} .}
%and $f$ provided smooth partition of unity and $\sumd$ denotes the dyadic sum.

As described in the outline of the paper, we now take the following steps to compute $\mathscr K(q; \al, \bb).$ \\

We write 
\es{\label{decompK}\mathscr K(q; \al, \bb) = \mathscr K_M(q; \al, \bb) + \mathscr K_E(q; \al,\bb),}
where $\mathscr K_M(q; \al, \bb)$ is the contribution from the sum over $c < C,$ where $C = \frac{\sqrt{a_1a_2MN}}{q^{\frac 23}},$ and $\mathscr K_E(q; \al, \bb)$ is the rest. We will show that the contribution from $\mathscr K_E(q; \al, \bb)$ is small in Section \ref{sec:truncationC}. This is possible by the decay of the Bessel functions and such a truncation bounds the size of the conductor inside the exponential sum.

For $\mathscr K_M(q; \al, \bb)$, we start by reducing the conductor inside the exponential sum from $cq$ to $c$ in Section \ref{sec:treatmentexp}.  This step takes advantage of the average over $\chi \bmod q$.

Before we show each step, we provide properties of Bessel functions
that will be used later. 
\begin{lem} \label{lem:Besselresult}
We have
 \es{\label{asympJxbig} J_{k-1} (2\pi x) = \frac{1}{\pi\sqrt x}\pg{ W(2\pi x)\e{x - \frac k4 + \frac 18}  +  \overline{W}(2\pi x)\e{-x + \frac k4 - \frac 18}},}
where $W^{(j)}(x) \ll_{j, k} x^{-j} $.  Moreover, 
\es{\label{asympJxSm} J_{k-1} (2 x) = \sum_{\ell = 0}^{\infty} (-1)^{\ell} \frac{x^{2\ell + k - 1}}{\ell! (\ell + k - 1)!},}
and
\es{\label{bound:Bessel1}  J_{k-1}(x) \ll \min (x^{-1/2}, x^{k-1}).} % \ll  \min (x^{-1/2}, x^{3}). }
Finally, the following integration is used when calculating the main terms of $\mathscr K(q; \al, \bb)$. If $\alpha, \beta, \gamma > 0$, then
\es{\label{int:bessel} \mathcal K \int_0^\infty \e{(\alpha + \beta)x + \gamma x^{-1}}J_{k - 1}(4\pi \sqrt{\alpha \beta x})\> \frac{dx}{x} = 
2\pi J_{k -1}(4\pi\sqrt{\alpha\gamma})J_{k -1}(4\pi\sqrt{\beta\gamma}),
}
and the integration is 0 if $\alpha, \beta > 0$ and $\gamma \leq 0.$
\end{lem}
These results are standard.  We refer the reader to \cite{Watt} for the first three claims, and to \cite{Ob} for the last claim.

\subsection{Truncating the sum over $c$.} \label{sec:truncationC}
In this section we show that we can truncate the sum over $c$ in $S(\mathbf a, \mathbf b, M, N; \al, \bb)$ with small error contribution. 
\begin{prop} \label{prop:truncateC} Let $C = \frac{\sqrt{a_1a_2MN}}{q^{\frac 23}}$, $k \geq 5$, and $\mathcal F(\mathbf a, \mathbf b, m, n, c)$ be defined as in (\ref{def:F}). Further, let
\est{\mathscr K_E(q; \al, \bb) = \frac {2\pi}{q}&\sumfour_{\substack{a_1, b_1, a_2, b_2 \geq 1 \\ (a_1a_2b_1b_2, q) = 1}} \frac{ \mathscr B(a_1, b_1; \al)}{a_1^{\frac 32}b_1} \frac{ \mathscr B(a_2, b_2; -\bb)}{a_2^{\frac 32}b_2} \sumd_M \sumd_N S_E(\mathbf a, \mathbf b, M, N; \al, \bb), }
where 
\est{S_E(\mathbf a, \mathbf b, M, N; \al, \bb) &= \sum_{c \geq C} \frac{1}{c} \sumtwo_{m, n \geq 1} \frac{\sigma_3(n;\al)\sigma_3(m; -\bb)}{n^{\frac 12}m^{\frac 12}}\mathcal F(\mathbf a, \mathbf b, m, n, c).}

Then 
$$ \mathscr K_E(q;\al, \bb) \ll q^{-\frac 5{12} + \eps}.$$
\end{prop}

\begin{proof}
Note that the contribution of terms when $a_1^3b_1^2N$ or $a_2^3b_2^2M$ is $\gg q^{\frac 32 + \eps}$ is $\ll_{\eps, A} q^{-A}$, due to the fast decay rate of $\mathcal G(\mathbf a, \mathbf b, m, n, c)$ defined in (\ref{def:G}).  Thus we will discount such terms in the rest of the proof. 
For $k\geq 5$, we let
\est{
\tD(m, n) &= \tD(m, n, \mathbf a, \mathbf b; C)  \\
&= \sum_{c\geq C} \frac{1}{c} \ \mathcal K \sumstar_{\substack{a \mod{cq} \\ a \equiv \overline{a_1b_1}a_2b_2\mod q} } e\pr{\frac{aa_2m + \bar{a}a_1n}{cq}} J\bfrac{4 \pi \sqrt{a_1a_2 mn}}{cq}  \\
&= \tD_1(m, n) + \tD_2(m, n),
}
where
\est{
\tD_1(m, n) := \sum_{\substack{c\geq C\\(c, q) = 1}} \frac{1}{c} \  \mathcal K \sumstar_{\substack{a \mod{cq} \\ a \equiv \overline{a_1b_1}a_2b_2\mod q} } e\pr{\frac{aa_2m + \bar{a}a_1n}{cq}} J\bfrac{4 \pi \sqrt{a_1a_2 mn}}{cq},
}
and $\tD_2(m, n)$ is the sum of the terms where $(c, q)>1$. Now for $(c, q) = 1$, the Weil bound gives
\est{
\sumstar_{\substack{a \mod{cq} \\ a \equiv \overline{a_1b_1}a_2b_2\mod q} }e\pr{\frac{aa_2m + \bar{a}a_1n}{cq}} \ll c^{1/2+\eps} \sqrt{(a_1m, a_2n, c)},
}
and from the bound in (\ref{bound:Bessel1}), we have
\est{
J\bfrac{4 \pi \sqrt{a_1a_2 mn}}{cq} \ll \bfrac{\sqrt{a_1a_2mn}}{cq}^{k-1}.
}

When $a_1^3b_1^2N$ and $a_2^3b_2^2M \ll q^{\frac 32 + \eps},$ we obtain that for $k \geq 5,$

\est{
\sum_{m, n} & \frac{\sigma_3(n;\al)\sigma_3(m; -\bb)}{\sqrt{mn}} \tD_1(m, n) \mathcal G(\mathbf a, \mathbf b, m, n, c) \ll q^{\eps} M^{\frac{k}{2}}N^{\frac k2}\sum_{c\geq C} \frac{1}{c} c^{1/2+\eps} \bfrac{\sqrt{a_1a_2}}{cq}^{k-1} \ll q^{\frac 7{12} + \eps}.
%M^{\frac 34}N^{\frac 34} a_1^{\frac 14} a_2^{\frac 14} q^{\frac 14 - \frac k2 + \eps} \\
%&\ll q^{\frac 52 - \frac k2 + \eps} \ll q^{\frac 12+ \eps}.
}
%where $\mathcal G (\mathbf a, \mathbf b, m, n, c)$ is defined as in (\ref{def:G}).  
In the above, we have used that $\max (a_1^3b_1^2N, a_2^3b_2^2M) \ll q^{\frac 32 + \eps}$.  Then summing over $a_1, b_1, a_2, b_2$ gives the desired bound.

Now, for $(c, q) >1$, we use the bound
\est{
\sumstar_{\substack{a \mod{cq} \\ a \equiv \overline{a_1b_1}a_2b_2\mod q} }e\pr{\frac{aa_2m + \bar{a}a_1n}{cq}} \ll (cq)^{1/2+\eps} \sqrt{(a_1m, a_2n, cq)}.
}
Hence, for $q|c$ and $k \geq 5,$ we obtain
\est{
\sum_{m, n} & \frac{\sigma_3(n;\al)\sigma_3(m; -\bb)}{\sqrt{mn}} \tD_2(m, n) \mathcal G(\mathbf a, \mathbf b, m, n, c) \ll q^{\eps}M^{\frac k2}N^{\frac k2} \sum_{\substack{c\geq C\\q|c}} \frac{1}{c} (qc)^{\frac 12} \bfrac{\sqrt{a_1a_2}}{cq}^{k-1}  \ll q^{\frac 1{12} + \eps}.
%\ll M^{\frac 34}N^{\frac 34} a_1^{\frac 14} a_2^{\frac 14} q^{\newr{-\frac 14} - \frac k2 + \eps}\\
%&\ll q^{\newr{2} - \frac k2 + \eps} \ll q^{ \newr{\eps}}.
} 
Then summing over $a_1, b_1, a_2, b_2$ gives the desired bound.  

\end{proof}

From this proposition,  we are left to consider only 
\es{\label{def:KM} \mathscr K_M(q; \al, \bb) = \frac {2\pi}{q}&\sumfour_{\substack{a_1, b_1, a_2, b_2 \geq 1 \\ (a_1a_2b_1b_2, q) = 1}} \frac{ \mathscr B(a_1, b_1; \al)}{a_1^{\frac 32}b_1} \frac{ \mathscr B(a_2, b_2; -\bb)}{a_2^{\frac 32}b_2} \sumd_M \sumd_N S_M(\mathbf a, \mathbf b, M, N; \al, \bb), }
where 
\es{\label{def:SM} S_M(\mathbf a, \mathbf b, M, N; \al, \bb) &:= \sum_{c < C} \frac{1}{c} \sum_{m, n \geq 1 } \frac{\sigma_3(n;\al)\sigma_3(m; -\bb)}{n^{\frac 12}m^{\frac 12}}\mathcal F(\mathbf a, \mathbf b, m, n, c).}

\subsection{Treatment of the exponential sum} \label{sec:treatmentexp}

Next, we reduce the conductor in the exponential sum in $\mathcal F(\mathbf a, \mathbf b, m, n, c)$ before applying Voronoi summation.

\begin{lemma}[Treatment of the exponential sum] 
	Assume that $(c, q) = 1$ and let 
	\est{
		Y := \sumstar_{\substack{a \mod{cq} \\ a \, \equiv \, \overline{a_1b_1}a_2b_2 \mod q} }e\pr{\frac{au + \bar{a}v}{cq}}.
	}
Then we have 
\begin{eqnarray*}
	Y &=&e\bfrac{(a_2b_2)^2u + (a_1b_1)^2 v}{cqa_1b_1a_2b_2} \sumstar_{x \mod{c}} e\bfrac{\bar q (a_1b_1x - a_2b_2) u}{a_1b_1c} e\bfrac{\bar q (a_2b_2\bar x - a_1b_1) v}{a_2b_2c}.
\end{eqnarray*}
\end{lemma}

\begin{proof}
By Chinese Remainder Theorem, for each $a \mod{cq}$, there exist unique $x \mod{c}$ and $y \mod{q}$ such that 
\es{ \label{eqn:aqc}
a = xq\bar q + y c \bar c,
}
where $\bar q$ denotes the inverse of $q$ modulo $c$, and $\bar c$ denotes the inverse of $c$ modulo $q$.  Using (\ref{eqn:aqc}) and the reciprocity relation
\est{\frac{\overline{a}}{b} + \frac{\overline{b}}{a} \equiv \frac{1}{ab} \mod 1, }
where $(a, b) = 1,$  $\overline{a}$ is the inverse of $a \bmod b$, and $\overline{b}$ is the inverse of $b \bmod a$, we obtain that 
\begin{align*}
Y &=  e\bfrac{\overline{a_1b_1}a_2b_2 \bar c u + a_1b_1\overline{a_2b_2} \bar c v}{q} \sumstar_{x \mod{c}} e\bfrac{x\bar qu + \bar x \bar q v}{c}.
\end{align*}
Thus
\begin{eqnarray*}
Y &=& e\bfrac{(a_2b_2)^2u + (a_1b_1)^2 v}{cqa_1b_1a_2b_2} \sumstar_{x \mod{c}} e\bfrac{x\bar q u + \bar x \bar qv}{c} e\bfrac{-\bar q a_2b_2u }{ca_1b_1}e\bfrac{- \bar q a_1b_1 v}{ca_2b_2},
%\\
%&=&e\bfrac{(a_2b_2)^2u + (a_1b_1)^2 v}{cqa_1b_1a_2b_2} \sumstar_{x \mod{c}} e\bfrac{\bar q (a_1b_1x - a_2b_2) u}{a_1b_1c} e\bfrac{\bar q (a_2b_2\bar x - a_1b_1) v}{a_2b_2c}.
\end{eqnarray*}
and the lemma follows. 
\end{proof}

Note that when $c<C <q$, we automatically have $(c, q) = 1 $.  The point of this lemma is that we may treat $e\bfrac{(a_2b_2)^2u + (a_1b_1)^2 v}{cqa_1b_1a_2b_2}$ as a smooth function with small derivatives, while the other exponentials have conductor at most $ca_ib_i \leq q^{1+\eps}$ after truncation.  It should be noted however, that we are most concerned with the contribution from the transition region of the Bessel function, where the conductor $ca_ib_i$ should be thought of as around size $q^{1/2}$.

\section{Applying Voronoi Summation}\label{sec:appvoronoi}

To calculate $\mathscr K_M(q; \al, \bb)$ as defined in \eqref{def:KM}, we start by evaluating 
$S_M(\mathbf a, \mathbf b, M, N; \al, \bb)$ defined in (\ref{def:SM}).  We write
\est{S_M(\mathbf a, \mathbf b, M, N; \al, \bb) = \sum_{c < C} \sumstar_{x \mod{c}} \frac{1}{c} \pr{\Sc^+_{\al, \bb} (c, x) + \Sc^-_{\al, \bb}(c,x)},}
where 
\est{\Sc^\pm_{\al, \bb}(c, x) & = i^{\mp k} \sum_{m \geq 1} \frac{\sigma_3(m; -\bb)}{m^{\frac 12}} f\pr{\frac{m}{M}}   e\pr{\pm \frac{a_2^2b_2m }{cqa_1b_1} } e\pr{\pm \frac{\bar q (a_1b_1x - a_2b_2) a_2m}{a_1b_1c} }\times \\
	& \ \ \ \times \sum_{n \geq 1} \sigma_3(n;\al) F^\pm_{\al, \bb}(m, n, c) e\pr{\pm \frac{\bar q (a_2b_2\bar x - a_1b_1) a_1n}{a_2b_2c}},}
where
\est{ F^\pm_{\al, \bb}(m, n, c) = \frac{1}{n^{\frac 12}} V_{\al, \bb}\pr{a_1^3 b_1^2 n, a_2^3b_2^2m; q}  J_{k-1} \pr{\frac{4\pi}{cq} \sqrt{a_2ma_1n} }  \ e\pr{\pm \frac{ a_1^2b_1 n}{cqa_2b_2}}  f\pr{\frac{n}{N}}  . }

 Let 
 \begin{equation}\label{eqn:lambda1eta1}
 \frac{\lambda_1}{\eta_1} = \frac{\bar q (a_2b_2\bar x - a_1b_1) a_1}{a_2b_2c},
 \end{equation} and 
 \begin{equation} \label{eqn:lambda2eta2}
 \frac{\lambda_2}{\eta_2} = \frac{\bar q (a_1b_1x - a_2b_2) a_2}{a_1b_1c},
 \end{equation} where $(\lambda_1, \eta_1) = (\lambda_2, \eta_2) = 1.$ Moreover we define
 
 \est{ %\label{def:Vprime} 
 	\mathcal V^\pm_{\al, \bb}(c; {\bf a}, {\bf b}, y, z) = \frac{1}{y^{1/2 }} \frac{1}{z^{1/2}} &V_{\al, \bb} \pr{a_1^3b_1^2y, a_2^3b_2^2z; q} J_{k-1} \pr{\frac{4\pi \sqrt{a_2za_1y}}{cq}} \times \\
 	&\times i^{\mp k}\e{\pm \frac{a_1^2b_1y}{cqa_2b_2} \pm  \frac{a_2^2b_2z}{cqa_1b_1}} f\pr{\frac yN} f\pr{\frac zM} .}
 We then apply Voronoi Summation as in Theorem \ref{thm:voronoi} to the sum over $n, m$ and obtain that $\Sc^+_{\al, \bb}(c, x) + \Sc^-_{\al, \bb}(c, x)$ is

\est{\sum_{i= 1}^3 \sum_{j = 1}^3 \Res_{s_1 = 1 - \alpha_i} \Res_{s_2 = 1 + \beta_j} (\mathcal T_{M, \al, \bb}^+(c,x, s_1, s_2) + \mathcal T_{M, \al, \bb}^-(c,x, s_1, s_2) ) +  \sum_{i = 1}^8 \pr{\mathcal T_{i, \al, \bb}^+(c,x) + \mathcal T_{i, \al, \bb}^-(c,x) },}
where in the region of absolutely convergence, 
\es{\label{def:calT1} \mathcal T_{M, \al, \bb}^\pm(c,x, s_1, s_2) &:= \mathcal F_M^\pm(c; \al, \bb) D_3\pr{s_1, \pm \frac{\lambda_1}{\eta_1}, \al }  D_3\pr{s_2, \pm \frac{\lambda_2}{\eta_2}, -\bb } , }

\est{   \mathcal F_M^\pm(c; \al, \bb)  &:=\mathcal F_M^\pm(c, s_1, s_2; \al, \bb)  := \int_{0}^{\infty} \int_0^{\infty} y^{s_1 - 1}z^{s_2 - 1} \mathcal V^\pm_{\al, \bb}(c; {\bf a}, {\bf b}, y, z) \>dy \> dz  ,}

\es{\label{def:calT2} \mathcal T_{1, \al, \bb}^\pm (c,x) &:= \frac{\pi^{3/2 + \alpha_1 + \alpha_2 + \alpha_3}}{\eta_1^{3+\alpha_1 + \alpha_2 + \alpha_3}} \sum_{i = 1}^3 \Res_{s = 1 + \beta_i} D_3\pr{s_2, \pm \frac{\lambda_2}{\eta_2}, -\bb } \sum_{n = 1}^{\infty} A_3\pr{n, \pm \frac{\lambda_1}{\eta_1}, \al} \mathcal F_1^\pm(c,n; \al, \bb),}

\est{\mathcal F_1^\pm(c,n; \al, \bb) &= \mathcal F_1^\pm(c,n, s; \al, \bb)=  \int_{0}^{\infty} \int_0^{\infty}  z^{s - 1} \mathcal V^\pm_{\al, \bb}(c; {\bf a}, {\bf b}, y, z) U_3\pr{\frac{\pi^3ny}{\eta_1^3}; \al} \>dy \> dz,  }
and $\mathcal T_{i, \al, \bb}^\pm (c,x)$, where $i = 2, 3, 4$ are defined similarly.  Further,

\es{\label{def:calT6} \mathcal T_{5, \al, \bb}^\pm(c,x) &:= \frac{\pi^{3 + \sum_{i = 1}^3 (\alpha_i - \beta_i)}}{\eta_1^{3+\sum_{i=1}^3 \alpha_i}\eta_2^{3 - \sum_{i=1}^3 \beta_i}}   \sum_{n = 1}^{\infty} \sum_{m = 1}^\infty A_3\pr{n, \pm \frac{\lambda_1}{\eta_1}, \al}  A_3\pr{m, \pm \frac{\lambda_2}{\eta_2}, -\bb} \mathcal F_5^\pm(c,n, m; \al, \bb),}

\est{\mathcal F_5^\pm(c,n, m; \al, \bb) &= \int_{0}^{\infty} \int_0^{\infty} \mathcal V^\pm_{\al, \bb}(c; {\bf a}, {\bf b}, y, z) U_3\pr{\frac{\pi^3mz}{\eta_2^3}; -\bb} U_3\pr{\frac{\pi^3ny}{\eta_1^3}; \al} \>dy \> dz, }
and $\mathcal T_{i, \al, \bb}^\pm(c,x)$, where $i = 6, 7, 8$ are defined similarly. 

As mentioned in Section \ref{sec:setupmoment}, there are nine terms from $\mathscr K(q; \al, \bb).$ In particular, we will show that these terms arise from 
\est{\sum_{c < C} \sumstar_{x \mod c} \frac{1}{c} \pr{\mathcal T_{M, \al, \bb}^+(c,x, s_1, s_2) + \mathcal T_{M, \al, \bb}^-(c,x, s_1, s_2)},}
and in fact each term comes from the residues at $s_1 = 1 - \alpha_i$ and $s_2 = 1 + \beta_j$ for $i = 1, 2, 3.$ We state the contribution from the residues $s_1 = 1 - \alpha_1$ and $s_2 = 1 + \beta_1$ in Proposition \ref{prop:mainTM} below, and prove it in Section \ref{sec:provepropTM}. By symmetry, the analogous result holds for the other residues.  Then, we will show that the rest of $\mathcal T_{i, \al, \bb}^\pm(c,x)$ are negligible in Section \ref{sec:properror} as stated in Proposition \ref{prop:Terror}.
%In particular, Theorem \ref{thm:mainmoment} follows from the following two propositions.

\begin{prop} \label{prop:mainTM} Let
	\est{
	R_{\alpha_1, \beta_1} := \frac {2\pi}{q}&\sumfour_{\substack{a_1, b_1, a_2, b_2 \geq 1 \\ (a_1a_2b_1b_2, q) = 1}} \frac{ \mathscr B(a_1, b_1; \al)}{a_1^{\frac 32}b_1} \frac{ \mathscr B(a_2, b_2; -\bb)}{a_2^{\frac 32}b_2} \\
	& \cdot \sumd_M \sumd_N \sum_{c < C} \sumstar_{x \mod c} \frac{1}{c} \Res_{s_1 = 1 - \alpha_1} \Res_{s_2 = 1 + \beta_1} \pr{\mathcal T_{M, \al, \bb}^+(c,x, s_1, s_2) + \mathcal T_{M, \al, \bb}^-(c,x, s_1, s_2)}. }
Then we have
\est{R_{\alpha_1, \beta_1} = H(0; \al, \bb)\int_{-\infty}^{\infty}\mathcal M(q;\pi(\al) + it, \pi(\bb) + it ) \> dt + O(q^{-1/2 + \eps}),}
where $(\pi(\al), \pi(\bb)) = (\beta_1, \alpha_2, \alpha_3; \alpha_1, \beta_2, \beta_3).$
\end{prop}

\begin{prop} \label{prop:Terror} For i = 1,.., 8, define
	\est{
	\mathcal E_i^\pm(q; \al,\bb) := \frac {2\pi}{q}&\sumfour_{\substack{a_1, b_1, a_2, b_2 \geq 1 \\ (a_1a_2b_1b_2, q) = 1}} \frac{ \mathscr B(a_1, b_1; \al)}{a_1^{\frac 32}b_1} \frac{ \mathscr B(a_2, b_2; -\bb)}{a_2^{\frac 32}b_2} \sumd_M \sumd_N \sum_{c < C} \sumstar_{x \mod c} \frac{\mathcal T_{i, \al, \bb}^\pm(c,x)}{c} . }
Then $$ \mathcal E_i^\pm(q; \al, \bb) \ll q^{-1/4 + \eps}.$$
\end{prop}
We will prove this proposition in Section \ref{sec:properror}.

\section{Proof of Proposition \ref{prop:mainTM}} \label{sec:provepropTM}
%Note that both $R_{\alpha_1, \beta_1}$ and $\mathcal M(q; \beta_1, \alpha_2, \alpha_3; \alpha_1, \beta_2, \beta_3)$ are meromorphic in $\alpha_i, \beta_i.$ So it follows that the error term is also meromorphic in $\alpha_i, \beta_i$.  Thus, in the proof of Proposition \ref{prop:mainTM}, it suffices to prove the result for the region $\textrm{Re}(\alpha_2), \textrm{Re}(\alpha_3) > \textrm{Re}(\alpha_1) $ and $\textrm{Re}(\beta_2), \textrm{Re}(\beta_3) < \textrm{Re}(\beta_1) $, the other regions following  by analytic continuation.  
We begin by collecting some lemmas which will be used in this section.
 
 \subsection{Preliminary Lemmas }
 
 \begin{lemma} \label{lem:countingnumberX} Let $(a, \ell) = 1.$ We have
 	\est{ f(c, \ell) := \sumstar_{\substack{x \mod {c\ell} \\ x \equiv a \mod {\ell}}  } 1 = c \prod_{\substack{p| c \\ p \nmid \ell} } \pr{1 - \frac {1}{p}} = \phi(c) \prod_{p | (\ell, c)} \pr{1 - \frac{1}{p}}^{-1} . }
 \end{lemma}
 \begin{proof}
 	We first prove that if $(m, n\ell) = 1$, then
 	\es{\label{eqn:multsumstarMN} f(mn, \ell) = f(n, \ell) \phi(m).}
 	For all $x$ satisfying $(x, mn\ell) = 1$, we can write $x = um \overline m + vn\ell \overline{n\ell},$ where $m\overline m \equiv 1 \mod {n\ell},$ $n\ell\overline {n\ell} \equiv 1 \mod {m},$ and %Also $u$ runs over modulo $n\ell$, and $v$ runs over modulo $m$, where 
 	$(u, n\ell) = (v, m) = 1.$ Moreover $x \equiv a \mod \ell$ if and only if $u \equiv a \mod \ell$.  By Chinese Remainder Theorem,
 	\est{f(mn, \ell) = \sumstar_{\substack{x \mod {mn\ell} \\ x \equiv a \mod {\ell}}  } 1 = \sumstar_{\substack{u \mod {n\ell} \\ u \equiv a \mod {\ell}}  } \sumstar_{v \mod m} 1 = f(n, \ell) \phi(m).}
 	
 	Let $c = c_1 c_2,$ where all prime factors of $c_1$ also divide $\ell$, and $(c_2, \ell) = 1.$ From (\ref{eqn:multsumstarMN}), we have that
 	$f(c, \ell) = f(c_1, \ell) \phi(c_2).$ 
 	
 	Now let $x$ be any residue modulo $c_1\ell$ with $x \equiv a \mod \ell$.  Then $(x, c_1\ell) = 1$  since $(a, \ell) = 1$. Thus all such $x$ can be uniquely written as $x = a + k\ell,$ where $k = 0,..., c_1-1$, so $f(c_1, \ell) = c_1.$ We then have $f(c, \ell) = c_1 \phi(c_2)$, and the statement follows from the identity $\phi(c_2) = c_2 \prod_{p|c_2}\left(1-\frac 1p\right).$   
 \end{proof}

 \begin{lemma} \label{lem:sumT} Let $\alpha, \beta, y, z$ be nonnegative real numbers satisfying $\alpha y, \beta z \ll q^{2}$ and define
 	\est{T = T(y, z, \alpha, \beta)  := \sum_{\delta  = 1 }^\infty \frac{1}{\delta} J_{k - 1} \pr{\frac{4\pi \sqrt{\alpha \beta yz}}{ \delta}} \K \e{\frac{\alpha y}{\delta} + \frac{\beta z}{  \delta}}.} 
	Further, let $L = q^{100}$ and $w$ be a smooth function on $\mathbb R^+$ with $w(x) = 1$ if $0 \leq x \leq 1$, and $w(x) =0 $ if $x > 2.$ Then for any $A>0$, we have 
 	\est{T &= 2\pi \sum_{\ell = 1}^{\infty} w \pr{\frac{\ell}{L}} J_{k - 1} \pr{4\pi \sqrt{ \alpha y\ell}}  J_{k - 1} \pr{4\pi \sqrt{\beta z\ell}} \\
 		& \ \ \ - 2\pi \int_{0 }^{\infty} w \pr{\frac{\ell}{L}} J_{k - 1} \pr{4\pi \sqrt{ \alpha y\ell}}  J_{k - 1} \pr{4\pi \sqrt{\beta z\ell}} \> d\ell + O_A(q^{-A}). }
 	
 \end{lemma}
 \begin{proof}
 	We will follow Iwaniec and Xiaoqing Li's arguments in Section 3 of \cite{IL} to evaluate $T$. Let $\eta(s)$ be a smooth function on $\mathbb R^+$ with $\eta(s) = 0$ if $0 \leq s < 1/4,$ $0 \leq \eta(s) \leq 1$ if $1/4 \leq s \leq 1/2$, and $\eta(s) = 1$ if $s > 1/2.$ We then obtain that
 	\est{T =  \K \sum_{\delta} \frac{\eta(\delta)}{\delta} J_{k - 1} \pr{\frac{4\pi \sqrt{\alpha \beta yz}}{\delta}}  \e{\frac{\alpha y}{\delta} + \frac{\beta z}{ \delta}}.}
 	After inserting this smooth function we apply Poisson summation to obtain that
 	\est{T = \sum_{\ell} \hat F(\ell) := \sum_{\ell} \K \int_0^{\infty} \frac{\eta(u)}{u} \e{\ell u + \frac{\alpha y}{u} + \frac{\beta z}{ u}}  J_{k - 1} \pr{\frac{4\pi \sqrt{\alpha \beta yz}}{ u}} \> du. }
 	By (\ref{asympJxbig}), we can write the integral above in terms of two integrals with the phase   
 	$$\ell u + \frac{\alpha y \ \pm \sqrt{\alpha \beta yz} + \beta z }{u} .$$
 	If $|\ell| > L $, the factor $\ell u$ dominates. Then integrating by parts $A$ times, we have that
 	\est{\int_0^{\infty} \frac{\eta(u)}{u} \e{\ell u + \frac{\alpha y}{u} + \frac{\beta z}{ u}}  J_{k - 1} \pr{\frac{4\pi \sqrt{\alpha \beta yz}}{ u}} \> du \ll q^{-A}. }
 	Therefore 
 	\est{T = \sum_{\ell} \hat F(\ell) w \pr{\frac{|\ell|}{L}} + \ O(q^{-A}).}
Now, we write $\sum_{\ell} \hat F(\ell) w \pr{\frac{|\ell|}{L}} = T_1 - T_2$, where
 	\est{T_1 :=  \sum_{\ell} w \pr{\frac{|\ell|}{L}} \K \int_0^{\infty} \frac{1}{u} \e{\ell u + \frac{\alpha y}{u} + \frac{\beta z}{u}}  J_{k - 1} \pr{\frac{4\pi \sqrt{\alpha \beta yz}}{u}} \> du, }
 	and 
 	\est{T_2 := \sum_{\ell} w \pr{\frac{|\ell|}{L}}\K \int_0^{\infty} \frac{1 - \eta(u)}{u} \e{\ell u + \frac{\alpha y}{u} + \frac{\beta z}{ u}}  J_{k - 1} \pr{\frac{4\pi \sqrt{\alpha \beta yz}}{u}} \> du .}
 	We use (\ref{int:bessel}) to evaluate $T_1$ and obtain
 	\es{ \label {eqn:T1} T_1 = 2\pi \sum_{\ell = 1}^{\infty} w \pr{\frac{\ell}{L}} J_{k - 1} \pr{4\pi \sqrt{ \alpha y\ell}}  J_{k - 1} \pr{4\pi \sqrt{\beta z\ell}}.}
 	For $T_2,$ we note that $\xi(u) = 1 - \eta(u) = 1$ if $0 < s < 1/4,$ $0 \leq \xi(u) \leq 1$ if $1/4 \leq s \leq 1/2$, and $\xi(u) = 0 $ if $s > 1/2.$ 
 	Interchanging the sum over $\ell$ and the integration over $u$ and applying Poisson summation formula, we have
 	\est{T_2 %&= \K \int_0^{\infty}\frac{\xi(u)}{u} \e{ \frac{\alpha y}{\delta} + \frac{\beta z}{ \delta}}  J_{k - 1} \pr{\frac{4\pi \sqrt{\alpha \beta yz}}{ u}} \sum_{\ell} w \pr{\frac{|\ell|}{L}} \e{\ell u}  \> du \\
 		&=  \K \int_0^{\infty}\frac{\xi(u)}{u} \e{ \frac{\alpha y}{\delta} + \frac{\beta z}{\delta}}  J_{k - 1} \pr{\frac{4\pi \sqrt{\alpha \beta yz}}{ u}} \sum_{\ell} L \hat w(L (\ell + u)) \> du.}
 	Since $\hat w(y) \ll (1 + |y|)^{-A},$ the main contribution comes from $\ell = 0$ and $0 \leq u < 1/4.$  Therefore 
 	\es{\label{eqn:T2ab}
 		T_2 &=  \K \int_0^{\infty}\frac{\xi(u)}{u} \e{ \frac{\alpha y}{ \delta} + \frac{\beta z}{\delta}}  J_{k - 1} \pr{\frac{4\pi \sqrt{\alpha \beta yz}}{ u}} L \hat w(L  u) \> du + O(q^{-A})\\
 		&= \K \int_0^{\infty}\frac{1}{u} \e{ \frac{\alpha y}{\delta} + \frac{\beta z}{ \delta}}  J_{k - 1} \pr{\frac{4\pi \sqrt{\alpha \beta yz}}{u}}  L \hat w(L  u) \> du + O(q^{-A})\\
 		&= 2\pi \int_0^{\infty} w \pr{\frac{\ell}{L}} J_{k - 1} \pr{4\pi \sqrt{\alpha y\ell}}  J_{k - 1} \pr{4\pi \sqrt{ \beta z\ell}} \> d\ell + O(q^{-A}), }
 	where the last equality comes from Plancherel's formula and (\ref{int:bessel}).
 \end{proof}

 Next, the following lemma deals with the sum and the integral involving $\ell$.
 \begin{lemma}  \label{lem:evalsumwtozeta}Let $w$ be a smooth function on $\mathbb R^+$ with $w(x) = 1$ if $0 \leq x \leq 1,$ and $w(x) = 0$ if $x > 2.$ Also we let $\gamma$ be a complex number where $\tRe \gamma \ll \frac{1}{\log q}$, $\tRe \gamma < 0$ and $L = q^{100}.$ Then
 	\est{\sum_{\ell = 1}^{\infty} w \pr{\frac \ell L} \frac{1}{\ell^{1 + \gamma}} - \int_0^{\infty} w \pr{\frac \ell L} \frac{1}{\ell^{1+ \gamma}} \> d\ell = \zeta(1 + \gamma) +  O(q^{-20}).}
 \end{lemma}
 \begin{proof}
 	Let $\tilde w(z)$ be the Mellin transform of $w,$ defined by
 	\est{\tilde w(z) = \int_0^{\infty} w(t) \frac{t^z}{t} \> dt.}
 	From the definition, $\tilde w(z)$ is analytic for $\tRe z > 0$, and integration by parts gives
 	$$\tilde w(z) = -\frac 1z\int_0^{\infty} w'(t) t^z dt,
 	$$ so $\tilde w(z)$ can be analytically continued to $\tRe z > -1$ except at $z = 0$ where it has a simple pole with residue $w(0) = 1$. For $\sigma > \min\{0, -\tRe \gamma\},$ we have
 	\es{ \label{sum:ellMellin} \sum_{\ell = 1}^{\infty} w \pr{\frac \ell L} \frac{1}{\ell^{1+\gamma}} &= \sum_{\ell = 1}^{\infty}  \frac{1}{2\pi i} \int_{(\sigma)} \tilde w(z) \pr{\frac \ell L}^{-z} \frac{1}{\ell^{1 + \gamma }} \> dz \\
 		&= \frac{1}{2\pi i} \int_{(\sigma)} \tilde w(z) L^z \zeta(1 + \gamma + z) \> dz.}
 	Shifting the contour to $\textrm{Re} (z) = - 1/4$, we have that (\ref{sum:ellMellin}) is 
 	\est{\zeta(1 + \gamma) + \tilde w(-\gamma) L^{-\gamma} + O(q^{-20}). }
 	The Lemma follows from noting that $\tilde w(-\gamma) L^{-\gamma} =\int_0^{\infty} w \pr{\frac \ell L} \frac{1}{\ell^{1+ \gamma}} \> d\ell$.    
 	
 \end{proof}

 \subsection{Calculation of residues} \label{sec:calRes} In this section, we will calculate 
 $$ \Res_{s_1 = 1 - \alpha_1} \Res_{s_2 = 1 + \beta_1} \pr{\mathcal T_{M, \al, \bb}^+(c,x, s_1, s_2) + \mathcal T_{M, \al, \bb}^-(c,x, s_1, s_2)}.$$
To do this, we essentially need to consider
\est{ \Res_{s_1 = 1 - \alpha_1} D_3\pr{s_1, \pm \frac{\lambda_1}{\eta_1}, \al} y^{s_1 - 1},}
  where $\frac{\lambda_1 }{\eta_1} $ is defined in (\ref{eqn:lambda1eta1}). Let $(a_1b_1, a_2b_2) = \lambda,$ $a_1b_1 = u_1 \lambda, a_2b_2 = u_2 \lambda,$ where $(u_1,u_2) =1. $ Note that $(u_1x - u_2, c) = (u_2\bar x - u_1, c) = \delta. $  Hence $$\delta_1 := ((a_2b_2\bar x - a_1b_1)a_1, a_2b_2c) = \lambda((u_2\bar x - u_1)a_1, u_2c) = \lambda \delta (a_1, u_2 c/\delta),$$
 and $\lambda_1 = \bar q(a_2b_2\bar x-a_1b_1)a_1/\delta_1$ and $\eta_1 = a_2b_2c/\delta_1.$

By (\ref{eqn:OriginalD}), we obtain that 
\est{ \mathcal R_1\left(\frac c\delta, \mathbf a, \mathbf b\right) &:= \Res_{s_1 = 1 - \alpha_1} D_3\pr{s_1, \pm \frac{\lambda_1}{\eta_1}, \al} \\
	&=  \frac{1}{\eta_1^{2 - 2\alpha_1 + \alpha_2 + \alpha_3}} \sumtwo_{\substack{1 \leq a_1, a_2 \leq \eta_1 \\ \eta_1 | a_1a_2}}\zeta\pr{1 - \alpha_1 + \alpha_2, \frac{a_2}{\eta_1}}\zeta\pr{1 - \alpha_1 + \alpha_3, \frac{a_3}{\eta_1} }.}
Hence
 \est{ \Res_{s_1 = 1 - \alpha_1} D_3\pr{s_1, \pm \frac{\lambda_1}{\eta_1}, \al} y^{s_1 - 1}  =\mathcal R_1\left(\frac c\delta, \mathbf a, \mathbf b\right) y^{-\alpha}.}
Similarly, we let $ \mathcal R_2(c, \mathbf a, \mathbf b) := \Res_{s_2 = 1 + \beta_1} D_3\pr{s_2, \pm \frac{\lambda_2}{\eta_2}, -\bb} $. Then
\est{\mathcal R_2\pr{\frac c\delta, \mathbf a, \mathbf b} = \frac{1}{\eta_2^{2 + 2\beta_1 - \beta_2 - \beta_3}} \sumtwo_{\substack{1 \leq a_1, a_2 \leq \eta_2 \\ \eta_2 | a_1a_2}}\zeta\pr{1 + \beta_1 - \beta_2, \frac{a_2}{\eta_2}}\zeta\pr{1 + \beta_1 - \beta_3, \frac{a_3}{\eta_2} },}
and 
\est{\Res_{s_2 = 1 + \beta_1} D_3\pr{s_1, \pm \frac{\lambda_2}{\eta_2}, \al} z^{s_2 - 1}  =\mathcal R_2\left(\frac c\delta, \mathbf a, \mathbf b\right) z^{\beta}.}

\subsection{Computing $R_{\alpha_1, \beta_1}$}
From the previous section, $R_{\alpha_1, \beta_1}$ can be written as 

\est{ R_{\alpha_1, \beta_1} = \frac{2\pi}{q} &\sumfour_{\substack{a_1, b_1, a_2, b_2 \geq 1 \\ (a_1a_2b_1b_2, q) = 1}}  \frac{ \mathscr B(a_1, b_1; \al)}{a_1^{\frac 32}b_1} \frac{ \mathscr B(a_2, b_2; -\bb)}{a_2^{\frac 32}b_2} \sumd_M \sumd_N    \mathcal A({\bf a}, {\bf b}, M, N) + O(q^{-1/2 + \eps}),}
where  $\mathcal A({\bf a}, {\bf b}, M, N)$ is defined as
\es{\label{def:AabMN} \sum_{c = 1}^{\infty}  \frac{\mathcal F (c)}{c}\sumstar_{x \mod c} \mathcal R_1\left(\frac c\delta, \mathbf a, \mathbf b\right) \mathcal R_2\left(\frac c\delta, \mathbf a, \mathbf b\right)  , }
and 
\est{  \mathcal F(c) := \mathcal F_{\al, \bb}(c, {\bf a}, {\bf b})  := \int_{0}^{\infty} \int_0^{\infty} & \frac{1}{y^{1/2 + \alpha_1}} \frac{1}{z^{1/2 - \beta_1}} V_{\al, \bb} \pr{a_1^3b_1^2y, a_2^3b_2^2z; q} J_{k-1} \pr{\frac{4\pi \sqrt{a_2ya_1z}}{cq}} \times \\
&\times f\pr{\frac yN} f\pr{\frac zM}  \mathcal K \e{\frac{a_1^2b_1y}{cqa_2b_2} + \frac{a_2^2b_2z}{cqa_1b_1}}  \>dy \> dz  .}
We remark that we can extend the sum over $c$ to all positive integers in a similar manner as in the truncation argument in Proposition \ref{prop:truncateC}. Now, we let
$$\frac{1}{c^2}\mathcal G(c, \mathbf a, \mathbf b) := \mathcal R_1(c, \mathbf a, \mathbf b) \mathcal R_2(c, \mathbf a, \mathbf b),$$
so that we can write the sum over $c$ in (\ref{def:AabMN}) as
\est{
 &\sum_{c = 1}^\infty  \frac{\mathcal F (c)}{c} \sum_{\delta | c} \sumstar_{\substack{x \mod c \\ (u_1x-u_2, c) = \delta}} \frac{1}{(c/\delta)^2}\mathcal G \pr{\frac{c}{\delta}, {\bf a}, {\bf b}} = \sum_{\delta = 1}^\infty \frac{1}{\delta} \sum_{c = 1}^\infty  \frac{\mathcal F(c\delta) }{c}  \sumstar_{\substack{x \mod {c\delta} \\ (u_1x-u_2, c\delta) = \delta}} \frac{\mathcal G(c, {\bf a}, {\bf b})}{c^2} \\
&= \sum_{\delta = 1}^\infty \frac{1}{\delta} \sum_{c = 1}^\infty  \frac{\mathcal F(c\delta) }{c}  \sumstar_{\substack{x \mod {c\delta} } } \sum_{b | \pr{\frac{u_1x-u_2}{\delta}, c}}  \frac{\mu(b) \mathcal G(c, {\bf a}, {\bf b})}{c^2} \\
&= \sum_{\substack{\delta \geq 1 \\ (\delta, u_1u_2) = 1}} \frac{1}{\delta}  \sum_{\substack{b \geq 1 \\ (b, u_1u_2) = 1}} \frac{\mu(b)}{b^3} \sum_{c \geq 1}  \frac{\mathcal G(cb, {\bf a}, {\bf b}) \mathcal F(cb\delta)}{c^3}  \sumstar_{\substack{x \mod {c\delta b} \\ x \equiv u_2 \bar u_1 \mod {b\delta}} }  1 ,}
where the sum over $x$ is 0 if $(u_1u_2, b\delta) \neq 1$ since $(u_1, u_2) = 1.$ 
Applying Lemma \ref{lem:countingnumberX} to the sum over $x$, we then obtain that
%$$ \sumstar_{\substack{x \mod {c\delta b} \\ x \equiv u_2 \bar u_1 \mod {b\delta}} }  1  = c \prod_{\substack{p| c \\ p \nmid b\delta}} \pr{1 - \frac 1p}.$$
\est{
& \mathcal A(\boldab, M, N)= \sum_{\substack{\delta \geq 1 \\ (\delta, u_1u_2) = 1}} \frac{1}{\delta}  \sum_{\substack{b \geq 1\\ (b, u_1u_2) = 1}} \frac{\mu(b)}{b^3} \sum_{c \geq 1}  \frac{1}{c^2}  \prod_{\substack{p| c \\ p \nmid b\delta}} \pr{1 - \frac 1p} \mathcal F(cb\delta) \ \mathcal G(cb, {\bf a}, {\bf b}) \\
%&= \sum_{\delta} \frac{1}{\delta} \sum_{h | (\delta, u_1u_2)} \mu(h) \sum_{\substack{b \\ (b, u_1u_2) = 1}} \frac{\mu(b)}{b^3} \sum_{c}  \frac{1}{c^2}  \prod_{\substack{p| c \\ p \nmid b\delta}} \pr{1 - \frac 1p} \mathcal F(cb\delta) \ \mathcal G(cb, {\bf a}, {\bf b})\\
&=  \sum_{\substack{h \geq 1 \\ h | u_1u_2}} \frac{\mu(h)}{h} \sum_{\delta \geq 1} \frac{1}{\delta}  \sum_{\substack{b \geq 1\\ (b, u_1u_2) = 1}} \frac{\mu(b)}{b^3} \sum_{c \geq 1}  \frac{1}{c^2}  \prod_{\substack{p| c \\ p \nmid b h \delta}} \pr{1 - \frac 1p} \mathcal F(cbh\delta ) \ \mathcal G(cb, {\bf a}, {\bf b}) \\
%&= \sum_{h | u_1u_2} \frac{\mu(h)}{h} \sum_{\substack{b \\ (b, u_1u_2) = 1}} \frac{\mu(b)}{b^3} \sum_{c}  \frac{\mathcal G(cb, {\bf a}, {\bf b})}{c^2}  \sum_{\gamma | c} \prod_{\substack{p| c \\ p \nmid bh\gamma}} \pr{1 - \frac 1p} \sum_{\substack{\delta \\ (c, \delta) = \gamma}} \frac{1}{\delta} \mathcal F(cbh\delta)  \\
&= \sum_{\substack{h \geq 1 \\ h | u_1u_2}} \frac{\mu(h)}{h} \sum_{\substack{b \geq 1 \\ (b, u_1u_2) = 1}} \frac{\mu(b)}{b^3} \sum_{c \geq 1}  \frac{\mathcal G(cb, {\bf a}, {\bf b})}{c^2}  \sum_{\gamma | c} \prod_{\substack{p| c \\ p \nmid bh\gamma}} \pr{1 - \frac 1p} \sum_{\delta \geq 1} \frac{1}{\delta} \mathcal F(cbh\delta) \sum_{g | (c/\gamma, \delta/\gamma)} \mu(g) \\
&= \sumsharp \mathscr G_{ {\bf a}, {\bf b}}(1; h, b, c, \gamma, g)\sum_{\delta \geq 1} \frac{1}{\delta} \mathcal F(cbhg\gamma \delta),
}
where 
\es{ \label{def:Gscr}\sumsharp \mathscr G_{ {\bf a}, {\bf b}}(s; h, b, c, \gamma, g) = \sum_{\substack{ h \geq 1 \\ h | u_1u_2}} \frac{\mu(h)}{h^s} \sum_{\substack{b \geq 1 \\ (b, u_1u_2) = 1}} \frac{\mu(b)}{b^{2 + s}} \sum_{c \geq 1}  \frac{\mathcal G(cb, {\bf a}, {\bf b})}{c^{1 + s}}  \sum_{\gamma | c}  \frac{1}{\gamma^s} \sum_{g | \frac{c}{\gamma}} \frac{\mu(g)}{g^s} \prod_{\substack{p| c \\ p \nmid bh\gamma}} \pr{1 - \frac 1p}.} %is the sum over $h, b, c, \gamma, g$, where $h | u_1u_2$, $(b, u_1u_2) = 1$
Next applying Lemma \ref{lem:sumT} to the sum over $\delta,$ and summing $\sumd_M \sumd_N$, we have that
\est{&\sumd_M \sumd_N \mathcal A (\boldab, M, N) \\
	&= 2\pi \frac{1}{2\pi i} \pr{\frac{q}{4\pi^2}}^{\delta(\al, \bb)}\int_{-\infty}^{\infty} \int_{(1)} \sumsharp \mathscr G_{ {\bf a}, {\bf b}}(1; h, b, c, \gamma, g) \pg{\sum_{\ell = 1}^{\infty} w \pr{\frac \ell L} - \int_0^{\infty} w \pr{\frac \ell L} \> d\ell}  \\
&\ \ \ \times   G\pr{\frac 12 + s; \al + it, \bb + it} H(s; \al, \bb)  \pr{\frac{a_2^3b_2^2}{a_1^3b_1^2}}^{it} (a_1^3b_1^2a_2^3b_2^2)^{-s} \frac{q^{3s}}{(4\pi^2)^{3s}} \\
& \ \ \ \times    \int_{0}^{\infty} \int_0^{\infty} \frac{y^{-\tfrac 12 - \alpha_1}}{y^{ s + it}} \frac{z^{-\tfrac 12 + \beta_1}}{z^{ s - it}} J_{k - 1} \pr{4\pi \sqrt{\frac{ a_1^2b_1y\ell}{a_2b_2cbhg\gamma q }}}  J_{k - 1} \pr{4\pi \sqrt{\frac{ a_2^2b_2z\ell}{a_1b_1cbhg\gamma q }}} \>dy \> dz \frac{ds}{s} \> dt .   }
The integration over $y$ and $z$ can be evaluated by Equation 707.14 in \cite{GR}, which is %on p. 33 Equation (4.6) of (book from Bodlein), or p.972 GR eq 707.14, which is
\est{ %\label{eqn:Besseltransform}
 \int_0^{\infty} v^{\mu} (vk)^{\frac 12} J_{\nu} (vk) \> dv = 2^{\mu + 1/2} k^{-\mu - 1} \frac{\Gamma\pr{\tfrac \mu 2 + \tfrac \nu 2 + \tfrac 34}}{\Gamma\pr{\tfrac \nu 2 - \tfrac \mu 2 + \tfrac 14}},}
for $- \textrm{Re} \nu - \frac 32 < \textrm{Re} \mu < 0.$ Then we apply Lemma \ref{lem:evalsumwtozeta} to the sum and the integration over $\ell.$ Therefore after summing over $a_1, a_2, b_1, b_2$,  we obtain that the main term of $R_{\alpha_1, \beta_1}$ is
\es{ \label{eqn:mainint}
& \frac{1}{2\pi i} \pr{\frac{q}{4\pi^2}}^{\delta(\pi(\al), \pi(\bb))}\int_{-\infty}^{\infty} \int_{(\eps)} G_{\alpha_1, \beta_1}\pr{\frac 12 + s; \al + it, \bb + it} H(s; \al, \bb) \mathcal M_{\alpha_1, \beta_1}(s) \\
& \hskip 1.5 in \times  \zeta(1 -\alpha_1 + \beta_1 - 2s)  \> \frac{q^{s}}{(4\pi^2)^{s}}\frac{ds}{s} \> dt,}
where $(\pi(\alpha), \pi(\beta)) = (\beta_1, \alpha_2, \alpha_3; \alpha_1, \beta_2, \beta_3)$, 
\est{G_{\alpha_i, \beta_j}\pr{\frac 12 + s; \al, \bb} =   \Gamma\pr{\frac k2 - s - \alpha_i}  \Gamma\pr{\frac k2 - s + \beta_j }   \prod_{\ell \neq i} \Gamma\pr{\frac k2 + s + \alpha_\ell } \prod_{\ell \neq j}\Gamma\pr{\frac k2 + s - \beta_\ell } ,}
and
\es{\label{def:Malbeta1} \mathcal M_{\alpha_1, \beta_1}(s) &:= \sumfour_{\substack{a_1, b_1, a_2, b_2 \geq 1 \\ (a_1a_2b_1b_2, q) = 1}}  \frac{ \mathscr B(a_1, b_1; \al)}{a_1^{2-2\alpha_1 - \beta_1 + 2s}b_1^{1-\alpha_1-\beta_1 + 2s}} \frac{ \mathscr B(a_2, b_2; -\bb)}{a_2^{2 + \alpha_1 + 2\beta_1 + 2s}b_2^{1 + \alpha_1 + \beta_1 + 2s}} \\ 
&\hskip 0.3in \times \sumsharp \mathscr G_{ {\bf a}, {\bf b}}(2s+\alpha_1 - \beta_1; h, b, c, \gamma, g).}

Now, in the ensuing discussion, we temporarily assume that $\textrm{Re}(\alpha_1) < \textrm{Re}(\alpha_2) , \textrm{Re}(\alpha_3) $ and $\textrm{Re}(\beta_1) > \textrm{Re}(\beta_2) , \textrm{Re}(\beta_3) $.  In this region, 
 \est{\mathcal R_1\left(\frac c\delta, \mathbf a, \mathbf b\right) 
 		&= \sumtwo_{\substack{n_2, n_3 \geq 1 \\ \frac{u_2c}{\delta(a_1, u_2c/\delta)} \ |\  n_2n_3}} \frac{1}{n_2^{1 + \alpha_2 - \alpha_1}n_3^{1 + \alpha_3 - \alpha_1}}}
and
\est{ \Res_{s_1 = 1 - \alpha_i} D_3\pr{s_1, \pm \frac{\lambda_1}{\eta_1}, \al} y^{s_1 - 1} &= \mathcal R_1\left(\frac c\delta, \mathbf a, \mathbf b\right) y^{-\alpha_1} = y^{-\alpha_1}\sumtwo_{\substack{n_2, n_3 \geq 1 \\ \frac{u_2c}{\delta(a_1, u_2c/\delta)} \ |\  n_2n_3}} \frac{1}{n_2^{1 + \alpha_2 - \alpha_1}n_3^{1 + \alpha_3 - \alpha_1}}.}
Similarly, 
\est{\Res_{s_2 = 1 + \beta_1} D_3\pr{s_2, \pm \frac{\lambda_2}{\eta_2}, -\bb} z^{s_2 - 1} = \mathcal R_2\left(\frac c\delta, \mathbf a, \mathbf b\right) z^{\beta_1} = z^{\beta_1}\sumtwo_{\substack{m_2, m_3 \geq 1 \\ \frac{u_1c}{\delta(a_2, u_1c/\delta)} \ |\  m_2m_3}} \frac{1}{m_2^{1 + \beta_1 - \beta_2}m_3^{1 + \beta_1 - \beta_3}} .}

Thus,
\es{\label{def:Gc} &\frac{1}{c^2} \mathcal G(c, {\bf a}, {\bf b}) = \sumtwo_{\substack{n_2, n_3 \geq 1 \\ \frac{u_2c}{(a_1, u_2c)} \ |\  n_2n_3}} \frac{1}{n_2^{1 + \alpha_2 - \alpha_1}n_3^{1 + \alpha_3 - \alpha_1}} \sumtwo_{\substack{m_2, m_3 \geq 1 \\ \frac{u_1c}{(a_2, u_1c)} \ |\  m_2m_3}} \frac{1}{m_2^{1 + \beta_1 - \beta_2}m_3^{1 + \beta_1 - \beta_3}}\\
&= \frac{(a_1, u_2c)(a_2, u_1c)}{u_1u_2c^2} \sum_{\substack{n = 1}}^\infty \frac{\sigma_2\pr{\frac{u_2cn}{(a_1,u_2c)}; \alpha_2 - \alpha_1, \alpha_3 - \alpha_1}}{n} \sum_{\substack{m = 1}}^\infty \frac{\sigma_2\pr{\frac{u_1cm}{(a_2, u_1c)}; \beta_1 - \beta_2, \beta_1 - \beta_3}}{m}.}

From this, we may then check that
\es{\label{def:Malbeta} \mathcal M_{\alpha_1, \beta_1}(s)
&= \prod_{j = 2}^3\zeta(1 + 2s + \alpha_j - \beta_1 ) \zeta(1 + 2s + \alpha_1 - \beta_j) \mathcal J_{\alpha_1, \beta_1}(s),}
where $\mathcal J_{\alpha_1, \beta_1}$ is absolutely convergent in the region Re$(s) = -1/4 + \eps.$  Although we have a priori only verified \eqref{def:Malbeta} for the region $\textrm{Re}(\alpha_1) < \textrm{Re}(\alpha_2) , \textrm{Re}(\alpha_3) $ and $\textrm{Re}(\beta_1) > \textrm{Re}(\beta_2) , \textrm{Re}(\beta_3) $, we see that \eqref{def:Malbeta} must hold for all values of $\alpha_i, \beta_j$ by analytic continuation.

We note that the pole of $\zeta(1-\alpha_1+\beta_1 - 2s)$ at $s = (\alpha_1 - \beta_1)/2$ and the poles of $\zeta(1 + 2s + \alpha_i - \beta_j)$ at $s = (\alpha_i - \beta_j)/2$ cancel with the zeros at the same point from $H(s; \al, \bb)$.  Thus, the integrand in \eqref{eqn:mainint} has only a simple pole at $s = 0$ and is analytic for all values of $s$ with $\tRe s > -1/4 + \epsilon$.  Moving the line of integration to Re$(s) = -1/4 + \eps,$ we then obtain the main term
\est{ &\zeta(1 -\alpha_1 + \beta_1 ) \mathcal M_{\alpha_1, \beta_1}(0)    \pg{\pr{\frac{q}{4\pi^2}}^{\delta(\pi(\al), \pi(\bb))} H(0; \al, \bb)  \int_{-\infty}^{\infty}  G\pr{\frac 12; \pi(\al) + it, \pi(\bb) + it}   \> dt},}
with negligible error term.  To finish the proof of Proposition \ref{prop:mainTM}, we will show that the local factor at prime $p$ of the Euler product of $\zeta(1 -\alpha_1 + \beta_1 )\mathcal M_{\alpha_1, \beta_1}(0)$ is the same as the one in $\mathcal A \mathcal Z\pr{\tfrac 12; \pi(\al), \pi(\bb)}$ defined in (\ref{def:mathcalZp}) and (\ref{def:mathcalAs}).  The details of this are in Appendix \ref{sec:Eulerverif}.

\section{Proof of Proposition \ref{prop:Terror} } \label{sec:properror}
To prove the proposition, it suffices to show that $\mathcal E_i^+(q; \al, \bb) \ll q^{-1/4 + \eps}$ for $i = 1$ and $i = 5$ since the proofs of upper bounds for other terms are similar. We start with a lemma that will be used in the proof.

\begin{lemma}  \label{lem:boundres} Let $\lambda, \eta$ be integers such that $(\lambda, \eta) = 1$ and $\lambda, \eta \ll q^{A}$, where $A$ is a fixed constant. Moreover, for $i,j = 1, 2, 3,$ $ \alpha_i \ll \frac 1{\log q},$ and $|\alpha_i - \alpha_j| \gg \frac{1}{q^{\eps_1}}$ when $i \neq j.$ Then for $\eps > \eps_1,$ 
\est{\Res_{s = 1 - \alpha_i} D_3\pr{s, \frac{\lambda}{\eta}, \al} \ll \frac{q^{\eps}}{\eta},}
where $D_3\pr{s,\frac \lambda\eta, \al }$ is defined in (\ref{eqn:OriginalD}).
\end{lemma}
\begin{proof} By symmetry, it suffices to prove the statement for the residue at $1- \alpha_1.$ For $\tRe(s) > 1 + \tRe(\alpha_1 - \alpha_j)$, where $j = 2, 3$, let  
\est{D(s) &:= \sum_{\substack{m = 1 \\ \eta| m}}^{\infty} \frac{\sigma_2(m; \alpha_2 - \alpha_1, \alpha_3 - \alpha_1)}{m^s} \\ 
&= \frac{1}{\eta^s}\sum_{d |\eta} \frac{\mu(d)\sigma_2\pr{\frac{\eta}{d}; \alpha_2 - \alpha_1, \alpha_3 - \alpha_1}}{d^{s + \alpha_2 + \alpha_3 - 2\alpha_1}} \zeta( s + \alpha_2 - \alpha_1)\zeta(s + \alpha_3 - \alpha_1),}
where we have used Lemma \ref{lem:multofsigma_k} to derive the last line.  Now, $D(s)$ can be continued analytically to the whole complex plane except for poles at $s = 1 + \alpha_1 - \alpha_j$ for $j = 2, 3.$  Moreover, $D(1) \ll \frac{q^{\eps}}{\eta}.$  

For $i = 1, 2, 3,$ and $\tRe(s + \alpha_i) > 1$, the sum in the Lemma can be rewritten as
\est{ 
& \frac{1}{\eta^{3s + \alpha_1 + \alpha_2 + \alpha_3}} \sumthree_{r_1, r_2, r_3 \mod \eta}\e{\frac{\lambda r_1 r_2r_3}{\eta}}\zeta\pr{s+ \alpha_1; \frac{r_1}{\eta}}\zeta\pr{s+ \alpha_2; \frac{r_2}{\eta}}\zeta\pr{s+ \alpha_3; \frac{r_3}{\eta}}. }

This sum can be analytically continued to the whole complex plane except for poles at $s = 1 - \alpha_i$ for $i = 1, 2, 3.$ After some arrangement, the contribution of the residue at $s = 1 - \alpha_1$ is
\est{ %\label{res1minusalp}&\frac{1}{\eta^{3 + \alpha_2 + \alpha_3 - 2\alpha_1}} \sumthree_{r_1, r_2, r_3 \mod \eta}\e{\frac{\lambda r_1 r_2r_3}{\eta}}\zeta\pr{1 + \alpha_2 - \alpha_1; \frac{r_2}{\eta}}\zeta\pr{1 + \alpha_3 - \alpha_1; \frac{r_3}{\eta}} \\
&\frac{1}{\eta^{2 + \alpha_2 + \alpha_3 - 2\alpha_1}} \sumtwo_{\substack{r_2, r_3 \mod \eta\\ \eta | r_2r_3}}\zeta\pr{1 + \alpha_2 - \alpha_1; \frac{r_2}{\eta}}\zeta\pr{1 + \alpha_3 - \alpha_1; \frac{r_3}{\eta}} = D(1), }
and the Lemma follows.

\end{proof}

\subsection{Bounding $\mathcal E_1^+(q; \al, \bb)$} \label{sec:proofE1}%Now we consider the contribution from $\mathcal T_2 (c, x). $
With the same notation as in Section \ref{sec:provepropTM} and $\mathscr B$ defined as in \eqref{eqn:defB}, 
we recall that
\est{
	\mathcal E_{1, \al, \bb}^+(q; \al, \bb) = \frac {2\pi}{q}&\sumfour_{\substack{a_1, b_1, a_2, b_2 \geq 1 \\ (a_1a_2b_1b_2, q) = 1}} \frac{ \mathscr B(a_1, b_1; \al)}{a_1^{\frac 32}b_1} \frac{ \mathscr B(a_2, b_2; -\bb)}{a_2^{\frac 32}b_2} \sumd_M \sumd_N  E_{1, \al, \bb}^+(\mathbf{a}, \mathbf{b}, M, N) ,}
where 
\est{E_{1, \al, \bb}^+(\mathbf{a}, \mathbf{b}, M, N) := \sum_{c < C} \sumstar_{x \bmod \delta c} \frac{\mathcal T_{1, \al, \bb}^+(c,x)}{c} = \sum_{\delta < C} \frac 1\delta \sum_{c < \frac{C}{\delta}} \sumstar_{\substack{x \bmod \delta c \\ (u_1x-u_2, c\delta) = \delta}} \frac{\mathcal T^+_{1, \al, \bb}(c\delta, x)}{c},}
\est{ \mathcal T_{1, \al, \bb}^+ (c\delta,x) &:= \frac{\pi^{3/2 + \alpha_1 + \alpha_2 + \alpha_3}}{\eta_1^{3+\alpha_1 + \alpha_2 + \alpha_3}} \sum_{i = 1}^3 \Res_{s = 1 + \beta_i} D_3\pr{s, \frac{\lambda_2}{\eta_2}, -\bb} \sum_{n = 1}^{\infty} A_3\pr{n,  \frac{\lambda_1}{\eta_1}, \al} \mathcal F_1^+(c\delta,n; \al, \bb),}
for $\lambda_1 = \frac{\bar q (u_2\bar x - u_1)a_1}{\delta(a_1, u_2c)},$ $\eta_1 = \frac{u_2 c}{(a_1, u_2c)},$  $\lambda_2 = \frac{\bar q (u_1 x - u_2)a_2}{\delta(a_2, u_1c)},$ $\eta_2 = \frac{u_1 c}{(a_2, u_1c)},$ $u_i = \frac{a_ib_i}{(a_1b_1, a_2b_2)}$ and
\est{\mathcal F_1^+(c\delta,n; \al, \bb) &= \int_{0}^{\infty} \int_0^{\infty} \frac{1}{y^{1/2 }} \frac{z^{s - 1}}{z^{1/2}} V_{\al, \bb} \pr{a_1^3b_1^2y, a_2^3b_2^2z; q} J_{k-1} \pr{\frac{4\pi \sqrt{a_2ya_1z}}{c\delta q}} f\pr{\frac yN} f\pr{\frac zM} \times \\
	&\hskip 0.5in \times  i^{- k}\e{\frac{a_1^2b_1y}{c\delta qa_2b_2} + \frac{a_2^2b_2z}{c\delta qa_1b_1}}  U_3\pr{\frac{\pi^3ny}{\eta_1^3}; \al} \>dy \> dz.  }
We first note that the contribution from the terms $a_1^3b_1^2y \gg q^{3/2 + \eps}$ or $a_2^3b_2^2 z \gg q^{3/2 + \eps}$ can be bounded by $q^{-A}$ for any $A$ due to the factor $V_{\alpha, \beta}(a_1^3b_1^2y, a_2^3b_2^2 z; q).$  So from now on we assume $a_1^3b_1^2N \ll q^{3/2 + \eps}$ and $a_2^3b_2^2 M \ll q^{3/2 + \eps}.$  

Moreover, the dyadic sum over $M$ and $N$ contains only $\ll \log^2q$ terms, so it suffices to prove that
\begin{equation}
E_{1, \al, \bb}^+(\mathbf{a}, \mathbf{b}, M, N)\ll a_1^{1/2}q^{3/4+\eps},
\end{equation}for fixed $\mathbf{a}, \mathbf{b}, M, N $ satisfying $a_1^3b_1^2N \ll q^{3/2 + \eps}$ and $a_2^3b_2^2 M \ll q^{3/2 + \eps}$.  On a first reading, the reader may set $a_1 = a_2 = b_1 = b_2 = 1$ as this simplifies the notation without substantially changing the calculation.

We now write
\begin{equation}
E_{1, \al, \bb}^+(\mathbf{a}, \mathbf{b}, M, N)  = H_1+ H_2,
\end{equation}
where $H_1$ is the contribution from the sum over $n \le \frac{\eta_1^3}{N} q^\eps$, and $H_2$ is the rest. 

\subsubsection{Bounding $H_1$} \label{sec:Esm}

By (\ref{lem:asympUVxsmall}), $ U_3 \pr{\frac{\pi^3 n y}{\eta_1^3}} \ll q^{\eps}$ when $n \ll \frac{\eta_1^3 q^\eps }{N} $. This and (\ref{bound:Bessel1}) gives us that
$$\mathcal F_1^+(c\delta,n; \al, \bb) \ll M^{1/2}N^{1/2}q^\eps \min\pg{\pr{\frac{\sqrt{a_1a_2MN}}{c\delta q}}^{-\frac 12},\pr{\frac{\sqrt{a_1a_2MN}}{c\delta q}}^{k-1}}.
$$

Then, from Lemma \ref{lem:boundres}, Lemma \ref{lem:A3B3Kloos}, (\ref{def:Ak}), and using the fact that $(a_2, u_1c) \leq a_2,$ and $\frac{1}{(a_1, u_2c)} \leq 1$,  we obtain that $H_1 $ is bounded by
\est{& M^{\frac 12}N^{\frac 12} q^{\eps} \sum_{\delta < C} \sum_{c < \frac C\delta} \frac{1}{\eta_1^3 \eta_2}\sum_{n \ll \frac{\eta_1^3q^\eps}{N}} \sum_{\substack{h | \eta_1, \ h | n^2}} \eta_1^{\frac 32}\sqrt h \min\pg{\pr{\frac{\sqrt{a_1a_2MN}}{c\delta q}}^{-\frac 12},\pr{\frac{\sqrt{a_1a_2MN}}{c\delta q}}^{k-1}} \\
&\ll M^{\frac 54} N^{\frac 14} q^{\eps} \frac{a_1^{\frac 34} a_2^{\frac 74}u_2^{\frac 32}}{u_1q^{\frac 32}}  \ll q^{3/4+\eps},
 }
as desired.  In the above, we have used $(a_1b_1, a_2b_2) \geq 1$, and $a_1^3b_1^2N \ll q^{3/2 + \eps}$ and $a_2^3b_2^2 M \ll q^{3/2 + \eps}$.

\subsubsection{Bounding $H_2$} \label{sec:proofE1b}
We start from re-writing $\mathcal F_1^+(c\delta, n; \al, \bb)$ as 
\est{ 
&\int_{-\infty}^{\infty} \int_{(\frac 1{\log q})}  V_1(s, t) \int_{0}^{\infty} z^{\beta_1}  F_{s - it}\pr{\frac zM} \e{\frac{a_2^2b_2z}{c\delta qa_1b_1}}\mathcal I(n, z) \> dz \> \frac{ds}{s} \> dt,}
where $V_1(s,t) := V_1({\bf a}, {\bf b}, \al, \bb, s, t, M, N) $ is defined as
\es{\label{def:V1st} \frac{1}{2\pi i} \pr{\frac{q}{4\pi^2}}^{\delta(\al, \bb)} \pr{ \frac{a_2^3b_2^2 }{a_1^3b_1^2}}^{it} G\pr{\frac 12 + s; \al + it, \bb + it} H(s; \al, \bb) \pr{\frac{q}{4\pi^2}}^{3s} M^{-\frac 12 - s + it} N^{-\frac 12 - s - it},}
and $\mathcal I(n, z) := \mathcal I_{\al}(\mathbf{a}, \mathbf{b}, N, n, z, c, \delta)$ is defined as
\es{\label{def:IabN}  \int_0^{\infty}F_{s + it}\pr{\frac yN} \e{\frac{a_1^2b_1y}{c\delta qa_2b_2} }  J_{k-1} \pr{\frac{4\pi \sqrt{a_2ya_1z}}{c\delta q}} U_3\pr{\frac{\pi^3ny}{\eta_1^3}; \al} \>dy,}
and
$F_{v} \pr{x} := \frac{1}{x^{\frac 12 + v}} f\pr{x}.$ Note that the $j$-th derivative, $F_{s \pm it}^{(j)}\pr{\frac{y}{N}} = O(|t|^j N^{-j}). $

Note that the trivial bound for $\mathcal F_1^+ (c\delta, n; \al, \bb)$ is 
\begin{equation}\label{eqn:Fbdd}
\mathcal F_1^+ (c\delta, n; \al, \bb) \ll \sqrt{MN} (qn)^\eps.
\end{equation}

There are two cases to consider: (1) $c > \frac{8\pi\sqrt{a_1a_2MN}}{\delta q}$, and (2)  $c \leq  \frac{8\pi\sqrt{a_1a_2MN}}{\delta q}$.
\\
\\
{\bf Case 1: $c > \frac{8\pi\sqrt{a_1a_2MN}}{\delta q}$.}

By (\ref{asympJxSm}), (\ref{lem:asymptUxbig}), and since $\frac{\pi^3 ny}{\eta_1^3} \gg q^\eps$, $\mathcal I(n, z)$ can be written as

\est{ %&\sum_{\ell = 0}^{\infty} \frac{(-1)^{\ell}}{\ell ! (\ell + k - 1)!} \int_0^{\infty}  G_3(y, z, \ell) \e{\frac{a_1^2b_1y}{c\delta qa_2b_2} } U_3\pr{\frac{\pi^3ny}{\eta_1^3}} \>dy \\
	& \sum_{j = 1}^K \bfrac{\pi^3 ny}{\eta_1^3}^{\frac{\beta_1 + \beta_2 + \beta_3}{3}} \pr{\frac{\eta_1}{\pi n^{\frac 13}N^{\frac 13}}}^{j} \sum_{\ell = 0}^{\infty} \frac{(-1)^{\ell}}{\ell ! (\ell + k - 1)!} \int_0^{\infty}   \mathscr F_j(y, z,\ell) \e{\frac{a_1^2b_1y}{c\delta qa_2b_2} }  \times \\
	& \hskip 2in \times \pg{ c_j \e{\frac{3n^{\frac 13}y^{\frac 13}}{\eta_1}} + d_j \e{-\frac{3n^{\frac 13}y^{\frac 13}}{\eta_1}} }  \>dy + O\pr{q^{-\eps (K + 1)}}, }
where $c_j, d_j$ are some constants, and 
\est{ %G_3(y, z,\ell) := F_{s + it}\pr{\frac yN}  \pr{\frac{2\pi \sqrt{a_2ya_1z}}{c\delta q}}^{2\ell + k -1}
	\mathscr F_j(y, z,\ell) = F_{s + it}\pr{\frac yN}  \pr{\frac{2\pi \sqrt{a_2ya_1z}}{c\delta q}}^{2\ell + k -1} \pr{\frac{y}{N}}^{\frac{-j}{3}} }
is supported on $y\in [N, 2N]$.  Moreover,  $\frac{\partial^i \mathscr F_j(y, z, \ell)}{\partial y^i} \ll \frac{1}{2^{2\ell}}\frac {|t|^i}{N^i}$ and $\mathscr F_j(y,z, \ell) \ll 1$. 
%Firstly, 
%we consider 
%\est{\int_0^{\infty} &\mathcal F_j(y, z, \ell) \e{\frac{3n^{\frac 13}y^{\frac 13}}{\eta_1} + \frac{a_1^2b_1y}{c\delta qa_2b_2}} \>dy.}
%Each integration of the function inside the exponential brings in a factor of order $\frac{n^{\frac 13}}{N^{\frac 23} \eta_1} + \frac{a_1^2b_1}{c\delta qa_2b_2} \geq \frac{n^{\frac 13}}{N^{\frac 23} \eta_1}.$ When $n \gg \frac{\eta_1^3q^{\eps}}{N}$ and $i$ is large enough, after integration by parts $i$ times, the contribution of these terms is negligible. 
Thus, picking $K$ large enough so that $q^{-\eps(K+1)}$ is negligible, it suffices to bound integrals of the form 
\est{\int_0^{\infty} &\mathscr F_j(y, z, \ell) \e{\theta_z(y, n)} \>dy,}
where
$$\theta_z(y, n) = \pm \frac{3n^{\frac 13}y^{\frac 13}}{\eta_1} + By , \ \ \ \ \textrm{and} \ \ \ B = \frac{a_1^2b_1}{c\delta q a_2b_2}.$$
%From properties of $\mathscr F_j(y, z, \ell)$, we obtain that  
%\es{\label{intbound:Fj} \int_0^{\infty} \mathscr F_j(y, z, \ell) \e{\theta_z(y, n)} \>dy \ll N. }
Taking the derivative of $\theta_z(y, n)$ with respect to $y$, we have that $$\theta'_z(y, n) = B \pm \frac{n^{\frac 13}}{y^{\frac 23}\eta_1} .$$
When $n \geq 64(B\eta_1)^3N^{2}$ or $n \leq  \frac{1}{4}{(B\eta_1)^3N^{2}},$  $|\theta_z'(y, n)| \gg \frac{n^{\frac 13}}{y^{\frac 23}\eta_1} \gg \frac{q^{\eps}}{N},$ since $n \gg \frac{\eta_1^3 q^\eps}{N}$.  Thus integrating by parts many times shows that the contribution from these terms is negligible. Therefore we only consider the contribution from when $\frac 1{4}(B\eta_1)^3N^{2} \leq n \leq 64(B\eta_1)^3N^{2}.$  Note however that
$$(B\eta_1)^3 N^2 \ll \frac{(a_1^2 b_1)^3N^2}{\delta^3 q^3} \ll \frac{q^\eps}{\delta^3},
$$  and that there are no terms of this form unless $N \gg \frac{q^{3/2}}{(a_1^2 b_1)^{3/2}}$ and $\delta \ll q^\eps$.
From (\ref{eqn:Fbdd}), trivially $\mathcal F_1^+(c\delta, n; \al, \bb) = O(M^{\frac 12} N^{\frac 12}(nq)^\eps).$   Hence the contribution to $H_2$ from these terms is bounded by
\est{& M^{\frac 12}N^{\frac 12}q^\eps \sum_{\delta < q^\eps} \sum_{c  \gg \frac{\sqrt{a_1a_2MN}}{\delta q}} \frac{1}{\eta_1^{\frac 32}\eta_2} \sum_{h | \eta_1} \sqrt h  \sum_{\substack{n \ll q^\eps \\ h^2 | n}} \frac{\eta_1}{ n^{\frac 13}N^{\frac 13}} \\ 
&\ll a_1^{\frac 54}b_1^{\frac 12} a_2^{\frac 34} M^{\frac 14}N^{\frac 14}  q^{\eps} \ll a_1^{\frac 12}q^{\frac 34+\eps},  }
similar to before.
\\
\\
{\bf Case 2: $c \leq \frac{8\pi\sqrt{a_1a_2MN}}{\delta q}$.}

By (\ref{asympJxbig}), we write $\mathcal I(n, z) $ as

\est{  \int_0^{\infty}  \pg{R^+(y, z)  + R^-(y, z)} \frac{\sqrt{c\delta q}}{\pi \pr{a_1a_2yz}^{\frac 14} }U_3\pr{\frac{\pi^3ny}{\eta_1^3}; \al} \e{\frac{a_1^2b_1y}{c\delta qa_2b_2} } \>dy , }
where 
\est{R^\pm(y, z) := F_{s + it}\pr{\frac yN}  W^\pm\pr{\frac{4\pi\sqrt{a_1a_2yz}}{c\delta q}}{    \e{\pm \pr{\frac{2\sqrt{a_1a_2yz}}{c\delta q} - \frac k4 + \frac 18}}},}
%\est{\mathcal H_{s+it}^\pm(y, z) := F_{s + it}\pr{\frac yN}  W^\pm\pr{\frac{4\pi\sqrt{a_1a_2yz}}{c\delta q}},}
and $W^+ = W$, $W^- = \overline{ W}.$

Similar to Case 1, we explicitly write $U_3\pr{\frac{\pi^3ny}{\eta_1^3}; \al}$ as in  Equation (\ref{lem:asymptUxbig}) so it suffices to bound

\est{\sum_{j = 1}^K  \bfrac{\pi^3 ny}{\eta_1^3}^{\frac{\beta_1 + \beta_2 + \beta_3}{3}} &\pr{\frac{\eta_1}{\pi n^{\frac 13}N^{\frac 13}}}^{j}\frac{1}{N^{\frac 14}}\frac{\sqrt{c\delta q}}{\pi \pr{a_1a_2z}^{\frac 14} }\int_0^{\infty} \mathscr H_j^\pm(y, z)\pg{ c_j \e{\frac{3n^{\frac 13}y^{\frac 13}}{\eta_1}} + d_j \e{-\frac{3n^{\frac 13}y^{\frac 13}}{\eta_1}} } \times \\
&\times \e{\frac{a_1^2b_1y}{c\delta qa_2b_2} }\e{\pm\pr{{\frac{2\sqrt{a_1a_2yz}}{c\delta q} - \frac k4 + \frac 18}}} \>dy + O\pr{q^{-\eps (K + 1)}},
}
where $\mathscr H_j^\pm(y, z) = \frac{F_{s + it}\pr{\frac yN}  W^\pm\pr{\frac{4\pi\sqrt{a_1a_2yz}}{c\delta q}}}{\pr{\frac yN}^{\frac j3 + \frac 14} }$ is supported on $y \in [N, 2N]$.  Note that $\frac{\partial^{(i)} \mathscr H_j^\pm(y,z) }{\partial y^i} \ll_{j,i} \frac {|t|^i}{N^i}$ and $\mathscr H_j^\pm(y,z) \ll 1$. Thus, the integration over $y$ is of the form
\est{\int_0^{\infty} &\mathscr H^\pm_j(y, z) \e{g_z(y, n)} \>dy,}
where 
$$g_z(y, n) = \pm \frac{3n^{\frac 13}y^{\frac 13}}{\eta_1} \pm \pr{2A\sqrt y + \frac k4 - \frac 18} + By , \ \ \ \ \ \ A = \frac{\sqrt{a_1a_2z}}{c\delta q}, \ \  B = \frac{a_1^2b_1}{c\delta q a_2b_2}.$$
%From the properites of $\mathscr H_j^\pm (y,z)$, we have that
%\es{\label{eqn:boundHjint}\int_0^{\infty} &\mathscr H^\pm_j(y, z) \e{g_z(y, n)} \>dy \ll N.}
Differentiating $g_z(y, n)$ with respect to $y$, we have $$g'_z(y, n) = \pm \frac{n^{\frac 13}}{y^{\frac 23}\eta_1}  \pm \frac{A}{y^{\frac 12}} + B.$$

When $a_2^{\frac 32}b_2M^{\frac 12} \geq 4a_1^{\frac 32}b_1N^{\frac 12}, $ it follows that $\frac A{y^{\frac 12}} \geq \frac A{y^{\frac 12}} - B \geq \frac 12 \frac A{y^{\frac 12}} $ and $\frac 32 \frac A{y^{\frac 12}} \geq \frac A{y^{\frac 12}} + B \geq \frac A{y^{\frac 12}}.$  Therefore 
$$ \frac{1}{2} \frac{A}{y^{\frac 12}} \leq \left|  \pm \frac{A}{y^{\frac 12}} + B \right| \leq \frac{3}{2} \frac{A}{y^{\frac 12}}.$$

When $n \geq 54(A\eta_1)^3N^{\frac 12}$ or $n \leq  \frac{1}{64}(A\eta_1)^3N^{\frac 12},$ we have that $|g_z'(y, n)| \gg \frac{n^{\frac 13}}{y^{\frac 23}\eta_1} \gg \frac{q^{\eps}}{N},$ since $n \gg \frac{\eta_1^3 q^\eps}{N}$.
Integrating by parts many times shows that these terms are negligible. We then consider only the terms when $\frac {1}{64}(A\eta_1)^3N^{\frac 12} \leq n \leq 54(A\eta_1)^3N^{\frac 12}.$  Note however that
$$(A\eta_1)^3 N^{1/2} \ll \bfrac{\sqrt{a_1a_2M} u_2c}{c\delta q}^3 N^{1/2}
\ll \bfrac{\sqrt{a_1}}{\delta q^{\frac 14}}^3 N^{1/2} \ll \frac{q^{\eps}}{\delta^3},
$$and that the left side is only $\gg 1$ if $N \gg q^{3/2}/a_1^3$ and $\delta \ll q^\eps$.

By (\ref{eqn:Fbdd}), the contribution of $\mathcal F_1^+ (c\delta, n; \al, \bb)$ to the terms in this range is $O(M^{\frac 12} N^{\frac 12}(nq)^\eps).$ So the contribution to $H_2$ from these terms is bounded by
\est{& M^{\frac 12}N^{\frac 12} q^\eps\sum_{\delta < q^\eps} \sum_{c  \ll \frac{\sqrt{a_1a_2MN}}{\delta q}} \frac{1}{\eta_1^{\frac 32}\eta_2} \sum_{h | \eta_1} \sqrt h  \sum_{\substack{n \ll q^\eps \\ h^2 | n}} \pr{\frac{\eta_1}{ n^{\frac 13}N^{\frac 13}}} \frac{\sqrt{c\delta q}}{\pr{a_1a_2MN}^{\frac 14} } \\
& \ll a_1^{\frac 54}b_1^{\frac 12} a_2^{\frac 34} M^{\frac 14}N^{\frac 14}  q^{\eps} \ll a_1^{\frac 12}q^{\frac 34+\eps} }
which suffices.

When $4 a_2^{\frac 32}b_2M^{\frac 12} \leq  a_1^{\frac 32}b_1N^{\frac 12}, $ we have that $ \frac 12 B <  B - \frac A{y^{\frac 12}}   <  B $ and $\frac 32 B >  A{y^{\frac 12}} + B >  B.$ By the same arguments as in Case 1, the range of $n$ that should be considered is of the size $(B\eta_1)^3N^{2}$ and give a contribution to $H_2$ bounded by $a_1^{1/2}q^{3/4 + \eps}.$ 
 
When $\frac 14 a_1^{\frac 32}b_1N^{\frac 12} <  a_2^{\frac 32}b_2M^{\frac 12} <  4 a_1^{\frac 32}b_1N^{\frac 12}, $ we have that $\frac A{y^{\frac 12}} \asymp B,$ and so the range of $n$ that should be considered is of the size $(A\eta_1)^3N^{\frac 12}$ by the same arguments as above. Hence the contribution from these terms to $H_2$ is  $  O(a_1^{1/4}q^{\frac 34 + \eps}).$

This completes the proof of Proposition \ref{prop:Terror} for $\mathcal E_1^+(q; \al, \bb)$. The same proof applies to bound $\mathcal E_{i}^\pm(q; \al, \bb)$ for $i = 2, 3, 4.$

\subsection{Bounding $\mathcal E^+_5(q; \al, \bb)$} \label{sec:proofE5} We first recall that 
\est{
	\mathcal E_5^+(q; \al, \bb) = \frac {2\pi}{q}&\sumfour_{\substack{a_1, b_1, a_2, b_2 \geq 1\\ (a_1a_2b_1b_2, q) = 1}} \frac{ \mathscr B(a_1, b_1; \al)}{a_1^{\frac 32}b_1} \frac{ \mathscr B(a_2, b_2; -\bb)}{a_2^{\frac 32}b_2} \sumd_M \sumd_N  E_{5, \al, \bb}^+(\mathbf{a}, \mathbf{b}, M, N) ,}
where 
\est{E_{5, \al, \bb}^+(\mathbf{a}, \mathbf{b}, M, N) := \sum_{c < C} \sumstar_{x \bmod \delta c} \frac{\mathcal T_{5, \al, \bb}^+(c,x)}{c} = \sum_{\delta < C} \frac 1\delta \sum_{c < \frac{C}{\delta}} \sumstar_{\substack{x \bmod \delta c \\ (u_1x-u_2, c\delta) = \delta}} \frac{\mathcal T^+_{5, \al, \bb}(c\delta, x)}{c};}

\est{ \mathcal T_{5, \al, \bb}^+(c\delta,x) &:= \frac{\pi^{3 + \sum_{i = 1}^3 (\alpha_i - \beta_i)}}{\eta_1^{3+\sum_{i=1}^3 \alpha_i}\eta_2^{3 - \sum_{i=1}^3 \beta_i}}   \sumtwo_{n,m \geq 1} A_3\pr{n,  \frac{\lambda_1}{\eta_1}, \al}  A_3\pr{m,  \frac{\lambda_2}{\eta_2}, -\bb} \mathcal F_5^+(c\delta,n, m; \al, \bb);}
for $\lambda_1 = \frac{\bar q (u_2\bar x - u_1)a_1}{\delta(a_1, u_2c)},$ $\eta_1 = \frac{u_2 c}{(a_1, u_2c)},$  $\lambda_2 = \frac{\bar q (u_1 x - u_2)a_2}{\delta(a_2, u_1c)},$ $\eta_2 = \frac{u_1 c}{(a_2, u_1c)},$  $u_i = \frac{a_ib_i}{(a_1b_1, a_2b_2)}$, and
\est{\mathcal F_5^+(c\delta,n, m, \al, \bb) &= \int_{0}^{\infty} \int_0^{\infty} \frac{1}{y^{1/2 }} \frac{1}{z^{1/2}} V_{\al, \bb} \pr{a_1^3b_1^2y, a_2^3b_2^2z; q} J_{k-1} \pr{\frac{4\pi \sqrt{a_2ya_1z}}{c\delta q}} f\pr{\frac yN} f\pr{\frac zM} \times \\
	&\hskip 0.5in \times  i^{- k}\e{\frac{a_1^2b_1y}{c\delta qa_2b_2} + \frac{a_2^2b_2z}{c\delta qa_1b_1}}  U_3\pr{\frac{\pi^3ny}{\eta_1^3}; \al} U_3\pr{\frac{\pi^3mz}{\eta_2^3}; - \bb}  \>dy \> dz.  }
The proofs in this section are very similar to the ones in the previous section.  Previously, we had one sum over $n$ and now we have a double sum over $m$ and $n$ which can be treated in a similar manner.  To be precise, we begin by dividing $E_{5, \al, \bb}^+(\mathbf{a}, \mathbf{b}, M, N)$ into $ \sum_{i = 1}^4 E^+_{5, i}(\mathbf{a}, \mathbf{b}, M, N),$ where $E^+_{5, i}(\mathbf{a}, \mathbf{b}, M, N) := E^+_{5, i, \al, \bb}(\mathbf{a}, \mathbf{b}, M, N)$ is the contribution from case $i$  below.  
\begin{enumerate}
\item $n \ll \frac{\eta_1^3q^{\eps}}{N}$ and $m \ll \frac{\eta_2^3q^{\eps}}{M};$
\item $n \gg \frac{\eta_1^3q^{\eps}}{N}$ and $m \ll \frac{\eta_2^3q^{\eps}}{M};$
\item $n \ll \frac{\eta_1^3q^{\eps}}{N}$ and $m \gg \frac{\eta_2^3q^{\eps}}{M};$
\item $n \gg \frac{\eta_1^3q^{\eps}}{N}$ and $m \gg \frac{\eta_2^3q^{\eps}}{M}.$
\end{enumerate}

By symmetry, the treatment for cases (2) and (3) is the same, so we will show only the second case. 

Similar to Section \ref{sec:proofE1}, the contribution from the terms $a_1^3b_1^2y \gg q^{3/2 + \eps}$ or $a_2^3b_2^2 z \gg q^{3/2 + \eps}$ can be bounded by $q^{-A}$ due to the factor $V_{\al, \bb}(a_1^3b_1^2y, a_2^3b_2^2 z; q).$  Thus it suffices to prove that  

%Moreover, the dyadic sum over $M$ and $N$ contains only $\ll \log^2q$ terms, so it suffices to prove that
\begin{equation} \label{eqn:E5ibound}
E_{5, i}^+(\mathbf{a}, \mathbf{b}, M, N)\ll a_1^{1/2}a_2^{1/2}q^{3/4+\eps},
\end{equation}
for fixed $\mathbf{a}, \mathbf{b}, M, N $ satisfying $a_1^3b_1^2N \ll q^{3/2 + \eps}$ and $a_2^3b_2^2 M \ll q^{3/2 + \eps}$.  In fact, we will prove the stronger bound $E_{5, i}^+(\mathbf{a}, \mathbf{b}, M, N)\ll a_1^{1/2}a_2^{1/2}q^{1/2+\eps}$.

\subsubsection{Bounding $E^+_{5,1}({\bf a}, {\bf b}, M, N)$} \label{sec:boundT3mnsmall}
For this case, $ U_3 \pr{\frac{\pi^3 n y}{\eta_1^3}; \al} \ll q^{\eps},$ and $ U_3 \pr{\frac{\pi^3 mz}{\eta_2^3}; -\bb } \ll q^{\eps}$  by  (\ref{lem:asympUVxsmall}).  Similar to the arguments in Section \ref{sec:Esm}, from  Lemma \ref{lem:boundres}, Lemma \ref{lem:A3B3Kloos} and (\ref{bound:Bessel1}), we have that for $k \geq 5$, $E^+_{5,1}({\bf a}, {\bf b}, M, N)$ is bounded by

\est{&\ll M^{\frac 12}N^{\frac 12} q^{\eps} \sum_{\delta < C} \sum_{c < \frac C\delta} \frac{1}{\eta_1^3 \eta_2^3}\sum_{n \ll \frac{\eta_1^3q^\eps}{N}} \sum_{\substack{h_1 | \eta_1, \ h_1 | n^2}} \eta_1^{\frac 32}\sqrt h_1 \sum_{m \ll \frac{\eta_2^3q^\eps}{M}} \sum_{\substack{h_2 | \eta_2, \ h_2 | m^2}} \eta_2^{\frac 32}\sqrt h_2  \\
&\hskip 2in \times  \min\pg{\pr{\frac{\sqrt{a_1a_2MN}}{c\delta q}}^{-\frac 12},\pr{\frac{\sqrt{a_1a_2MN}}{c\delta q}}^{k-1}} \\
%&\ll M^{\frac 12}N^{\frac 12} q^{\eps} \sum_{\delta < C} \sum_{c < \frac C\delta} \frac{1}{\eta_1^{\frac 32} \eta_2}\sum_{\substack{h | \eta_1}} \sqrt h  \sum_{\substack{n \ll \frac{\eta_1^3q^\eps}{N} \\ h^2| n}} \min\pg{\pr{\frac{\sqrt{a_1a_2MN}}{c\delta q}}^{-\frac 12},\pr{\frac{\sqrt{a_1a_2MN}}{c\delta q}}^{k-1}} \\
&\ll M^{-\frac 12}N^{-\frac 12}q^{\eps}  \sum_{\delta < C} \pg{\sum_{\frac{\sqrt{a_1a_2MN}}{q\delta} \ll c < \frac C\delta} \eta_1^{\frac 32} \eta_2^{\frac 32} \pr{\frac{\sqrt{a_1a_2MN}}{c\delta q}}^{k-1} + \sum_{c \ll \frac{\sqrt{a_1a_2MN}}{q\delta}} \eta_1^{\frac 32} \eta_2^{\frac 32} \pr{\frac{\sqrt{a_1a_2MN}}{c\delta q}}^{-\frac 12} }\\
&\ll M^{\frac 32} N^{\frac 32} q^{\eps} \frac{a_1^{2} a_2^{2}u_1^{\frac 32}u_2^{\frac 32}}{q^{4}} \ll  q^{\frac 12 + \eps}.
 }

%$$ E^+_{5,1}({\bf a}, {\bf b}, M, N) \ll a_1^{1/2} a_2^{1/2} q^{\frac 12 + \eps}.$$

\subsubsection{Bounding $E^+_{5,2}({\bf a}, {\bf b}, M, N)$}

We can write $\mathcal F_5^+(c\delta, n, m; \al, \bb)$ as 
\est{ 
& \int_{-\infty}^{\infty} \int_{(\frac 1{\log q})} V_1(s,t) \int_{0}^{\infty}  F_{s - it}\pr{\frac zM} \e{\frac{a_2^2b_2z}{c\delta qa_1b_1}} U_3\pr{\frac{\pi^3mz}{\eta_2^3}; -\bb} \mathcal I(n,z) \> dz \> \frac{ds}{s} \> dt,}
where $V_1(s,t)$ and $\mathcal I(n, z)$ are defined as in (\ref{def:V1st}) and (\ref{def:IabN}), respectively, and $F_{v} \pr{x} = \frac{1}{x^{\frac 12 + v}} f\pr{x}.$ Note that $F_{s \pm it}^{(j)}\pr{\frac{y}{N}} \ll |t|^j N^{-j}. $

The integration over $z$ can be bounded trivially, and the sum over $m, h_2$ can be treated in the same way as in Section \ref{sec:boundT3mnsmall}. For the integration over $y$, we argue as in Case 1 and 2 of Section \ref{sec:proofE1b} and obtain that $E^+_{5,2}({\bf a}, {\bf b}, M, N) \ll a_1^{\frac 12}q^{\frac 12 + \eps}.$

\subsubsection{Bounding $E^+_{5,4}({\bf a}, {\bf b}, M, N)$} 
%By the same arguments as in Section \ref{sec:boundT2}, it is sufficient to consider when $(A_1\eta_1)^3 N^{\frac 12}q^{-\eps} \ll n \ll (A_1\eta_1)^3 N^{\frac 12}q^{\eps}$ and $(A_2\eta_1)^3 M^{\frac 12}q^{-\eps} \ll  m \ll (A_2\eta_1)^3 M^{\frac 12}q^{\eps},$ where $A_1 = \frac{\sqrt{a_1a_2M}}{c\delta q}$ and $A_2 = \frac{\sqrt{a_1a_2N}}{c\delta q}.$ (When $m, n$ are outside this range, it gives negligible contribution). 
We split into two cases as follows.
\\
{\bf Case 1: $c > \frac{8\pi\sqrt{a_1a_2MN}}{\delta q}$.} We use (\ref{asympJxSm}) and (\ref{lem:asymptUxbig}), and the integral that we consider is of the form
\est{%\pr{\frac{\sqrt{a_1a_2MN}}{c\delta q}}^{k-1}
\int_0^{\infty} \int_0^{\infty} G(y, z) \e{\frac{a_1^2 b_1y}{c\delta q a_2b_2} +\frac{a_2^2 b_2z}{c\delta q a_1b_1} \pm \frac{3 n^{\frac 13} y^{\frac 13}}{\eta_1} \pm \frac{3 m^{\frac 13} z^{\frac 13}}{\eta_2}} \> dy \> dz,}
where $\frac{\partial^j \partial^i G(y, z)}{\partial y^j \partial z^i} \ll \frac{1}{N^j M^i},$ $G(x,y) \ll 1,$ and it is supported in $[N, 2N] \times [M, 2M]$. Therefore, the integration over $y, z$ above is $O(MN).$

By the same arguments as case 1 of Section \ref{sec:proofE1b}, it is sufficient to consider when $c_1(B_1\eta_1)^3N^2 \ll n \ll c_2(B_1\eta_1)^3N^2 $ and $c_1(B_2\eta_2)^3M^2 \ll m \ll c_2(B_2\eta_2)^3M^2,$ where $c_1, c_2$ are some constants, $B_1 = \frac{a_1^2b_1}{c\delta qa_2b_2}$, and $B_2 = \frac{a_2^2b_2}{c\delta qa_1b_1}$, since the terms outside these ranges give negligible contribution from integration by parts many times. By the same arguments as in Section \ref{sec:proofE1}, 
$$ (B\eta_1)^3N^2 \ll \frac{(a_1^2b_1)^3N^2}{\delta^3q^3} \ll \frac{q^\eps}{\delta^3},  \ \ \ \ (B\eta_2)^3M^2 \ll \frac{(a_2^2b_2)^3M^2}{\delta^3q^3} \ll \frac{q^\eps}{\delta^3}.$$
So there are no terms of this form unless $N \gg \frac{q^{3/2}}{(a_1^2b_1)^{3/2}}, M \gg \frac{q^{3/2}}{(a_2^2b_2)^{3/2}},$ and $\delta \ll q^\eps.$  We then obtain that the contribution from these terms to $E^+_{5,4}({\bf a}, {\bf b}, M, N)$ is bounded by 

\est{ & M^{\frac 12}N^{\frac 12} q^\eps \sum_{\delta < q^\eps} \sum_{ c  \gg \frac{\sqrt{a_1a_2MN}}{\delta q}} \frac{1}{\eta_1^{\frac 32}\eta_2^{\frac 32}} \sum_{h_1 | \eta_1} \sqrt h_1  \sum_{\substack{n \ll q^\eps \\ h_1^2 | n}} \sum_{h_2 | \eta_1} \sqrt h_2  \sum_{\substack{m \ll q^\eps \\ h_2^2 | m}} 1 \ll a_1^{\frac 12}a_2^{\frac 12} q^{\frac 12 + \eps}. }%\sum_{q^{-\eps}\frac{\sqrt{a_1a_2MN}}{\delta q} \ll c  \ll \frac{\sqrt{a_1a_2MN}}{ q}} \frac{(a_2,u_1 c)}{c^{\frac 52}}  .}
%and so the contribution of these terms is $O\pr{q^{-\frac 12 + \eps}}.$
\\
\\
{\bf Case 2: $c \leq \frac{8\pi\sqrt{a_1a_2MN}}{\delta q}$.} For this case, we use (\ref{asympJxbig}), and the integral that we consider is of the form

\est{\frac{\eta_1 \eta_2}{m^{\frac 13}n^{\frac 13} M^{\frac 13}N^{\frac 13}} \frac{\sqrt{c\delta q}}{M^{\frac 14}N^{\frac 14}(a_1a_2)^{\frac 14}}&\int_0^{\infty} \int_0^{\infty} \mathcal H(y, z) \e{g(y,z, n, m)} \> dy \> dz, }
where $\frac{\partial^j \partial^i \mathcal H(y, z)}{\partial y^j \partial z^i} \ll \frac{1}{N^j M^i},$ $\mathcal H(x,y)$ is supported in $[N, 2N] \times [M, 2M]$, and
\est{
& g(y, z, n, m) =  \frac{a_1^2 b_1y}{c\delta q a_2b_2} +\frac{a_2^2 b_2z}{c\delta q a_1b_1} \pm \frac{3 n^{\frac 13} y^{\frac 13}}{\eta_1} \pm \frac{3 m^{\frac 13} z^{\frac 13}}{\eta_2} \pm \frac{2\sqrt{a_1a_2yz}}{c\delta q} .}
We note that the integration over $y, z$ above is $O(MN).$ Hence we obtain that
\est{\frac{\partial g(y, z, n, m)}{\partial y} = B_1 \pm \frac{ n^{\frac 13} }{y^{\frac 23}\eta_1}  \pm \frac{A_1}{ y^{\frac 12}},}
and 
\est{\frac{\partial g(y, z, n, m)}{\partial z} = B_2 \pm  \frac{ m^{\frac 13} }{z^{\frac 23}\eta_2 } \pm \frac{A_2}{ z^{\frac 12}},}
where $A_1 = \frac{\sqrt{a_1a_2z}}{c\delta q}$ and $A_2 = \frac{\sqrt{a_1a_2y}}{c\delta q}.$ We will divide into three cases to consider. 

{\it Case 2.1:}  $a_2^{\frac 32}b_2M^{\frac 12} \geq  4 a_1^{\frac 32}b_1N^{\frac 12}.$ 
For this case, we have that $\left|\frac {A_1}{y^{\frac 12}}  \pm B_1 \right| \asymp \frac {A_1}{y^{\frac 12}},$ and $\left| \frac {A_2}{z^{\frac 12}} \pm B_2 \right| \asymp B_2.$ By similar arguments to case 2 of section \ref{sec:proofE1b}, we consider the ranges $n \asymp (A_1\eta_1)^3N^{\frac 12}$ and $m \asymp (B_2 \eta_2)^3M^2.$ By the same arguments as in Section \ref{sec:proofE1}, we note that
	$$ (A\eta_1^3)N^{\frac 12} \ll \pr{\frac{\sqrt{a_1}}{\delta q^{\frac 12}}}^3 N^{\frac 12} \ll \frac{q^\eps}{\delta^3}, \ \ \ \ (B\eta_2)^3M^2 \ll \frac{(a_2^2b_2)^3M^2}{\delta^3q^3} \ll \frac{q^\eps}{\delta^3}$$
and there are no terms of this from unless $N \gg q^{\frac 32}/a_1^3,$  $M \gg \frac{q^{3/2}}{(a_2^2b_2)^{3/2}}$ and $\delta \ll q^{\eps}.$ Hence the contribution from these terms to $E^+_{5, 4}(\mathbf a, \mathbf b, M, N)$ is 
\est{ & M^{\frac 16}N^{\frac 1{6}} \sum_{\delta < q^\eps} \sum_{ c  \ll \frac{\sqrt{a_1a_2MN}}{\delta q}} \frac{1}{\eta_1^{\frac 12}\eta_2^{\frac 12}} \sum_{h_1 | \eta_1} \sqrt h_1  \sum_{\substack{n \ll q^\eps \\ h_1^2 | n}} \frac{1}{n^{\frac 13}}\sum_{h_2 | \eta_1} \sqrt h_2  \sum_{\substack{m \ll q^\eps \\ h_2^2 | m}} \frac{1}{m^{\frac 13}} \ll a_1^{\frac 12} a_2^{\frac 12}q^{\frac 12 + \eps}. }%\sum_{q^{-\eps}\frac{\sqrt{a_1a_2MN}}{\delta q} \ll c  \ll \frac{\sqrt{a_1a_2MN}}{ q}} \frac{(a_2,u_1 c)}{c^{\frac 52}}  .}
%Summing over $a_i, b_i, M, N$, we obtain that the contribution of these terms to $\mathcal E^+_{5, 4}$ is $ O\pr{q^{-\frac 12 + \eps}}.$

{\it Case 2.2:} $a_1^{\frac 32}b_1N^{\frac 12} \geq 4a_2^{\frac 32}b_2M^{\frac 12}.$  For this case, we do the same calculation as in case 2.1 and obtain that the contribution is also $O\pr{a_1^{
\frac 12} a_2^{\frac 12}q^{\frac 12 + \eps}}.$

{\it Case 2.3:}  $\frac 14 a_1^{\frac 32}b_1N^{\frac 12} <  a_2^{\frac 32}b_2M^{\frac 12} <  4 a_1^{\frac 32}b_1N^{\frac 12}.$ For this case, we have that $\frac {A_1}{y^{\frac 12}} \asymp B_1,$ and $\frac {A_2}{z^{\frac 12}} \asymp B_2.$ By similar arguments to case 2 of Section \ref{sec:proofE1b}, we can focus on the ranges $n \asymp (A_1\eta_1)^3N^{\frac 12}$ and $m \asymp (A_1\eta_1)^3N^{\frac 12}$.  The contribution from these terms to $ E^+_{5,4}(\mathbf a, \mathbf b, M, N)$ is then $\ll a_1^{\frac 12} a_2^{\frac 12}q^{\frac 12 + \eps}.$

%\est{ & M^{\frac 1{6}}N^{\frac 1{6}}\sum_{\delta < q^\eps} \sum_{ c  \ll \frac{\sqrt{a_1a_2MN}}{\delta q}} \frac{1}{\eta_1^{\frac 12}\eta_2^{\frac 12}} \sum_{h_1 | \eta_1} \sqrt h_1  \sum_{\substack{n \ll q^\eps \\ h_1^2 | n}} \frac{1}{n^{\frac 13}}\sum_{h_2 | \eta_1} \sqrt h_2  \sum_{\substack{m \ll q^\eps \\ h_2^2 | m}} \frac{1}{m^{\frac 13}} \ll  a_1^{\frac 12} a_2^{\frac 12}q^{\frac 12 + \eps}. }%\sum_{q^{-\eps}\frac{\sqrt{a_1a_2MN}}{\delta q} \ll c  \ll \frac{\sqrt{a_1a_2MN}}{ q}} \frac{(a_2,u_1 c)}{c^{\frac 52}}  .}

\section{Conclusion of the proof of Theorem \ref{thm:mainmoment}}
Recall that from (\ref{eqn:HM6}) and (\ref{eqn:MainboundLambda1}), we want to evaluate
	\est{H(0;\al,\bb) \mathcal M_6(q) = \mathscr M_1(q; \al, \bb) + \mathscr M_1(q; \bb, \al).}
By (\ref{decomM1}), we see that $\mathscr M_1(q; \al,\bb) = \mathscr D(q; \al, \bb) + \mathscr K(q; \al, \bb)$, and in Lemma \ref{lem:diagonal}, we showed that 
$$ \mathscr D(q; \al, \bb) = H(0; \al, \bb) \int_{-\infty}^{\infty} \mathcal M(q; \al + it, \bb + it) \> dt + O(q^{-3/4 + \eps}),$$
which is one of the twenty main terms of the asymptotic formula. Then we decomposed $\mathscr K(q; \al, \bb)$ as $\mathscr K_M(q; \al, \bb) + \mathscr K_E(q; \al, \bb)$. We proved in Section \ref{sec:truncationC} that $\mathscr K_E(q; \al, \bb) \ll q^{-1/2 + \eps},$ and then using Voronoi Summation formula, we extracted another nine main terms of the asymptotic formula from $\mathscr K_M(q; \al, \bb)$ with an error term $O(q^{-\frac 14 + \eps})$ (see Proposition \ref{prop:mainTM} and \ref{prop:Terror}, \S \ref{sec:provepropTM}, \S \ref{sec:properror} and Appendix \ref{sec:Eulerverif}). As briefly discussed in \S \ref{sec:setupmoment}, those terms correspond to $\mathcal M(q; \pi(\al) + it, \pi(\bb) + it), $ where $\pi$ is the transposition $ (\alpha_i, \beta_j)$ for $i = 1, 2, 3$ in $S_6/S_3 \times S_3.$ Hence $\mathscr M_1(q; \al, \bb)$ gives ten main terms the desired asymptotic formula, and similarly the remaining ten terms comes from $\mathscr M_1(q; \bb, \al).$

Therefore combining everything together, we have that
\est{H(0;\al,\bb) \mathcal M_6(q) = H(0; \al, \bb) \int_{-\infty}^{\infty} \sum_{\pi \in S_6/S_3 \times S_3} \mathcal M(q; \pi(\al) + it, \pi(\bb) + it) \> dt + O(q^{-1/4 + \eps}).}

If $|\alpha_i - \beta_j| \gg q^{-\eps}$ for all $1\leq i, j \leq 3$, then $H(0; \al, \bb) \gg q^{-\eps}$ and we immediately get 
\est{\mathcal M_6(q) = \int_{-\infty}^{\infty} \sum_{\pi \in S_6/S_3 \times S_3} \mathcal M(q; \pi(\al) + it, \pi(\bb) + it) \> dt + O(q^{-1/4 + \eps}).}
However, since all expressions above - including the term bounded by $O(q^{-1/2+\eps}$) - are analytic in the $\alpha_i$ and $\beta_j$, we see that this in fact holds in general.

\appendix
\section{Comparing  the main term of $ R_{\alpha_1, \beta_1}$ and $\mathcal M(q; \pi(\al), \pi(\bb))$} \label{sec:Eulerverif}

To finish the proof of Proposition \ref{prop:mainTM}, we will show that the local factor at prime $p$ of the Euler product of $\zeta(1 -\alpha_1 + \beta_1 )\mathcal M_{\alpha_1, \beta_1}(0)$ is the same as the one in $\mathcal A \mathcal Z\pr{\tfrac 12; \pi(\al), \pi(\bb)}$, where $(\pi(\al), \pi(\bb)) = (\beta_1, \alpha_1, \alpha_2; \alpha_1, \beta_2, \beta_3)$ and $\mathcal M_{\alpha_1, \beta_1}(s)$ is defined as in (\ref{def:Malbeta}).  To simplify the presentation, we will work within the ring of formal Dirichlet series, so that we need not worry about convergence issues in this section.  Indeed, if we show that $\zeta(1 -\alpha_1 + \beta_1 )\mathcal M_{\alpha_1, \beta_1}(0)$ is the same as  $\mathcal A \mathcal Z\pr{\tfrac 12; \pi(\al), \pi(\bb)}$ as formal series, then they must have the same region of absolute convergence.  Thus, as analytic functions, they agree on the region of absolute convergence, and so must be the same by analytic continuation.  Note that we have already verified that there is a non-empty open region of absolute convergence at the end of \S \ref{sec:provepropTM}.

For notational convenience,  $\al_{2, 3} = (\alpha_2, \alpha_3),$ and $ -\bb_{2,3} = (-\beta_2, -\beta_3)$ in this section. %We compared it with the sum from (\ref{diagAlt}) when $\alpha_1$ is replaced by $\beta_1$ and $\beta_1$ is replaced by $\alpha_1$. 

\subsection{Euler product at prime $p$ of $\mathcal A \mathcal Z\pr{\tfrac 12; \pi(\al), \pi(\bb)}$} \label{sec:Eulerproduct} We start from rearranging the sums in $\mathcal A \mathcal Z(s; \pi(\al), \pi(\bb))$ by the same method as in (\ref{prod3L}).  When $\textrm{Re}(s + \beta_1 + \alpha_2 + \alpha_3),  \textrm{Re}(s - \alpha_1 - \beta_2 - \beta_3) > 1,$ we recall that from Equations (\ref{def:Cs}) and (\ref{def:mathcalAs}), $\mathcal A \mathcal Z(s; \pi(\al), \pi(\bb))$ is
\est{ %\label{EulerConj}
	  \sumsix_{\substack{a_1, b_1, a_2, b_2, m, n \geq 1\\ a_1n = a_2m \\ a_1b_1 = a_2b_2 \\ (a_i, q) = (b_j, q) = 1 }}  \frac{ \mathscr B(a_1, b_1; \pi(\al))}{(a_1b_1)^{2s}} \frac{ \mathscr B(a_2, b_2; -\pi(\bb))}{(a_2b_2)^{2s}} \frac{\sigma_3(n;\pi(\al))\sigma_3(m; -\pi(\bb))}{(a_1n)^{ s} (a_2m)^{ s}} .}
%and we take $s = 0.$ We will rearrange the sum in both (\ref{EulerOff}) and (\ref{EulerConj}) so that it is easier to compare the Euler product from both of them. 

%We rearrange (\ref{EulerConj}) by the same method as in  (\ref{prod3L}). 
Using Lemma \ref{lem:multofsigma_k} and the proof of Lemma \ref{prod3L} and using the fact that
\es{\label{eqn:sig3to2}\sigma_3(a; \alpha_1, \alpha_2, \alpha_3) = \sum_{df = a} d^{-\alpha_1} \sigma_2(f; \alpha_2, \alpha_3),}
we see after a change of variables that 

\es{\label{Euler2Conj} \mathcal A \mathcal Z\pr{\tfrac 12; \pi(\al), \pi(\bb)} &= \zeta( 1- \alpha_1 + \beta_1) \sumfour_{\substack{d_1, d_2, e_1, e_2 \geq 1\\ d_1e_1 = d_2 e_2 \\ (d_ie_i, q) = 1}} \frac{1}{d_1^{1  + \alpha_2 + \alpha_3}d_2^{1  - \beta_2 -\beta_3}} \frac{1}{e_1^{1  + \beta_1}e_2^{1  - \alpha_1}} \mathcal J\pr{ e_1, e_2},}
%\begin{align*}
%&\sumfour_{\substack{d_1, d_2, e_1, e_2 \\ d_1e_1 = d_2 e_2 \\ (d_ie_i, q) = 1}} \frac{1}{d_1^{ 2s + \alpha_1 + \alpha_2}d_2^{2s - \beta_2 -\beta_3}} \frac{1}{e_1^{2s + \beta_1}e_2^{2s - \alpha_1}} \mathcal J(s; e_1, e_2)
%\sum_{k} \frac{1}{k^{2s - \alpha_1 + \beta_1},}
%\end{align*}
where

\es{ \label{defCalJ}
	\mathcal J( e_1, e_2) &=  \sumtwo_{\substack{ j_1, j_2 \geq 1}}   \frac{\sigma_2(j_1e_1; \altt)\sigma_2(j_2e_2; -\bbtt) (j_1, j_2)^{1  - \alpha_1 + \beta_1}}{j_1^{1  - \alpha_1}j_2^{1  + \beta_1}}.}

Since both $\zeta(1 -\alpha_1 + \beta_1 )\mathcal M_{\alpha_1, \beta_1}(0)$ and $\mathcal A \mathcal Z\pr{\tfrac 12; \pi(\al), \pi(\bb)}$ have the factor $\zeta(1 - \alpha_1 + \beta_1)$, it suffices to consider only the local factor at prime $p$ of  the sum over $d_i, e_i$ in (\ref{Euler2Conj}).  For $p\neq q$, this is %For notational convience,  $p^{r}$ means the highest prime power dividing $r$, where $r = e_i, d_i, j_i, f_i, u_i, \ell_i, n, m, \gamma, c, b.$ %(variables are power of p now.) 

%For fixed $d_1, e_1, d_2, e_2$ such that $d_1 + e_1 = d_2 + e_2$, the factor is
\es{\label{localpRHS}\sumfour_{\substack{\delta_1, \delta_2, \epsilon_1, \epsilon_2 \geq 0 \\ \delta_1 + \epsilon_1 = \delta_2 + \epsilon_2}}\frac{1}{p^{D + \epsilon_1 + \epsilon_2 + \epsilon_1 \beta_1 - \epsilon_2\alpha_1 }} \sum_{k \geq 0} \frac{\mathcal J_p(\epsilon_1, \epsilon_2, k)}{p^k},}
where $p^{\delta_i} \| d_i$, $p^{\epsilon_i} \| e_i$, $p^{\iota_i} \| j_i,$
\es{ \label{def:D} D := \delta_1 +  \delta_2 + \delta_1(\alpha_2 + \alpha_3) - \delta_2(\beta_2 - \beta_3),} 
and
\es{\label{defJpk} \mathcal J_p(\epsilon_1, \epsilon_2, k) &:=    \sigma_2(p^{k + \epsilon_1}; \altt) \sigma_2 (p^{k + \epsilon_2}; -\bbtt)  + \sum_{0 \leq \iota_1 < k } \frac{\sigma_2(p^{\iota_1 + \epsilon_1}; \altt) \sigma_2 (p^{k + \epsilon_2}; -\bbtt)}{p^{\beta_1(k-\iota_1)}}  \\
	& \hskip 2.5in + \sum_{0 \leq \iota_2 < k } \frac{\sigma_2(p^{k + \epsilon_1}; \altt) \sigma_2 (p^{\iota_2 + \epsilon_2}; -\bbtt)}{p^{ - \alpha_1(k-\iota_2)}}.}
For $p = q$, we have that $\delta_i = \epsilon_i = 0,$ and the local factor at $p$ is 
\es{\label{localqRHS} \sum_{k \geq 0} \frac{\mathcal J_p(0,0, k)}{p^k}.}

We also comment here that when $\alpha_i = \beta_i = 0$ for $i = 1, 2, 3,$ using $\sigma_2(p^k) = k + 1$ in (\ref{localpRHS}), (\ref{defJpk}), (\ref{localqRHS}) and some straightforward calculation,  we derive that the local factor at $p \neq q$ of $\mathcal A \mathcal Z \pr{\tfrac 12; 0 , 0}$ is
$$\left( 1 - \frac 1p\right)^{-9} C_p,$$
where $C_p$ is defined in (\ref{def:c3}), and the local factor at $q$ is
$$ \left( 1 - \frac 1q\right)^{-5}\left( 1 + \frac 4q + \frac 1{q^2}\right). $$
This explains the presence of the arithmetic factor in our Conjecture \ref{conj:CFKRSnoshift}, as in the work \cite{CFKRS}.

%To match with the local factor of $\mathcal M_{\alpha_1, \beta_1}$, we consider terms with fixed  $k, \delta_1, \epsilon_1, \delta_2, \epsilon_2$ such that $\delta_1 + \epsilon_1 = \delta_2 + \epsilon_2$, which is 
%\est{\frac{1}{p^{D + \epsilon_1 + \epsilon_2 + \epsilon_1 \beta_1 - \epsilon_2\alpha_1  + k}} \mathcal J_p(\epsilon_1, \epsilon_2, k).}

\subsection{The Euler product at $p$ of $\mathcal M_{\alpha_1, \beta_1}(0)$} First, by the definition of $\mathscr B(a, b; \alpha_1, \alpha_2, \alpha_3)$ in (\ref{eqn:defB}), $\mathcal G(c, {\bf a}, {\bf b})$ in (\ref{def:Gc}), $\sumsharp \mathscr G_{ {\bf a}, {\bf b}}(s; h, b, c, \gamma, g)$ in (\ref{def:Gscr}), Equation (\ref{eqn:sig3to2}), and a change of variables, we obtain that $\mathcal M_{\alpha_1, \beta_1}(0)$ can be re-written as

\est{ %\label{EulerOff2}
& \sumfour_{\substack{d_1, f_1, d_2, f_2 \geq 1\\ (d_if_i, q) = 1}}  \frac{1}{d_1^{1 - \alpha_1 - \beta_1 + \alpha_2 + \alpha_3} d_2^{1 + \alpha_1 + \beta_1 - \beta_2 - \beta_3}} \frac{1}{f_1^{1 - \beta_1} f_2^{1 + \alpha_1}} \\
	& \hskip 0.5in \cdot \sum_{\substack{h \geq 1 \\ h | u_1u_2}} \frac{\mu(h)}{h^{\alpha_1 - \beta_1}} \sum_{\substack{b \geq 1\\ (b, u_1u_2) = 1}} \frac{\mu(b)}{b^{ \alpha_1 - \beta_1 }} \sum_{c \geq 1}  \frac{c}{c^{ \alpha_1 - \beta_1}}  \sum_{\gamma | c}  \frac{1}{\gamma^{\alpha_1 - \beta_1}} \sum_{g | \frac{c}{\gamma}} \frac{\mu(g)}{g^{\alpha_1 - \beta_1}} \prod_{\substack{p| c \\ p \nmid bh\gamma}} \pr{1 - \frac 1p}  \\
	& \hskip 0.5in  \cdot \sum_{n \geq 1}  \sum_{a_1 | f_1} \frac{\mu(a_1)}{a_1^{\alpha_2 + \alpha_3}} \pr{\frac{(a_1, u_2cb) }{a_1u_2cbn}}^{1 - \alpha_1} \sigma_2 \pr{\frac{u_2cbn}{(a_1, u_2cb)} ; \altt} \sigma_2\pr{\frac{f_1}{a_1}; \altt} \\
	& \hskip 0.5in \cdot \sum_{m \geq 1}\sum_{a_2 | f_2} \frac{\mu(a_2)}{a_2^{- \beta_2 - \beta_3}} \pr{\frac{ (a_2, u_1cb)}{a_2u_1cbm}}^{1 + \beta_1}  \sigma_2 \pr{\frac{u_1cbm}{(a_2, u_1cb)} ; -\bbtt}\sigma_2\pr{\frac{f_2}{a_2}; -\bbtt}. }
In Section \ref{sec:calRes}, $u_i = \frac{a_ib_i}{(a_1b_1, a_2b_2)},$ but after changing variables, we write that $u_i  = \frac{f_id_i}{(f_1d_1, f_2d_2)}.$ By comparing Euler products, we can show that % We will consider the sum over $n$ and $a_1$. We want to show that it is equal to 
\est{ %\label{eqn:rearrOff}
	\sum_{n \geq 1} \sum_{a_1 | f_1} \frac{\mu(a_1)}{a_1^{\alpha_2 + \alpha_3}} \pr{\frac{(a_1, u_2cb) }{a_1u_2cbn}}^{1 - \alpha_1} \sigma_2 \pr{\frac{u_2cbn}{(a_1, u_2cb)} ; \altt} \sigma_2\pr{\frac{f_1}{a_1}; \altt} = \sum_{n' \geq 1} \frac{\sigma_2(f_1u_2cbn'; \altt)}{(u_2cbn')^{1-\alpha_1}}.}
We also have a similar expression for the sum over $m$ and $a_2.$  Hence we can write $\mathcal M_{\alpha_1, \beta_1}(0)$ as
\es{\label{EulerOff3}& %\sumfour_{\substack{d_1, f_1, d_2, f_2 \\ (d_if_i, q) = 1}}  \frac{1}{d_1^{1 - \alpha_1 - \beta_1 + \alpha_2 + \alpha_3} d_2^{1 + \alpha_1 + \beta_1 - \beta_2 - \beta_3}} \frac{1}{f_1^{1 - \beta_1} f_2^{1 + \alpha_1}} \\
	%& \hskip 0.5in \cdot \sum_{h | u_1u_2} \frac{\mu(h)}{h^{\alpha_1 - \beta_1}} \sum_{\substack{b \\ (b, u_1u_2) = 1}} \frac{\mu(b)}{b^{ \alpha_1 - \beta_1 }} \sum_{c}  \frac{c}{c^{ \alpha_1 - \beta_1}}  \sum_{\gamma | c}  \frac{1}{\gamma^{\alpha_1 - \beta_1}} \sum_{g | \frac{c}{\gamma}} \frac{\mu(g)}{g^{\alpha_1 - \beta_1}} \prod_{\substack{p| c \\ p \nmid bh\gamma}} \pr{1 - \frac 1p}  \\
	%& \hskip 0.5in  \cdot \sum_{n, m}  \frac{\sigma_2(f_1u_2cbn; \altt)}{(u_2cbn)^{1-\alpha_1}} \frac{\sigma_2(f_2u_1cbm; -\bbtt)}{(u_2cbm)^{1+\beta_1}} \\
	\sumfour_{\substack{d_1, f_1, d_2, f_2  \geq 1\\ (d_if_i, q) = 1}}  \frac{1}{d_1^{1 - \alpha_1 - \beta_1 + \alpha_2 + \alpha_3} d_2^{1 + \alpha_1 + \beta_1 - \beta_2 - \beta_3}} \frac{1}{f_1^{1 - \beta_1} u_2^{1 - \alpha_1} f_2^{1 + \alpha_1} u_1^{1 + \beta_1}} \\
	& \hskip 0.5in \cdot \sum_{\substack{h \geq 1 \\ h | u_1u_2}} \frac{\mu(h)}{h^{\alpha_1 - \beta_1}} \sum_{\substack{b \\ (b, u_1u_2) = 1}} \frac{\mu(b)}{b^{  2  }} \sum_{c}  \frac{1}{c}  \sum_{\gamma | c}  \frac{1}{\gamma^{\alpha_1 - \beta_1}} \sum_{g | \frac{c}{\gamma}} \frac{\mu(g)}{g^{\alpha_1 - \beta_1}} \prod_{\substack{p| c \\ p \nmid bh\gamma}} \pr{1 - \frac 1p}  \\
	& \hskip 0.5in  \cdot \sumtwo_{n, m \geq 1}  \frac{\sigma_2(f_1u_2cbn; \altt)}{n^{1-\alpha_1}} \frac{\sigma_2(f_2u_1cbm; -\bbtt)}{m^{1+\beta_1}}. \\}
We note here that $d_1f_1u_2 = d_2f_2u_1$ by the definition of $u_1, u_2.$ Next, we consider the local factor at $p \neq q$ of (\ref{EulerOff3}), which is of the form
\es{\label{eqn:initialfactorp}\sumfour_{\substack{\delta_1, \delta_2, \xi_1, \xi_2 \geq 0 \\ \delta_1 + \ell_1 = \delta_2 + \ell_2}}\frac{1}{p^{D'(\ell_1, \ell_2) + \ell_2\beta_1 - \ell_1 \alpha_1 - (\xi_1 + \xi_2)(\beta_1 - \alpha_1) }} \mathscr L_p(\delta_1, \delta_2, \xi_1, \xi_2),}
where $p^{\delta_i} \| d_i$, $p^{\xi_i} \| f_i$, $p^{\upsilon_i} \| u_i$, $\ell_1 = \xi_1 + \upsilon_2 $,  $\ell_2 = \xi_2 + \upsilon_1$, $\min\{ \upsilon_1, \upsilon_2\} = 0,$  $D'(\ell_1, \ell_2) = D + (\delta_2 - \delta_1)(\alpha_1 + \beta_1) + \ell_1 + \ell_2$, and $D$ is defined in (\ref{def:D}). We will examine $\mathscr L_p(d_1, d_2, f_1, f_2)$ below but before that analysis, we need the following two Lemmas.

\begin{lem} \label{lem:localfactorgammac} The contribution to the local factor at $p$ from
$$ \sum_{\gamma | c}  \frac{1}{\gamma^{\alpha_1 - \beta_1}} \sum_{g | \frac{c}{\gamma}} \frac{\mu(g)}{g^{\alpha_1 - \beta_1}} \prod_{\substack{p| c \\ p \nmid bh\gamma}} \pr{1 - \frac 1p} 
$$is 1 if $p \nmid c$ or $p | bh$. Otherwise, it is 
	$	1 - \frac{1}{p} + \frac{1}{p^{1+\alpha_1 - \beta_1}}.$
	
%	FAI: THIS HAS BEEN CHANGED:ORIGINALLY IT WAS $	1 - \frac{1}{p} + \frac{1}{p^{\alpha_1 - \beta_1}}.$
\end{lem}
\begin{proof}
	For $p \nmid c$ or $p | bh$, the contribution to the local factor is 1 because
	$$\sum_{\gamma | c}  \frac{1}{\gamma^{\alpha_1 - \beta_1}} \sum_{g | \frac c\gamma} \frac{\mu(g)}{g^{\alpha_1 - \beta_1}} = \sum_{a | c} \frac{1}{a^{\alpha_1 - \beta_1}  }\sum_{g | a} \mu(g) = 1.$$
	Now suppose $p | c$ and $ p \nmid bh$.  Below we write $p^{c_p} \| c$ and $p^{\gamma_p} \| \gamma.$  Then the contribution to the local factor at $p$ is
	\est{&\pr{ 1 - \frac{1}{p^{\alpha_1 - \beta_1}}}\pr{1 - \frac{1}{p}} + \sum_{1 \leq \gamma_p  < c_p} \frac{1}{p^{\gamma_p(\alpha_1 - \beta_1)}} \pr{ 1 - \frac{1}{p^{\alpha_1 - \beta_1}}} + \frac{1}{p^{c_p(\alpha_1 - \beta_1)}} = 1 - \frac{1}{p} + \frac{1}{p^{1+\alpha_1 - \beta_1}}.}
\end{proof}

\begin{lem} \label{arrPl1l2} Let 
	\es{\label{def:Pcnm} \mathscr P (\ell_1, \ell_2) := \sumthree_{c_p \geq 0, \  n_p, m_p \geq 0} \frac{1}{p^{c_p}}\frac{\sigma_2(p^{\ell_1 + n_p + c_p}; \altt)}{p^{n_p(1-\alpha_1)}} \frac{\sigma_2(p^{\ell_2 + m_p + c}; -\bbtt)}{p^{m_p(1+\beta_1)}}.}
	Then
	$$ \mathscr P (\ell_1, \ell_2) = \sum_{k \geq 0} \frac{\mathcal J_p(\ell_1, \ell_2, k)}{p^k} + \frac{p^{\alpha_1 - \beta_1}}{p^2} \mathscr P (\ell_1 + 1, \ell_2 + 1),$$
	where $\mathcal J_p(\ell_1, \ell_2, k)$ is defined as in (\ref{defJpk}).
	
\end{lem}
\begin{proof} We have
	\est{ \mathscr P (\ell_1, \ell_2) &= \sum_{c_p \geq 0} \frac{\sigma_2(p^{\ell_1  + c_p}; \altt) \sigma_2(p^{\ell_2 + c_p}; -\bbtt)}{p^{c_p}} +  \sumtwo_{c_p \geq 0, \  n_p \geq 1} \frac{\sigma_2(p^{\ell_1 + n_p + c_p}; \altt)\sigma_2(p^{\ell_2  + c_p}; -\bbtt)}{p^{c_p + n_p- n_p\alpha_1}} \\
		& + \sumtwo_{c_p \geq 0, \  m_p \geq 1} \frac{\sigma_2(p^{\ell_1  + c_p}; \altt)\sigma_2(p^{\ell_2 + m_p + c_p}; -\bbtt)}{p^{c_p + m_p + m_p\beta_1}}  + \frac{p^{\alpha_1 - \beta_1}}{p^2} \mathscr P (\ell_1 + 1, \ell_2 + 1) \\
		%&= \sum_{k \geq 0} \frac{\sigma_2(p^{\ell_1  + k}; \altt) \sigma_2(p^{\ell_2 + k}; -\bbtt)}{p^k} +  \sum_{k \geq 0 } \sum_{ 0 \leq \iota_2 < k} \frac{\sigma_2(p^{\ell_1 + k}; \altt)\sigma_2(p^{\ell_2  + \iota_2}; -\bbtt)}{p^{k- (k - \iota_2)\alpha_1}} \\
		%& + \sum_{k \geq 0 } \sum_{ 0 \leq \iota_1 < k}\frac{\sigma_2(p^{\ell_1  + \iota_1}; \altt)\sigma_2(p^{\ell_2 + k}; -\bbtt)}{p^{k + (k - \iota_1)\beta_1}}  + \frac{p^{\alpha_1 - \beta_1}}{p^2} \mathscr P (\ell_1 + 1, \ell_2 + 1) \\
		&= \sum_{k \geq 0} \frac{\mathcal J_p(\ell_1, \ell_2, k)}{p^k} + \frac{p^{\alpha_1 - \beta_1}}{p^2} \mathscr P (\ell_1 + 1, \ell_2 + 1),}
	after some arrangement.
\end{proof}

Now we examine $\mathscr L_p(\delta_1, \delta_2, \xi_1, \xi_2)$ which we separate into two cases below. 

\subsection*{Case 1: $\delta_1 + \xi_1 = \delta_2 + \xi_2$.}
For this case, we have $\upsilon_1 = \upsilon_2 = 0,$ so $\xi_i = \ell_i$. Hence $u_1u_2 = 1$ and $p \not| h.$ From Lemma \ref{lem:localfactorgammac}, we then obtain that
\est{\mathscr L_p(\delta_1, \delta_2, \xi_1, \xi_2) &= \mathscr P (\ell_1, \ell_2) + \frac{1}{p^2} \pr{ p^{\beta_1 - \alpha_1} - 1}\mathscr P (\ell_1 + 1, \ell_2 + 1) - \frac{1}{p^2}\mathscr P (\ell_1 + 1, \ell_2 + 1).  }
%where 
%\es{\label{def:Lpc0b0} \mathscr L( c = 0 , b = 0) :=  \sumtwo_{n, m \geq 0} \frac{\sigma_2(p^{\ell_1 + n}; \altt)}{p^{n(1-\alpha_1)}} \frac{\sigma_2(p^{\ell_2 + m}; -\bbtt)}{p^{m(1+\beta_1)}},}
%\es{\label{def:Lpc1b0} \mathscr L( c \geq 1 , b = 0) :=  \sumthree_{c \geq 1, \  n, m \geq 0} \frac{1}{p^c}\frac{\sigma_2(p^{\ell_1 + n + c}; \altt)}{p^{n(1-\alpha_1)}} \frac{\sigma_2(p^{\ell_2 + m + c}; -\bbtt)}{p^{m(1+\beta_1)}} \pr{1 - \frac{1}{p} + \frac{1}{p^{1 + \alpha_1 - \beta_1}}},}
%and 
%\es{\label{def:Lpb1} \mathscr L( b = 1) :=   \sumthree_{c \geq 0, \  n, m \geq 0} \frac{1}{p^c}\frac{\sigma_2(p^{\ell_1 + n + c + 1}; \altt)}{p^{n(1-\alpha_1)}} \frac{\sigma_2(p^{\ell_2 + m + c + 1}; -\bbtt)}{p^{m(1+\beta_1)}}.}

%Therefore, $\mathscr L(b = 1) = \mathscr P (\ell_1 + 1, \ell_2 + 1).$ We can also write $\mathscr L( c = 0 , b = 0) + \mathscr L (c \geq 1, b = 0)$ as
%\est{ \mathscr P (\ell_1, \ell_2) + \frac{1}{p^2} \pr{ p^{\beta_1 - \alpha_1} - 1}\mathscr P (\ell_1 + 1, \ell_2 + 1).}
%\begin{rem*}
%The second term comes from changing $c$ to $c + 1$ for factor $-\frac 1p + \frac {1}{p^{1 + \alpha_1 - \beta_1}}.$
%\end{rem*}

%\begin{rem*}
%The last term is when $n, m \geq 1$. We change variables from $n$ to $n + 1$ %and $m$ to $m + 1$
%\end{rem*}

From Lemma \ref{arrPl1l2}, $\mathscr L_p(d_1, d_2, f_1, f_2) $ can be written as
\est{
	& \sum_{k \geq 0} \frac{\mathcal J_p(\ell_1, \ell_2, k)}{p^k} + \frac{1}{p^2}\pr{p^{\alpha_1 - \beta_1} - 2 + p^{\beta_1 - \alpha_1}} \mathscr P (\ell_1 + 1, \ell_2 + 1)\\
	&= \sum_{k \geq 0} \frac{\mathcal J_p(\ell_1, \ell_2, k)}{p^k} + \frac{1}{p^2}\pr{p^{\alpha_1 - \beta_1} - 2 + p^{\beta_1 - \alpha_1}}\sum_{k \geq 0} \frac{\mathcal J_p(\ell_1 + 1, \ell_2 + 1, k)}{p^k}  \\
	& \hskip 2in +  \frac{p^{\alpha_1 - \beta_1}}{p^4}\pr{p^{\alpha_1 - \beta_1} - 2 + p^{\beta_1 - \alpha_1}}\mathscr P (\ell_1 + 2, \ell_2 + 2) \\
	&= \sum_{k \geq 0} \frac{1}{p^k} \pg{\mathcal J_p(\ell_1, \ell_2, k) + \pr{p^{\alpha_1 - \beta_1} - 2 + p^{\beta_1 - \alpha_1}} \sum_{1 \leq m_p \leq \lfloor \frac{k}{2}\rfloor}p^{(m_p - 1)(\alpha_1 - \beta_1)} \mathcal J_p(\ell_1 + m_p, \ell_2 + m_p, k - 2m_p)}.}

\subsection*{Case 2: $\delta_1 + \xi_1 \neq \delta_2 + \xi_2$} For this case, $\upsilon_1 + \upsilon_2 \geq 1$. So $p|u_1u_2$, and $b_p = 0$, where $p^{b_p} \| b.$  By Lemma \ref{lem:localfactorgammac} and \ref{arrPl1l2}, we have
%\est{\mathscr L_p(\delta_1, \delta_2, \xi_1, \xi_2) = \mathscr L(c = 0, h = 0) + \mathscr L(c = 1, h = 0 ) - \frac{1}{p^{\alpha_1 - \beta_1}}\mathscr L (h = 1),}
%where 
%\es{\label{def:Lpc0h0} \mathscr L( c = 0 , h = 0) :=  \sumtwo_{n, m \geq 0} \frac{\sigma_2(p^{\ell_1 + n}; \altt)}{p^{n(1-\alpha_1)}} \frac{\sigma_2(p^{\ell_2 + m}; -\bbtt)}{p^{m(1+\beta_1)}},}
%\es{\label{def:Lpc1h0} \mathscr L( c \geq 1 , h = 0) :=  \sumthree_{c \geq 1, \  n, m \geq 0} \frac{1}{p^c}\frac{\sigma_2(p^{\ell_1 + n + c}; \altt)}{p^{n(1-\alpha_1)}} \frac{\sigma_2(p^{\ell_2 + m + c}; -\bbtt)}{p^{m(1+\beta_1)}} \pr{1 - \frac{1}{p} + \frac{1}{p^{1 + \alpha_1 - \beta_1}}},}
%and 
%\es{\label{def:Lph1} \mathscr L( h = 1) :=   \sumthree_{c \geq 0, \  n, m \geq 0} \frac{1}{p^c}\frac{\sigma_2(p^{\ell_1 + n + c }; \altt)}{p^{n(1-\alpha_1)}} \frac{\sigma_2(p^{\ell_2 + m + c }; -\bbtt)}{p^{m(1+\beta_1)}} = \mathscr P(\ell_1, \ell_2).}
%We rearrange these terms (and changing variable c to c + 1 for the second term), and we use (\ref{arrPl1l2}). Then we can write them as
\est{&\mathscr L_p(\delta_1, \delta_2, \xi_1, \xi_2) = \pr{1 - \frac{1}{p^{\alpha_1 -\beta_1}}} \mathscr P(\ell_1, \ell_2) - \frac{1}{p^2}\pr{1 - \frac{1}{p^{\alpha_1 -\beta_1}}} \mathscr P(\ell_1 + 1, \ell_2+ 1)\\
	&=  \pr{1 - \frac{1}{p^{\alpha_1 -\beta_1}}} \pg{\sum_{k \geq 0} \frac{\mathcal J_p(\ell_1, \ell_2, k)}{p^k} + \frac{1}{p^2}\pr{p^{\alpha_1 - \beta_1} - 1} \mathscr P (\ell_1 + 1, \ell_2 + 1)}\\
	&= \pr{1 - \frac{1}{p^{\alpha_1 -\beta_1}}}  \left\{ \sum_{k \geq 0} \frac{\mathcal J_p(\ell_1, \ell_2, k)}{p^k} + \frac{1}{p^2}\pr{p^{\alpha_1 - \beta_1} - 1}\sum_{k \geq 0} \frac{\mathcal J_p(\ell_1 + 1, \ell_2 + 1, k)}{p^k} \right. \\
	& \hskip 3in +  \left.\frac{p^{\alpha_1 - \beta_1}}{p^4}\pr{p^{\alpha_1 - \beta_1} - 1}\mathscr P (\ell_1 + 2, \ell_2 + 2) \right\} \\
	&= \pr{1 - \frac{1}{p^{\alpha_1 -\beta_1}}} \sum_{k \geq 0} \frac{1}{p^k} \Bigg\{\mathcal J_p(\ell_1, \ell_2, k)  \\
	& \hskip 1.5in  + \pr{p^{\alpha_1 - \beta_1} - 1} \sum_{1 \leq m_p \leq \lfloor \frac{k}{2}\rfloor}p^{(m_p - 1)(\alpha_1 - \beta_1)} \mathcal J_p(\ell_1 + m_p, \ell_2 + m_p, k - 2m_p) \Bigg\}.}

From both cases, we obtain that (\ref{eqn:initialfactorp}) is 
\es{\label{eqn:secondfactor}&\sumfour_{\substack{\delta_1, \delta_2, \ell_1, \ell_2 \geq 0 \\ \delta_1 + \ell_1 = \delta_2 + \ell_2  }}\frac{1}{p^{D'(\ell_1, \ell_2) - \ell_1\beta_1 + \ell_2\alpha_1 }} \mathscr L_p(\delta_1, \delta_2, \ell_1, \ell_2) \\
	&+ \sumfour_{\substack{\delta_1, \delta_2, \ell_1, \ell_2 \geq 0 \\ \delta_1 + \ell_1 = \delta_2 + \ell_2 }} \pr{ \sum_{\substack{0 \leq \xi_2 < \ell_2 \\ \xi_1 = \ell_1}} \frac{\mathscr L_p(\delta_1, \delta_2, \xi_1, \xi_2)}{p^{(\ell_2- \ell_1)\beta_1 - \xi_2(\beta_1 - \alpha_1) }} + \sum_{\substack{0 \leq \xi_1 < \ell_1 \\ \xi_2 = \ell_2}} \frac{\mathscr L_p(\delta_1, \delta_2, \xi_1, \xi_2)}{p^{(\ell_2- \ell_1)\alpha_1 - \xi_1(\beta_1 - \alpha_1) }}} \frac{1}{p^{D'(\ell_1, \ell_2)}} \\
	&:= \sumtwo_{\delta_1, \delta_2 \geq 0} \mathcal S_p(\delta_1, \delta_2), }
say. 
%\est{\mathcal S_p(\delta_1, \delta_2) &= \sumtwo_{\substack{ \ell_1, \ell_2 \\ \delta_1 + \ell_1 = \delta_2 + \ell_2  }} \left(\frac{1}{p^{D'(\ell_1, \ell_2) } } \frac{1}{p^{  - \ell_1 \beta_1 + \ell_2\alpha_1  }}  \sum_{k' \geq 0} \frac{1}{p^{k'}} \Bigg\{\mathcal J_p(\ell_1, \ell_2, k') \right. \\
%	& \hskip 0.3in + \left. \pr{p^{\alpha_1 - \beta_1} - 2 + p^{\beta_1 - \alpha_1}} \sum_{1 \leq m_p \leq \lfloor \frac{k'}{2}\rfloor}p^{(m_p - 1)(\alpha_1 - \beta_1)} \mathcal J_p(\ell_1 + m_p, \ell_2 + m_p, k' - 2m_p) \right\} \\
%	& \hskip 0.3in +\frac{1}{p^{D'(\ell_1, \ell_2) } }   \pr{\frac{1}{p^{(\ell_2 - \ell_1)\alpha_1}} + \frac{1}{p^{(\ell_2 - \ell_1)\beta_1}} - \frac{2}{p^{-\ell_1\beta_1 + \ell_2\alpha_1}}} \sum_{k' \geq 0} \frac{1}{p^{k'}} \Bigg\{\mathcal J_p(\ell_1, \ell_2, k')  \\
	%&  \hskip 0.3in \left. \left. + \pr{p^{\alpha_1 - \beta_1} - 1} \sum_{1 \leq m_p \leq \lfloor \frac{k'}{2}\rfloor}p^{(m_p - 1)(\alpha_1 - \beta_1)} \mathcal J_p(\ell_1 + m_p, \ell_2 + m_p, k' - 2m_p) \right\} \right). \\
%}
For fixed $\delta_1, \delta_2$, where $\delta_2 \geq \delta_1$, we rearrange the term $\mathcal S_p(\delta_1, \delta_2)$ and  obtain that

%For fixed $d_1, d_2, \ell_1, \ell_2$ we collect the terms for $f_1, f_2$ that give $\ell_1, \ell_2.$  Let $D'(\ell_1, \ell_2) = D + (d_2 - d_1)(\alpha_1 + \beta_1) + \ell_1 + \ell_2.$
%To compare with the other side term, for fixed $d_1, d_2$ (without loss of generality, we let $d_2 \geq d_1.$), we consider the terms with $\mathcal J_p(e_1, e_2, k).$  These terms are from terms with $ (\ell_1, \ell_2, k', m) =(e_1, e_2, k, 0), (e_1 - 1, e_2 - 1, k + 2, 1), ..., (e_1 - e_2, 0, k + 2 e_2, e_2).$ Combining the terms with  $\mathcal J_p(e_1, e_2, k),$ we have
\es{\label{eqn:thirdfactor} S_p(\delta_1, \delta_2) &= \sumtwo_{\substack{ \epsilon_1, \epsilon_2 \geq 0 \\ \delta_1 + \epsilon_1 = \delta_2 + \epsilon_2  }} \sum_{k \geq 0} \frac{\mathcal J_p(\epsilon_1, \epsilon_2, k)}{p^{D'(\epsilon_1, \epsilon_2) + k } } \Bigg\{ \frac{1}{p^{(\epsilon_2 - \epsilon_1)\alpha_1}} + \frac{1}{p^{(\epsilon_2 - \epsilon_1)\beta_1}} - \frac{1}{p^{-\epsilon_1\beta_1 + \epsilon_2\alpha_1}}  \\
	& \hskip 0.5 in +  \pr{p^{\alpha_1 - \beta_1} - 2 + p^{\beta_1 - \alpha_1}}\sum_{1 \leq \ell_2 \leq \epsilon_2  } \frac{p^{(\ell_2 - 1)(\alpha_1 - \beta_1)}}{p^{-(\epsilon_1 - \ell_2)\beta_1 + (\epsilon_2 - \ell_2)\alpha_1 }} \\
	&  \hskip 0.5in + \left. \pr{p^{\alpha_1 - \beta_1} - 1}  \sum_{1 \leq \ell_2 \leq \epsilon_2}  p^{(\ell_2 - 1)(\alpha_1 - \beta_1)} \pr{\frac{1}{p^{(\epsilon_2 - \epsilon_1)\alpha_1}} + \frac{1}{p^{(\epsilon_2 - \epsilon_1)\beta_1}} - \frac{2 p^{\ell_2( \alpha_1 - \beta_1)}}{p^{-\epsilon_1\beta_1 + \epsilon_2\alpha_1}}}  \right\} \\
	%&= \frac{J_p(\ell_1, \ell_2, k)}{p^{D'(e_1, e_2) + k } } \left\{ \frac{1}{p^{(e_2 - e_1)\alpha_1}} + \frac{1}{p^{(e_2 - e_1)\beta_1}} - \frac{1}{p^{-e_1\beta_1 + e_2\alpha_1}} \right. \\
	%& \ \ +  \pr{p^{\alpha_1 - \beta_1} - 2 + p^{\beta_1 - \alpha_1}} p^{e_1\beta_1 - e_2\alpha_1} \frac{p^{(2e_2 + 1)(\alpha_1 - \beta_1)} - p^{\alpha_1 - \beta_1}}{p^{2(\alpha_1 - \beta_1)} - 1} \\
	%& \ \  + \left. \pr{p^{e_2(\alpha_1 - \beta_1)} - 1 }\pr{\frac{1}{p^{(e_2 - e_1)\alpha_1}} + \frac{1}{p^{(e_2 - e_1)\beta_1}} }   - 2 \pr{p^{\alpha_1 - \beta_1} - 1}  p^{e_1\beta_1 - e_2\alpha_1} \frac{p^{(2e_2 + 1)(\alpha_1 - \beta_1)} - p^{\alpha_1 - \beta_1}}{p^{2(\alpha_1 - \beta_1)} - 1} \right\} \\
	%&= \frac{J_p(\ell_1, \ell_2, k)}{p^{D'(e_1, e_2) + k } } \left\{ p^{e_1\alpha_1 - e_2 \beta_1}  + p^{e_2\alpha_1 - 2e_2\beta_1 + e_1\beta_1} - p^{e_1\beta_1 - e_2\alpha_1} \right. \\
	%& \hskip 1in -  \left. \pr{p^{\alpha_1 - \beta_1} -  \frac{1}{p^{\alpha_1 - \beta_1}}} p^{e_1\beta_1 - e_2\alpha_1} \frac{p^{(2e_2 + 1)(\alpha_1 - \beta_1)} - p^{\alpha_1 - \beta_1}}{p^{2(\alpha_1 - \beta_1)} - 1}   \right\} \\
	%&=  \sumtwo_{\substack{ \epsilon_1, \epsilon_2 \\ \delta_1 + \epsilon_1 = \delta_2 + \epsilon_2  }} \sum_{k \geq 0}  \frac{\mathcal J_p(\epsilon_1, \epsilon_2, k)}{p^{D + \epsilon_1 + \epsilon_2 + k + (d_2 - d_1)(\alpha_1 + \beta_1) } }  p^{\epsilon_1\alpha_1 - \epsilon_2 \beta_1}   
	&=  \sumtwo_{\substack{ \epsilon_1, \epsilon_2 \geq 0 \\ \delta_1 + \epsilon_1 = \delta_2 + \epsilon_2  }} \sum_{k \geq 0} \frac{\mathcal J_p(\epsilon_1, \epsilon_2, k)}{p^{D + \epsilon_1 + \epsilon_2 + k + \epsilon_1 \beta_1 - \epsilon_2 \alpha_1 } }. }
%since $e_1 + d_1 = e_2 + d_2.$ 
By similar calculation $\mathcal S_p(\delta_1, \delta_2)$ yields the same value for $\delta_1 > \delta_2.$ In summary from (\ref{eqn:initialfactorp}), (\ref{eqn:secondfactor}) and (\ref{eqn:thirdfactor}), the local factor at $p$ of $\mathcal M_{\alpha_1, \beta_1}$ is 
\est{\sumtwo_{\substack{ \epsilon_1, \epsilon_2 \geq 0\\ \delta_1 + \epsilon_1 = \delta_2 + \epsilon_2  }} \sum_{k \geq 0} \frac{\mathcal J_p(\epsilon_1, \epsilon_2, k)}{p^{D + \epsilon_1 + \epsilon_2 + k + \epsilon_1 \beta_1 - \epsilon_2 \alpha_1 } }, }
which is the same as the local factor at $p$ of $\mathcal A \mathcal Z \pr{\tfrac 12; \pi(\al), \pi(\bb)}$ in (\ref{localpRHS}), as desired. 

For $p = q$, we use similar arguments, with $\delta_i = \epsilon_i = 0$, so that the Euler factor is 
$$\sum_{k \geq 0} \frac{\mathcal J_p(0, 0, k)}{p^{k } }, $$
which is the same as the local factor at $q$ of $\mathcal A \mathcal Z \pr{\tfrac 12; \pi(\al), \pi(\bb)}$ in (\ref{localqRHS}). This completes the proof of Proposition \ref{prop:mainTM}.

%\section{Derivation of Conjecture \ref{conj:CFKRSnoshift}}\label{app:CFKRSconj}

\section{Voronoi Summation}\label{sec:voronoi}
In this section, we state the Voronoi Summation formula for the shifted $k$-divisor function defined in (\ref{def:sigma_k}).  The proof of this formula is essentially the same as the proof by Ivic of the Voronoi Summation formula for the $k$-divisor function in \cite{Ivic}, so we will state the results and refer the reader to Ivic for detailed proofs.

Let $\omega$ be a smooth compactly supported function.  For $\al = (\alpha_1,...,\alpha_k)$, let $\sigma_k(n ; \al) = \sigma_k(n; \alpha_1, ..., \alpha_k).$
Define
\est{
S\pr{\frac ac, \al} = \sum_{n= 1}^\infty \sigma_k(n; \al) e\bfrac{an}{c} \omega(n),
}
where $(a, c) = 1$, and the Mellin transform of $\omega$
\est{
\tilde \omega(s) = \int_0^\infty \omega(x) x^{s-1} dx.  
}Since $\omega$ was chosen to be from the Schwarz class, $\tilde \omega$ is entire and decays rapidly on vertical lines.  We have the Mellin inversion formula
\est{
\omega(n) = \frac{1}{2\pi i} \int_{(c)} \tilde \omega(s)  \frac{ds}{n^s},
}where $c$ is any vertical line. Let 
\es{\label{def:Ak} A_k&\pr{m, \frac ac, \al} \\
 &= \frac 12 \sum_{\substack{m_1,..., m_k \geq 1 \\ m_1...m_k = m}} m_1^{\alpha_1}...m_k^{\alpha_k} \sum_{a_1 \mod c} ... \sum_{a_k \mod c} \pg{\e{\frac{aa_1...a_k + \ab \cdot \mathbf m}{c}} + \e{\frac{-aa_1...a_k + \ab \cdot \mathbf m}{c}} },}
and

\es{\label{def:Bk} B_k&\pr{m, \frac ac, \al} \\
 &= \frac 12 \sum_{\substack{m_1,..., m_k \geq 1 \\ m_1...m_k = m}} m_1^{\alpha_1}...m_k^{\alpha_k} \sum_{a_1 \mod c} ... \sum_{a_k \mod c} \pg{\e{\frac{aa_1...a_k + \ab \cdot \mathbf m}{c}} - \e{\frac{-aa_1...a_k + \ab \cdot \mathbf m}{c}} }.}
Moreover, we define
\est{G_k(s, n, \al) = \frac{c^{ks}}{\pi^{ks}n^s}\pr{\prod_{\ell = 1}^k \frac{\Gamma\pr{\frac{s-\alpha_\ell}{2}}}{\Gamma\pr{\frac{1 - s+\alpha_\ell}{2}}}}, \ \ \ \ \ H_k(s, n, \al) = \frac{c^{ks}}{\pi^{ks}n^s}\pr{\prod_{\ell = 1}^k \frac{\Gamma\pr{\frac{1+s -\alpha_\ell}{2}}}{\Gamma\pr{\frac{2 - s+\alpha_\ell}{2}}}},  }
and for $0 < \sigma < \frac 12 - \frac 1k + \frac{\sum_{\ell = 1}^{k} \textrm{Re}(\alpha_\ell)}{k},$ we let
\es{\label{def:UandV} U_k(x; \al) = \frac{1}{2\pi i} \int_{(\sigma)} \prod_{\ell = 1}^k \frac{\Gamma\pr{\frac{s-\alpha_\ell}{2}}}{\Gamma\pr{\frac{1 - s+\alpha_\ell}{2}}} \frac{ds}{x^s}, \ \ \ \ \ V_k(x; \al) = \frac{1}{2\pi i} \int_{(\sigma)} \prod_{\ell = 1}^k \frac{\Gamma\pr{\frac{1+ s-\alpha_\ell}{2}}}{\Gamma\pr{\frac{2 - s+\alpha_\ell}{2}}} \frac{ds}{x^s}.}

We note that by Stirling's formula, both integrals for $U_k$ and $V_k$ are absolutely convergent.

Finally, we define the Dirichlet series to be
\est{
	D_k \pr{s, \frac ac, \al} = \sum_n \frac{\sigma_k(n; \al) e\bfrac{an}{c}}{n^s},
}
which converges absolutely for $\tRe s > 1$.  %For $k=2$ and $\al = 0$, this becomes the well-known Estermann D-function.  
We have that
\es{\label{eqn:OriginalD} 
	D_k \pr{s, \frac ac, \al}
	&= \sum_{n_1,...,n_k} \frac{e\bfrac{an_1n_2...n_k}{c}}{n_1^{s+\alpha_1}...n_k^{s+\alpha_k}} \\
	&= \frac{1}{c^{ks+\alpha_1+...+\alpha_k}}\sum_{a_1,...,a_k \bmod c}e\bfrac{aa_1...a_k}{c} \prod_{j=1}^k \zeta\left(s+\alpha_j, \frac{a_j}{c}\right),
}
where $ \zeta \left(s, \frac ac\right)$ is the Hurwitz zeta function defined for $\tRe s > 1$ as
\es{ \label{def:hurwitz}
	\zeta \left(s, \frac ac\right) = \sum_{n=0}^\infty \frac{1}{\lr{n+\frac ac}^s}.
}

The Hurwitz zeta function may be analytically continued to all of $\mathbb C$ except for a simple pole at $s = 1$. Therefore, $D_k\left(s, \tfrac ac \right)$ can be analytically continued to all of $\mathbb C$ except for a simple pole at $s = 1 - \alpha_j$ for $j = 1,.., k.$

\begin{thm} \label{thm:voronoi} With notations as above and $(a, c) = 1$ and $\alpha_i \ll \frac{1}{1000k\log (|c| + 100)}, $  we have
\est{
S\pr{\frac ac, \al} &= \sum_{\ell = 1}^{k} \Res_{s = 1 - \alpha_\ell} \tilde \omega(s) D_k\pr{s, \frac ac, \al} \\
& + \frac{\pi^{k/2 + \alpha_1 + ... + \alpha_k}}{c^{k+\alpha_1 + ... + \alpha_k}}\sum_{n = 1}^{\infty} A_k\pr{n, \frac ac, \al} \int_0^{\infty} \omega(x) U_k\pr{\frac{\pi^knx}{c^k}; \al} \> dx\\
& + i^{3k}\frac{\pi^{k/2 + \alpha_1 + ... + \alpha_k}}{c^{k+\alpha_1 + ... + \alpha_k}}\sum_{n = 1}^{\infty} B_k\pr{n, \frac ac, \al} \int_0^{\infty} \omega(x) V_k\pr{\frac{\pi^knx}{c^k}; \al} \> dx.}
\end{thm}
We refer the reader to the proof of Theorem 2 in \cite{Ivic} for details.  Next, we collect properties of $A_3(n, a/c, \al)$, $B_3(n, a/c, \al)$, $U_3(x; \al)$ and $V_3(x; \al)$. These are useful for bounding error terms of $\mathcal M_6(q).$

\begin{lemma} \label{lem:A3B3Kloos}
	Let $a, n, \gamma$ be integers such that $(a, \gamma) = 1$. Moreover $A_3(n, a/c, \al)$, $B_3(n, a/c, \al)$ are defined as in (\ref{def:Ak}) and (\ref{def:Bk}). Then
	\est{\sum_{n_1n_2n_3 = n}& n_1^{\alpha_1}n_2^{\alpha_2}n_3^{\alpha_3} \sumthree_{r_1, r_2, r_3 \mod \gamma} \e{\frac{ar_1r_2r_3 + r_1n_1 + r_2n_2 + r_3n_3}{\gamma}} \\
		&= \gamma \sum_{h | \gamma, h^2 | n} h \Delta(n, h, \gamma)  S \pr{\frac{n}{h^2}, -\overline{a}, \frac{\gamma}{h}},}
	where $\Delta(n, h, \gamma)$ is a divisor function satisfying $\Delta(n, h, \gamma) \ll (\gamma n)^\eps.$
	Moreover, 
	\est{A_3\pr{n, \frac{a}{\gamma}, \al} \ll (\gamma n)^\eps\gamma^{\frac 32} \sum_{h | \gamma, h^2 | n} \sqrt h,  \ \ \ \ \ \ \ B_3\pr{n, \frac{a}{\gamma}, \al} \ll (\gamma n)^\eps\gamma^{\frac 32} \sum_{h | \gamma, h^2 | n} \sqrt h.} 
\end{lemma}

The proof of this lemma can be found in Equations (8.7)-(8.9) in \cite{Ivic}.

\begin{lemma} \label{lem:asymptUandV}
If $U(x; \al) := U_3(x; \al)$ and $V(x; \al) := V_3(x; \al),$ as defined in (\ref{def:UandV}), and $\alpha_i \ll \frac{1}{1000k\log (|c| + 100)}, $
then for any $0 < \eps < 1/6$ and $x > 0,$ we have
\es{\label{lem:asympUVxsmall}
U(x; \al) \ll x^{\eps}, \hskip 1in V(x; \al) \ll x^{\eps}.}
Moreover for any fixed integer $K \geq 1$ and $x \geq x_0 > 0,$
\es{\label{lem:asymptUxbig}
U(x; \al) = \sum_{j = 1}^K \frac{c_j \cos(6x^{\frac 13}) + d_j \sin (6x^{\frac 13 })}{x^{\frac j3 + \frac{\alpha_1 + \alpha_2 + \alpha_3}{3}}} + O\pr{\frac 1{x^{\frac {K+1}{3} + \frac{\alpha_1 + \alpha_2 + \alpha_3}{3}}}},}
 \es{\label{lem:asymptVxbig}
V(x; \al) = \sum_{j = 1}^K \frac{e_j \cos(6x^{\frac 13}) + f_j \sin (6x^{\frac 13})}{x^{\frac j3 + \frac{\alpha_1 + \alpha_2 + \alpha_3}{3} } } + O\pr{\frac 1{x^{\frac {K+1}{3} + \frac{\alpha_1 + \alpha_2 + \alpha_3}{3}}}},}
with suitable constants $c_j,..., f_j$, and $c_1 = 0,  d_1 = -\frac{2}{\sqrt {3\pi}}$ $e_1 =-\frac{2}{\sqrt {3\pi}}, f_1 = 0 $.  
\end{lemma}
The proof of this lemma is a minor modification of the proof of Lemma 3 in \cite{Ivic}.

\section*{Acknowledgment} Vorrapan Chandee acknowledges support from Coordinating Center for Thai Government Science
and Technology Scholarship Students (CSTS) and
National Science and Technology Development Agency (NSTDA) grant of year 2014.
Part of this work was done while the first author was visiting Mathematical Institute, University of Oxford - she is grateful for their kind hospitality.

\end{document}